\documentclass[12pt]{amsart}
\usepackage{txfonts}
\usepackage{amscd,amssymb,mathabx,subfigure}
\usepackage[all]{xy}
\usepackage{graphicx,calrsfs}
\usepackage{imakeidx}
\usepackage[colorlinks,plainpages,backref,urlcolor=blue,hyperindex=true]{hyperref}
\usepackage{tikz}
\usetikzlibrary{calc}
\usepackage{mdwlist}
\usetikzlibrary{graphs, positioning}

\newtheorem{theorem}{Theorem}[section]
\newtheorem{corollary}[theorem]{Corollary}
\newtheorem{lemma}[theorem]{Lemma}
\newtheorem{prop}[theorem]{Proposition}

\theoremstyle{definition}

\newtheorem{example}[theorem]{Example}
\newtheorem{remark}[theorem]{Remark}
\newtheorem{conjecture}[theorem]{Conjecture}

\newtheorem{problem}[theorem]{Problem}
\newtheorem*{ack}{Acknowledgments}

\newcommand{\N}{\mathbb{N}}
\newcommand{\Z}{\mathbb{Z}}

\newcommand{\C}{\mathbb{C}}
\newcommand{\FF}{\mathbb{F}}

\newcommand{\PP}{\mathbb{P}}
\newcommand{\CP}{\mathbb{CP}}
\renewcommand{\k}{\Bbbk}

\DeclareMathAlphabet{\pazocal}{OMS}{zplm}{m}{n}
\newcommand{\A}{{\pazocal{A}}}
\newcommand{\B}{{\pazocal{B}}}
\newcommand{\XX}{{\pazocal X}}
\newcommand{\NN}{{\pazocal{N}}}
\newcommand{\HH}{{\pazocal H}}

\newcommand{\RR}{{\mathcal R}}
\newcommand{\VV}{{\mathcal V}}

\newcommand{\F}{{\mathcal{F}}}
\newcommand{\cP}{{\mathcal{P}}}
\newcommand{\cE}{{\mathcal{E}}}
\newcommand{\M}{{\mathcal{M}}}

\newcommand{\E}{{\rm Ess}}
\newcommand{\pA}{{\bar{\A}}}

\newcommand{\g}{{\mathfrak{g}}}
\newcommand{\h}{{\mathfrak{h}}}

\renewcommand{\sl}{{\mathfrak{sl}}}

\newcommand{\e}{{e}}
\newcommand{\ii}{\mathrm{i}}

\DeclareMathOperator{\rank}{rank}

\DeclareMathOperator{\im}{im}

\DeclareMathOperator{\id}{id}

\DeclareMathOperator{\Sym}{Sym}
\DeclareMathOperator{\ch}{char}
\DeclareMathOperator{\GL}{GL}
\DeclareMathOperator{\SL}{SL}

\DeclareMathOperator{\Hom}{{Hom}}
\DeclareMathOperator{\spn}{span}

\DeclareMathOperator{\proj}{pr}
\DeclareMathOperator{\ev}{ev}
\DeclareMathOperator{\Ev}{Ev}

\DeclareMathOperator{\Aut}{Aut}

\DeclareMathOperator{\Lie}{Lie}
\DeclareMathOperator{\supp}{supp}

\DeclareMathOperator{\reg}{reg}
\DeclareMathOperator{\mult}{mult}
\DeclareMathOperator{\depth}{depth}
\DeclareMathOperator{\bLie}{\overline{\Lie}}
\DeclareMathOperator{\Net}{Net}
\DeclareMathOperator{\AG}{AG}

\newcommand{\surj}{\twoheadrightarrow}
\newcommand{\inj}{\hookrightarrow}

\newcommand{\isom}{\xrightarrow{\,\simeq\,}}
\newcommand{\abs}[1]{\left| #1 \right|}
\def\set#1{{\{ #1\}}}

\def\dot{\mathchar"013A}  
\newcommand{\hdot}{{\raise1pt\hbox to0.35em{\Huge $\dot$}}} 

\newenvironment{romenum}
{ 

\begin{enumerate}}{\end{enumerate}}

\newcommand{\cdga}{\ensuremath{\mathsf{cdga}}}

\definecolor{dkgreen}{RGB}{0,100,0}
\definecolor{dkbrown}{RGB}{139,69,19}

\begin{document}

\title[Milnor fibrations, modular resonance, and algebraic monodromy]{%
The Milnor fibration of a hyperplane arrangement: from modular resonance
to algebraic monodromy}

\author[Stefan Papadima]{Stefan Papadima$^1$}
\address{Simion Stoilow Institute of Mathematics, 
P.O. Box 1-764,
RO-014700 Bucharest, Romania}
\email{\href{mailto:Stefan.Papadima@imar.ro}{Stefan.Papadima@imar.ro}}
\thanks{$^1$Partially supported by the Romanian Ministry of 
National Education, CNCS-UEFISCDI, grant PNII-ID-PCE-2012-4-0156}

\author[Alexander~I.~Suciu]{Alexander~I.~Suciu$^2$}
\address{Department of Mathematics,
Northeastern University,
Boston, MA 02115, USA}
\email{\href{mailto:a.suciu@neu.edu}{a.suciu@neu.edu}}
\urladdr{\href{http://www.northeastern.edu/suciu/}%
{www.northeastern.edu/suciu/}}
\thanks{$^2$Partially supported by 
NSF grant DMS--1010298, NSA grant H98230-13-1-0225, 
and Simons Foundation collaboration grant 354156.}



\begin{abstract}
A central question in arrangement theory is to determine whether the 
characteristic polynomial $\Delta_q$ of the algebraic monodromy acting 
on the homology group $H_q(F(\A),\C)$ of the Milnor fiber of a complex 
hyperplane arrangement $\A$ is determined by the intersection 
lattice $L(\A)$.  Under simple combinatorial conditions, we show 
that the multiplicities of the factors of $\Delta_1$ corresponding 
to certain eigenvalues of order a power of a prime $p$ are equal 
to the Aomoto--Betti numbers $\beta_p(\A)$, which in turn are 
extracted from $L(\A)$.  When $\A$ defines an arrangement of 
projective lines with only double and triple points, this leads 
to a combinatorial formula for the algebraic monodromy. 
To obtain these results, we relate nets on the underlying matroid of $\A$
to resonance varieties in positive characteristic. Using modular invariants
of nets, we find a new realizability obstruction (over $\C$) for matroids, and we 
estimate the number of essential components in the first 
complex resonance variety of $\A$. Our approach also reveals a 
rather unexpected connection of modular resonance
with the geometry of $\SL_2(\C)$-representation varieties, 
which are governed by the Maurer--Cartan equation.
\end{abstract}

\maketitle
\setcounter{tocdepth}{1}
\tableofcontents

\section{Introduction and statement of results}
\label{sect:intro}

\subsection{The Milnor fibration}
\label{subsec:mf}

In his seminal book on complex hypersurface singularities, 
Milnor \cite{Mi} introduced a fibration that soon became the 
central object of study in the field, and now bears his name. 
In its simplest manifestation, Milnor's construction associates 
to each homogeneous polynomial $Q\in \C[z_1,\dots, z_{\ell}]$ 
a smooth fibration over $\C^*$, by restricting the polynomial map  
$Q\colon \C^{\ell} \to \C$ to the complement of the zero-set of $Q$. 

The Milnor fiber of the polynomial, $F=Q^{-1}(1)$, is an affine manifold, 
and thus has the homotopy type of a  finite CW-complex of 
dimension $\ell-1$.  The monodromy of the fibration 
is the map $h\colon F\to F$, $z\mapsto e^{2\pi \ii/n} z$, 
where $n=\deg Q$.  The induced homomorphisms 
in homology, $h_q\colon H_q(F,\C)\to H_q(F,\C)$, are all 
diagonalizable, with $n$-th roots of unity as eigenvalues.  

A key question, then, is to compute the characteristic polynomials 
of these operators in terms of available data.  We will only address here 
the case $q=1$,  which is already far from solved if the 
polynomial $Q$ has a non-isolated singularity at $0$. 

\subsection{Hyperplane arrangements}
\label{subsec:hyp}
Arguably the simplest situation is when the polynomial $Q$ completely 
factors into distinct linear forms.  This situation is neatly described 
by a hyperplane arrangement, that is, a finite collection $\A$ of 
codi\-mension-$1$ linear subspaces in  $\C^{\ell}$. Choosing 
a linear form $f_H$ with kernel $H$ for each hyperplane $H\in \A$, 
we obtain a homogeneous polynomial, $Q=\prod_{H\in \A} f_H$, 
which in turn defines the Milnor fibration of the arrangement. 

To analyze this fibration, we turn to the rich combinatorial structure 
encoded in the intersection lattice of the arrangement, $L(\A)$, 
that is, the poset of all intersections of hyperplanes in $\A$ 
(also known as flats), ordered by reverse inclusion, and ranked 
by codimension.  We then have the following, much studied problem, 
which was raised in \cite[Problem 9A]{HR} and \cite[Problem 4.145]{Kb}, 
and still remains open.

\begin{problem}
\label{prob:mfa}
Given a hyperplane arrangement $\A$, is the characteristic 
polynomial of the algebraic monodromy of the Milnor fibration, 
$\Delta_{\A}(t)=\det(tI-h_1)$, determined by the intersection lattice 
$L(\A)$?  If so, give an explicit combinatorial formula to compute it. 
\end{problem}

Without essential loss of generality, we may assume that the 
ambient dimension is $\ell=3$, in which case the projectivization 
$\bar\A$ is an arrangement of lines in $\CP^2$.  In 
Theorem \ref{thm:main0}, we give a positive
answer to Problem \ref{prob:mfa} in the case 
when those lines intersect only in double or triple points. 
Despite its apparent simplicity, this case already poses 
quite a challenge.  It was previously attacked in a number of papers, 
including \cite{Di, DIM, Li}, but only partial answers 
were obtained as a result.  Our approach, though, provides 
a complete answer in this setting.

As the multiplicities of those intersection points increase, 
we still get some answers, albeit not complete ones. 
For instance, in Theorem \ref{thm:main1} we identify 
in combinatorial terms the number of times the cyclotomic 
factor $\Phi_3(t)$ appears in $\Delta_{\A}(t)$, under the 
assumption that $\bar\A$ has no intersection points of 
multiplicity $3r$, with $r>1$, while in Theorem \ref{thm:2main1} 
we treat the analogous problem for the cyclotomic factors 
$\Phi_2(t)$ and $\Phi_4(t)$.

\subsection{Combinatorics and the algebraic monodromy}
\label{subsec:char poly}
In order to describe our results in more detail, we need to introduce 
some notation. Let $M(\A)$ be the complement of the arrangement, 
and let $Q\colon M(\A) \to \C^*$ be the Milnor fibration, with Milnor 
fiber $F(\A)$. Set $\e_d(\A)=0$ if $d\not\mid n$.  The polynomial 
\begin{equation}
\label{eq:delta-arr}
\Delta_{\A}(t) = (t-1)^{n-1}  \cdot \prod_{1<d\mid n} \Phi_d(t)^{\e_d(\A)}
\end{equation}
encodes the structure of the vector space $H_1(F(\A),\C)$, 
viewed as a module over the group algebra $\C[\Z_n]$ via the action of 
the monodromy operator $h_1$.  More precisely, 
\begin{equation}
\label{eq:h1f}
H_1(F(\A),\C) = (\C[t]/(t-1))^{n-1} \oplus \bigoplus_{1<d\mid n} 
(\C[t]/\Phi_d(t))^{\e_d(\A)}.
\end{equation}
Therefore, Problem \ref{prob:mfa} amounts to deciding whether 
the integers $\e_d(\A)$ are combinatorially determined, and, if so, 
computing them explicitly.

Let $L_s(\A)$ be the set of codimension $s$ flats in $L(\A)$. For 
each such flat $X$, let $\A_X$ be the subarrangement consisting of 
all hyperplanes that contain $X$.  Finally, let $\mult(\A)$ be the set 
of integers  $q\ge 3$ for which there is a flat $X\in L_2(\A)$ such that 
$X$ has multiplicity $q$, i.e., $\abs{\A_X}=q$.  Not all divisors of $n$ 
appear in the above formulas.  Indeed, as shown by Libgober in \cite{Li82, Li02}, 
if $d$ does not divide one of the integers comprising $\mult (\A)$, the exponent 
$\e_d(\A)$ vanishes.  In particular, if $\mult(\A) \subseteq \{3\}$, then only 
$e_3(\A)$ may be non-zero.  Our first main result computes this integer 
under this assumption.

\begin{theorem}
\label{thm:main0}
Suppose $L_2(\A)$ has only flats of multiplicity $2$ and $3$.  Then 
the characteristic polynomial of the algebraic monodromy of the Milnor fibration 
is given by 
\begin{equation}
\label{eq:delta23}
\Delta_{\A}(t)=(t-1)^{\abs{\A}-1}\cdot (t^2+t+1)^{\beta_3(\A)},
\end{equation} 
where $\beta_3(\A)$ is an integer between $0$ and $2$ that
depends only on $L_{\le 2}(\A)$. 
\end{theorem}

As we shall explain below, the combinatorial invariant $\beta_3(\A)$ is 
constructed from the mod $3$ cohomology ring of $M(\A)$. 
Our theorem implies at once a result of 
Libgober (\cite[Theorem 1.1]{Li}), which states that, under the same assumption 
on multiplicities, the question whether $\Delta_{\A}(t)$ equals $(t-1)^{\abs{\A}-1}$ 
or not can be decided combinatorially.  Our result is much stronger, in 
that it gives a completely combinatorial formula for the polynomial 
$\Delta_{\A}(t)$, by showing that the cyclotomic factor $\Phi_3(t)$ 
appears with multiplicity $e_3(\A)=\beta_3(\A)$, and also by showing that 
this multiplicity can be at most $2$. 

\subsection{Resonance varieties and multinets}
\label{subsec:intro-res}
Fix a commutative Noetherian ring $\k$.  A celebrated theorem 
of Orlik and Solomon \cite{OS} asserts that the cohomology ring 
$H^*(M(\A),\k)$ is isomorphic to the OS-algebra 
of the underlying matroid, $A^*(\A)\otimes \k$, 
and thus is determined by the (ranked) intersection poset $L(\A)$.  
Key to our approach are the resonance varieties of $\A$, which 
keep track in a subtle way of the vanishing cup products in this ring.  
For our purposes here, we will only be interested in resonance 
in degree $1$.  

For an element $\tau\in A^1(\A)\otimes \k= \k^{\A}$, left-multiplication 
by $\tau$ in the cohomology ring gives rise to a $\k$-cochain complex, 
$(A^*(\A)\otimes \k, \tau \cdot)$. The (first) resonance variety of $\A$ 
over $\k$, denoted $\RR_1(\A,\k)$, is the locus of those
elements $\tau$ for which the homology in degree $1$ 
of this complex  is non-zero.  When $\k$ is a field, this set 
is a homogeneous subvariety of the affine space $\k^{\A}$.
When $\k=\C$, all the irreducible components of 
$\RR_1(\A,\C)$ are linear subspaces, intersecting 
transversely at $0$, see \cite{CS99, LY}.  In positive characteristic, 
the components of $\RR_1(\A,\k)$ may be non-linear, or may have 
non-transverse intersection, see \cite{Fa07}. 
Very useful to us will be a result of Falk and Yuzvinsky \cite{FY} 
and Marco Buzun\'ariz \cite{MB}, 
which describes all components of $\RR_1(\A,\C)$ in terms of 
multinets on the arrangement $\A$ and its subarrangements. 

A {\em $k$-multinet}\/ on $\A$ is a partition of the arrangement into 
$k\ge 3$ subsets $\A_{\alpha}$, together with an assignment 
of multiplicities $m_H$ to each $H\in \A$, and a choice of rank $2$ 
flats, called the base locus.  All these data must satisfy certain 
compatibility conditions. For instance, any two hyperplanes from 
different parts of the partition intersect in the base locus, while the sum 
of the multiplicities over each part  is constant.  Furthermore, if 
$X$ is a flat in the base locus,  then the sum 
$n_{X}=\sum_{H\in\A_\alpha\cap \A_X} m_H$ is independent 
of $\alpha$.

A multinet as above is {\em reduced}\/ if all the multiplicities 
$m_H$ are equal to $1$.  If, moreover, all the numbers 
$n_X$ are equal to $1$, the multinet is, in fact, a 
{\em net}---a classical notion from combinatorial geometry.

Hyperplane arrangements may be viewed as simple matroids 
realizable over the field of complex numbers. (From now on, 
when we speak about {\em realizable}\/ matroids, we mean 
realizability over $\C$.)  For an arbitrary simple matroid $\M$, one may speak
about (reduced) multinets and nets, as well as resonance varieties $\RR_1(\M,\k)$
with arbitrary coefficients.  Let $B_{\k}(\M) \subseteq \k^{\M}$ be the 
constant functions, and let $\sigma \in B_{\k}(\M)$ be the function 
taking the constant value $1$. The {\em cocycle space}\/ 
$Z_{\k}(\M) \subseteq \k^{\M}$ is defined by the linear condition 
$\sigma\cdot\tau=0$.  Plainly, $\sigma\in \RR_1(\M,\k)$ if and 
only if $Z_{\k}(\M) \ne B_{\k}(\M)$. When $\k$ is a field, the 
{\em Aomoto--Betti number}\/ of the matroid is defined as 
\begin{equation}
\label{eq:betak}
\beta_{\k}(\M)= \dim_{\k} Z_{\k}(\M) / B_{\k}(\M).
\end{equation}
Clearly, this integer depends only on $p:=\ch (\k)$, and so will often be 
denoted simply by $\beta_{p}(\M)$.

Having multinets in mind, let us consider a finite set $\k$ with $k\ge 3$ elements, and
define $B_{\k}(\M) \subseteq \k^{\M}$ as before. The subset of `special' 
$\k$-cocycles, $Z'_{\k}(\M) \subseteq \k^{\M}$, consists of those functions $\tau$ with the
property that their restriction to an arbitrary flat from $L_2(\M)$ is either constant or bijective.
Given a partition, $\M =\coprod_{\alpha \in \k} \M_{\alpha}$, we associate to it the element
$\tau \in \k^{\M}$ defined by $\tau_u=\alpha$, for $u\in \M_{\alpha}$.

Our starting point is the following result, which relates nets to modular resonance, 
and which will be proved in \S\ref{subsec:special}. 

\begin{theorem}
\label{teo=lambdaintro}
Let $\M$ be a simple matroid, and let $k\ge 3$ be an integer. Then:
\begin{romenum}
\item \label{li1}
For any $k$-element set $\k$, the above construction induces a bijection,
\[
\xymatrix{\lambda_{\k} \colon \{\text{$k$-nets on $\M$}\}  \ar^(.52){\simeq}[r]&
Z'_{\k}(\M)\setminus B_{\k}(\M)} .
\]
\item \label{li2}
If $k \not\equiv 2\bmod 4$, there is a commutative ring $\k$ of cardinality 
$k$ such that $Z'_{\k}(\M) \subseteq Z_{\k}(\M)$. 
If, in fact, $k =p^s$, for some prime $p$, then $\k$ can be chosen to be the  
Galois field $\k=\FF_{p^s}$.
\end{romenum}
\end{theorem}

\subsection{Matroid realizability and essential components}
\label{subsec:intr-mat}

It is well-known that non-trivial $k$-nets on simple matroids exist, for all $k\ge 3$.
For realizable matroids, the picture looks completely different: 
by a result of Yuzvinsky \cite{Yu09}, non-trivial $k$-nets exist only 
for $k=3$ or $4$; many examples of $3$-nets appear naturally, while 
the only known $4$-net comes from the famous Hessian configuration \cite{Yu04}. 

The difference between realizable and non-realizable matroids 
comes to the fore in \S\ref{sec:matr}.  Starting from the 
affine geometries $\AG (m, \k)$, where $\k=\FF_{p^s}$ is a Galois field 
with at least $3$ elements, we construct by a process 
of rank $3$ truncation a family of matroids with ground set $\k^m$, 
which we denote by $\M_\k(m)$.  
We then show in Proposition \ref{prop:mpm} that $\beta_p(\M_\k(m))\ge m$. 
Furthermore, we show in Corollary \ref{cor:strong ox} that, for $m\ge 2$, 
the lattice $L_{\le 2}(\AG (m, \k))$ is realizable over 
$\C$ if and only if $m=2$ and $\k=\FF_3$, thereby strengthening 
a classical result from matroid theory \cite{Ox}.   

Using a delicate analysis of $3$-nets supported by the matroids 
$\M(m)=\M_{\FF_3}(m)$ and a result of Yuzvinsky \cite{Yu04} in 
complex projective geometry, we establish in Corollary \ref{cor=t16gral} the 
following non-realizability criterion. 

\begin{theorem}
\label{thm:intro-mat}
Let $\M$ be a simple matroid, and suppose there are $3$-nets 
$\NN$, $\NN'$, and $\NN''$ on $\M$ such that $[\lambda_{\FF_3}(\NN)]$, 
$[\lambda_{\FF_3}(\NN')]$, and $[\lambda_{\FF_3}(\NN'')]$ are independent in 
$Z_{\FF_3}(\M)/ B_{\FF_3}(\M)$.  
Then $\M$ is not realizable over $\C$. 
\end{theorem}

For an arrangement $\A$, the irreducible components of $\RR_1(\A,\C)$
corresponding to multinets on $\A$ are called {\em essential}. We denote 
those components  arising from $k$-nets by $\E_k (\A)$. 
By the the above discussion, $\E_k(\A)=\emptyset$ for $k\ge 5$. 
In \S\ref{ssec=55}, we use Theorem \ref{teo=lambdaintro}
to obtain a good estimate on the size of these sets in the 
remaining cases. 

\begin{theorem}
\label{thm:essintro}
Let $\A$ be an arrangement.  For $k=3$ or $4$, 
\begin{equation}
\label{eq=essboundintro}
\abs{\E_k(\A)} \le \frac{k^{\beta_{\k}(\A)}-1}{(k-1)!}, 
\end{equation}
where $\k =\FF_k$.  
Moreover, the sets 
$\E_3(\A)$ and $\E_4(\A)$ cannot be simultaneously non-empty. 
\end{theorem}

\subsection{Modular bounds}
\label{subsec:intro-bound}

Work of Cohen and Orlik \cite[Theorem 1.3]{CO}, 
as sharpened by Papadima and Suciu \cite[Theorem 11.3]{PS-tams}, 
gives the following inequalities:
\begin{equation}
\label{eq:bound}
\text{$\e_{p^s} (\A) \le \beta_p(\A)$, for all $s\ge 1$}.
\end{equation}

In other words, the exponents of prime-power order $p^s$ are bounded above by the  
(combinatorially defined) $\beta_p$-invariants of the 
arrangement. As shown in \cite{PS-tams},  these bounds 
are of a topological nature: they are valid for spaces much more
general than arrangement complements, but they are far from being 
sharp in complete generality.  The modular bounds were first used  
in \cite{MP} to study the algebraic monodromy of the Milnor fibration, 
especially in the context of (signed) graphic arrangements.

We are now ready to state our next main result, which in particular 
shows that, under certain combinatorial conditions, the 
above modular bounds are sharp, at least for the prime 
$p=3$ and for $s=1$.

\begin{theorem}
\label{thm:main1}
Let $\M$ be a simple matroid.
Suppose $L_2(\M)$ has no flats of multiplicity $3r$, for any $r>1$. 
Then, the following conditions are equivalent:
\begin{romenum}
\item \label{a1}
$L_{\le 2}(\M)$ admits a reduced $3$-multinet. 
\item \label{a2}
$L_{\le 2}(\M)$ admits a $3$-net. 
\item \label{a3}
$\beta_3(\M) \ne 0$.
\suspend{romenum}
Moreover, if $\M$ is realized by an arrangement $\A$, the following hold:
\resume{romenum}
\item \label{a5}
$\beta_3(\A)\le 2$.
\item \label{a6}
$e_3(\A)=\beta_3(\A)$.
\item \label{a7}
$\abs{\E_3(\A)} = (3^{\beta_{3}(\A)}-1)/2$.
\end{romenum}
\end{theorem}

In the matroidal part of the above result, the equivalence
\eqref{a1}$\Leftrightarrow$\eqref{a2} follows immediately 
from Lemma \ref{lem:rednet}.  The key matroidal equivalence,
\eqref{a2}$\Leftrightarrow$\eqref{a3}, which uses 
Theorem \ref{teo=lambdaintro}, is proved in \S\ref{ssec=41}. 
As we saw in \S\ref{subsec:intr-mat}, the invariant $\beta_3(\M)$ 
can take arbitrary large values.  The striking fact, though, is that its 
range of values is drastically constrained in the realizable case from
Theorem \ref{thm:main1}\eqref{a5}.

Indeed, for arrangements $\A$, we establish the crucial inequality \eqref{a5}  
in Theorem \ref{thm:main4}, 
using in an essential way Theorem \ref{thm:intro-mat}. 
In Theorem \ref{thm:b3e3}, we prove the implication 
\eqref{a5}$\Rightarrow$\eqref{a6}. Finally, 
equality \eqref{a7} is established in \S\ref{ssec=55}.
In the particular case when $\mult(\A)\subseteq \{3\}$, parts \eqref{a5} 
and \eqref{a6} together imply Theorem \ref{thm:main0}.  

Our assumption on multiplicities is definitely needed. This is illustrated in 
Example \ref{ex:B3 bis}, where we produce a family of 
arrangements $\{\A_{3d+1}\}_{d\ge 1}$ having rank-$2$ 
flats of multiplicity $3(d+1)$: these arrangements support no 
reduced $3$-multinets, yet satisfy $e_3(\A_{3d+1})=\beta_3(\A_{3d+1})=1$;
in particular, property \eqref{a7} fails.  
Nevertheless, both \eqref{a5} and \eqref{a6} hold for this 
family of arrangements, as well as for the related 
family of monomial arrangements from Example \ref{ex:cevad}, 
which also violate our hypothesis.  

Our approach also allows us to characterize $4$-nets in terms of 
mod $2$ resonance, and to find combinatorial conditions which
imply that the modular bounds \eqref{eq:bound} are again sharp, 
for $p=2$ and $s\le 2$.  The next result is proved in \S\ref{subsec:e2beta2}.

\begin{theorem}
\label{thm:2main1}
For a simple matroid $\M$, the following are equivalent:
\begin{romenum}
\item \label{2m1}
$Z'_{\FF_4}(\M) \ne B_{\FF_4}(\M)$.
\item \label{2m3}
$L_{\le 2}(\M)$ supports a $4$-net. 
\end{romenum} 
If $\M$ is realized by an arrangement $\A$ with $\beta_2(\A)\le 2$ 
and the above conditions hold, then 
$\e_2(\A)= \e_4(\A)=\beta_2(\A)=2$. 
\end{theorem}

\subsection{Flat connections}
\label{ssec=flatintro}

Foundational results due to Goldman and Millson \cite{GM}, together with 
related work from \cite{KM, DPS, DP}, imply that the local geometry of representation 
varieties of fundamental groups of arrangement complements in linear algebraic groups,
near the trivial representation, is determined by the global geometry of varieties of flat connections 
on Orlik--Solomon algebras with values in the corresponding Lie algebras.

In this paper, we establish a link between the information on modular resonance encoded by
non-constant special $\k$-cocycles on an arrangement $\A$, and $\g$-valued flat connections on 
$A(\A)\otimes \C$, for an arbitrary finite-dimensional complex Lie algebra $\g$. More precisely,
we denote by $\HH^{\k}(\g) \subseteq \g^{\k}$ the subspace of vectors with zero sum of coordinates,
and declare a vector in $\HH^{\k}(\g)$ to be regular if the span of its coordinates has dimension
at least $2$. Inside the variety of flat connections, $\F (A(\A)\otimes \C, \g)$, the elements which
do not come from Lie subalgebras of $\g$ of dimension at most $1$ are also called regular.

In Proposition \ref{prop=liftev}, we associate to every special cocycle 
$\tau \in Z'_{\k}(\A)\setminus B_{\k}(\A)$ an embedding 
$\ev_{\tau}\colon \HH^{\k}(\g) \inj \F (A(\A)\otimes \C, \g)$ 
which preserves the regular parts. Building on recent work 
from \cite{DP, MPPS}, we then exploit this construction in two ways,
for $\g=\sl_2 (\C)$. On one hand, as noted in Remark \ref{rem=lambdainv}, 
the construction gives the inverse of the map $\lambda_{\k}$ from 
Theorem \ref{teo=lambdaintro}\eqref{li1}.  On the other hand, 
we use a version of this construction, involving a subarrangement of $\A$ 
as a second input, to arrive at the following result, which is proved in 
Theorem \ref{thm=freg3}.

\begin{theorem}
\label{teo=modtoflatintro}
Suppose that, for every subarrangement $\B \subseteq \A$, all essential components 
of $\RR_1(\B,\C)$ arise from nets on $\B$. Then 
\[
\F_{\reg}(A(\A)\otimes \C, \sl_2 (\C))= \bigcup_{\B, \tau} \; 
\ev^{\B}_{\tau} (\HH^{\k}_{\reg} (\sl_2 (\C)))\, ,
\]
where the union is taken over all $\B \subseteq \A$ 
and all non-constant special $\k$-cocycles $\tau\in Z'_{\k}(\B)\setminus B_{\k}(\B)$.
\end{theorem} 

When $\A$ satisfies the above combinatorial condition (for instance, when $\A$ is an 
unsigned graphic arrangement), it follows that the variety of $\sl_2 (\C)$-valued 
flat connections has an interesting property: it can be reconstructed in an explicit way 
from information on modular resonance. 

\subsection{Discussion}
\label{subsec:disc}

We return now to Problem \ref{prob:mfa}, and discuss the literature 
surrounding it, as well as our approach to solving it 
in some notable special cases. 
Nearly half the papers in our bibliography are directly related to this 
problem. This (non-exhaustive) list of papers may give the reader an idea about the 
intense activity devoted to this topic, and the variety of tools used 
to tackle it. 

In \cite{BDS, CL, CDO, Di, D12, DIM, DL, Li02, Li}, mostly geometric 
methods (such as superabundance of linear systems of polynomials, 
logarithmic forms, and Mixed Hodge theory) have been used. 
It seems worth mentioning that our approach also provides  
answers to rather subtle geometric questions. For instance, 
a superabundance problem raised by Dimca in \cite{Di} 
is settled in Remark \ref{rem:fourth}.

The topological approach to Problem \ref{prob:mfa} traces its origins 
to the work of Cohen and Suciu \cite{CS95, CS99} on Milnor fibrations 
and characteristic varieties of arrangements, which builds in turn on 
Arapura's theory \cite{Ar} of characteristic varieties of quasi-projective 
manifolds.  This theory, as refined in \cite{ACM, Di07}, provides a 
geometric interpretation of these topologically defined varieties in 
terms of (orbifold) pencils.

A crucial ingredient in our approach is the idea to connect the Orlik--Solomon 
algebra in positive characteristic to the monodromy of the Milnor fibration.  This idea, 
which appeared in \cite{CO, De02}, was developed and generalized in \cite{PS-tams}. 
The modular bounds from \eqref{eq:bound}, first exploited in a systematic 
way by M\u{a}cinic and Papadima in \cite{MP}, have since been put to 
use in \cite{BY, DIM, TY}. Theorem \ref{thm:main0} is used by Dimca in \cite{D14}
to solve a difficult problem, namely the combinatorial computation of the 
equivariant Poincar\'{e}-Deligne polynomial for the Milnor fiber of a triple 
point line arrangement.

On the combinatorial side, multinets and their relationship with complex 
resonance varieties, established in \cite{FY, MB} and further developed in 
\cite{PY, Yu09} play an important role in \cite{DS13, DIM, DP11, Su14, Su17}, 
and are key to our approach. Here, the novelty in our viewpoint is to relate
(multi)nets to modular resonance and varieties of flat connections.

\subsection{Conclusion}
\label{subsec:conclude}

The many examples we discuss in this paper show a strikingly 
similar pattern, whereby the only interesting primes, as far as 
the algebraic monodromy of the Milnor fibration goes, are $p=2$ 
and $p=3$.  Furthermore, all rank $3$ simplicial arrangements examined 
by Yoshinaga in \cite{Yo} satisfy $e_3(\A)=0$ or $1$, and $e_d(\A)=0$, 
otherwise. Finally, we do not know of any arrangement $\A$ 
of rank at least $3$ for which $\beta_p(\A) \ne 0$ if 
$p>3$. By \cite{MP}, no such example may be found among 
subarrangements of non-exceptional Coxeter arrangements. 

Theorems \ref{thm:main1} and \ref{thm:2main1}, together 
with these and other considerations lead us to formulate the following 
conjecture. 

\begin{conjecture}
\label{conj:mf}
Let $\A$ be an arrangement of rank at least $3$.  Then 
$\e_{p^s}(\A)=0$ for all primes $p$ and integers $s\ge 1$, with two 
possible exceptions:
\begin{equation}
\label{eq:e2e3}
\e_2(\A)= \e_4(\A)=\beta_2(\A) \:\text{ and }\: \e_3(\A)=\beta_3(\A). 
\end{equation}
\end{conjecture}

When $\e_d(\A)=0$ for all divisors $d$ of $\abs{\A}$ which 
are not prime powers, this conjecture would give the following
complete answer to Problem \ref{prob:mfa}:
\begin{equation}
\label{eq:delta arr}
\Delta_{\A}(t)=(t-1)^{\abs{\A}-1} ((t+1)(t^2+1))^{\beta_2(\A)} 
(t^2+t+1)^{\beta_3(\A)}. 
\end{equation} 

Recent results from \cite{MPP, D16, DS} establish the validity of this conjecture, in the
strong form \eqref{eq:delta arr}, for all complex reflection arrangements.

\subsection{Organization of the paper}
\label{subsec:org}

We start in \S\ref{sect:nets} with a review of matroids and multinets, 
and establish some simple lemmas which will be of use later on. 
In \S\ref{sect:res} we discuss the Orlik--Solomon algebra, the resonance 
varieties, and the Aomoto--Betti numbers of a matroid,  and we explore  
some of the constraints imposed on those numbers 
by the existence of nets on the matroid.  We then construct in 
\S\ref{sect:small} suitable parameter sets for nets 
on matroids, and relate these parameter sets to modular resonance, 
leading to a proof of Theorem \ref{teo=lambdaintro}  
and the combinatorial parts of Theorems \ref{thm:main1} and 
\ref{thm:2main1}. 

We switch our point of view in \S\ref{sect:flat}, where 
we study the space of $\g$-valued flat connections 
on the Orlik--Solomon algebra of a simple matroid $\M$, 
and the closely related holonomy Lie algebra $\h(\M)$. 
This analysis is continued in \S\ref{sect:flat res net}, where 
we find a combinatorial condition insuring that the variety of 
$\sl_2 (\C)$-valued flat connections on the Orlik--Solomon algebra 
of an arrangement $\A$ can be reconstructed explicitly from information 
on modular resonance.  The definitions and results from these 
two sections regarding the space of flat connections are not 
used in the remainder of the paper, but are of independent interest. 

In \S\ref{sect:res vars} we narrow our focus to realizable matroids, 
and recall the description of the resonance variety $\RR_1(\A,\C)$ of an 
arrangement $\A$ in terms of multinets on subarrangements of $\A$.
As an application of these techniques, we prove
Theorem \ref{thm:essintro} and Theorem \ref{thm:main1}\eqref{a7}. 
In \S\ref{sect:cjl milnor} we use the jump loci for homology in rank $1$ 
local systems to derive information on the characteristic polynomial of 
the algebraic monodromy of the Milnor fiber $F(\A)$.  In the process, 
we establish implication 
\eqref{a5}$\Rightarrow$\eqref{a6} from Theorem \ref{thm:main1},
and we finish the proof of Theorem \ref{thm:2main1}.

In the last section, 
we construct in \S\ref{ss71}--\ref{subsec:realize} an infinite family of rank 
$3$ matroids, $\M(m)$, which are realizable over $\C$ if and only if $m= 2$, 
and which have the property that $\beta_3(\M(m))=m$.  
Finally, in \S\ref{ss72}--\ref{ss74}   we use this information 
to establish the key part \eqref{a5} of Theorem \ref{thm:main1}, 
thereby completing the proof of this theorem. 

To recap, the logical dependence of the remaining sections is given by the 
following Leitfaden:
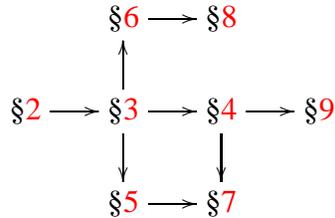
\[
\xymatrixcolsep{18pt}
\xymatrixrowsep{18pt}
\xymatrix{
& \S\ref{sect:res vars} \ar[r] & \S\ref{sect:cjl milnor} \\
\S\ref{sect:nets} \ar[r] &  \S\ref{sect:res} \ar[r] \ar[u] \ar[d] 
& \S\ref{sect:small} \ar[r]  \ar[d] &\S\ref{sec:matr}  \\
& \S\ref{sect:flat} \ar[r] & \S\ref{sect:flat res net}
}
\]

\section{Matroids and multinets}
\label{sect:nets}

The combinatorics of a hyperplane arrangement is encoded 
in its intersection lattice, which in turn can be viewed as a 
lattice of flats of a realizable matroid.  In this section, we 
discuss multinet structures on matroids, with special emphasis on nets. 
As a byproduct, we prove the combinatorial equivalence 
\eqref{a1} $\Leftrightarrow$ \eqref{a2}  from Theorem \ref{thm:main1}.  

\subsection{Matroids}
\label{subsec:matroids}
We start by reviewing the notion of matroid.  There are many ways to 
axiomatize this notion, which unifies several concepts in linear algebra, 
graph theory, discrete geometry, and the theory of hyperplane arrangements, 
see for instance Wilson's survey \cite{W}.  We mention here only the 
ones that will be needed in the sequel. 

A {\em matroid}\/ is a finite set $\M$, 
together with a collection of subsets, 
called the {\em independent sets}, which satisfy the following axioms: 
(1) the empty set is independent; (2) any proper subset of 
an independent set is independent; and (3) if $I$ and $J$ 
are independent sets and $\abs{I} > \abs{J}$, then there exists 
$u \in I \setminus J$ such that $J \cup \{u\}$ is independent.  
A maximal independent set is called a {\em basis}, while a 
minimal dependent set is called a {\em circuit}. 

The {\em rank}\/ of a subset $S\subset \M$ is the size of the largest 
independent subset of $S$. A subset is {\em closed}\/ if it is maximal 
for its rank; the closure $\overline{S}$ of a subset $S\subset \M$ is 
the intersection of all closed sets containing $S$. Closed sets are 
also called {\em flats}. 

We will consider only {\em simple}\/ matroids, defined by the 
condition that all subsets of size at most two are independent.

The set of flats of $\M$, ordered by inclusion, forms a 
geometric lattice, $L(\M)$, whose atoms are the elements 
of $\M$. We will denote by $L_s(\M)$ the set of rank-$s$ flats, 
and by $L_{\le s}(\M)$  the sub-poset of flats of rank 
at most $s$.  We say that a flat $X$ has multiplicity $q$ if 
$\abs{X}=q$.  The join of two flats $X$ and $Y$ is given by 
$X\vee Y=\overline{X\cup Y}$, while the meet is given by 
$X\wedge Y=X\cap Y$.

\subsection{Hyperplane arrangements}
\label{subsec:arrs}

An arrangement of hyperplanes is a finite set $\A$ of 
codi\-mension-$1$ linear subspaces in a finite-dimensional, 
complex vector space $\C^{\ell}$.  We will assume throughout 
that the arrangement is central, that is, all the hyperplanes pass 
through the origin.  Projectivizing, we obtain an arrangement 
$\pA=\{\bar{H}\mid H\in \A\}$ of projective, codimension-$1$ 
subspaces in $\CP^{\ell-1}$, from which $\A$ can be reconstructed 
via a coning construction.

The combinatorics of the arrangement is encoded in its 
{\em intersection lattice}, $L(\A)$.  This is the poset of all 
intersections of hyperplanes in $\A$ (also known as {\em flats}), 
ordered by reverse inclusion, and ranked by codimension.  
The join of two flats $X,Y\in L(\A)$ is given by $X\vee Y=X\cap Y$, while 
the meet is given by $X\wedge Y=\bigcap \{Z\in L(\A) \mid X+Y \subseteq Z\}$.
Given a flat $X$, we will denote by $\A_X$ the subarrangement 
$\{H\in \A\mid H\supset X\}$. 

We may view $\A$ as a simple matroid, whose points correspond to the 
hyperplanes in $\A$, with dependent subsets given by linear algebra, 
in terms of the defining equations of the hyperplanes.  In this way, the 
lattice of flats of the underlying matroid is identified with $L(\A)$. 
Under this dictionary, the two notions of rank coincide.  

A matroid $\M$ is said to be {\em realizable}\/ (over $\C$) 
if there is an arrangement $\A$ such that $L(\M)=L(\A)$. 
The simplest situation is when $\M$ has rank $2$, in 
which case $\M$ can always be realized by a pencil 
of lines through the origin of $\C^2$. 

For most of our purposes here, it will be enough to assume that 
the arrangement $\A$ lives in $\C^3$, in which case $\pA$ is an 
arrangement of (projective) lines in $\CP^2$. This is clear when the 
rank of $\A$ is at most $2$, and may be achieved otherwise 
by taking a generic $3$-slice. This operation does not
change the poset $L_{\le 2}(\A)$, or derived invariants 
such as $\beta_p(\A)$, nor does it change the monodromy 
action on $H_1(F(\A),\C)$. 

For a rank-$3$ arrangement, the set $L_1(\A)$ is in $1$-to-$1$ correspondence 
with the lines of $\pA$, while $L_2(\A)$ is in $1$-to-$1$ correspondence 
with the intersection points of $\pA$.  The poset structure of $L_{\le 2}(\A)$ 
corresponds then to the incidence structure of the point-line configuration $\pA$. 
This correspondence is illustrated in Figure \ref{fig:braid}.  
We will say that a flat $X \in L_2(\A)$ has multiplicity $q$ if 
$\abs{\A_X}=q$, or, equivalently, if the point $\bar{X}$ has 
exactly $q$ lines from $\bar{\A}$ passing through it. 

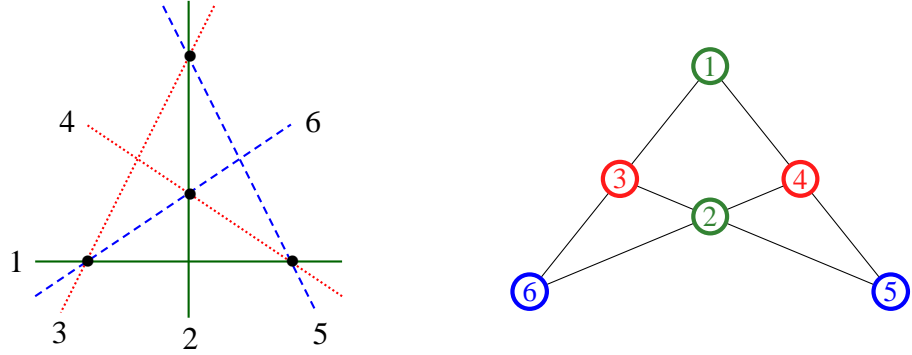
\begin{figure}
\centering
\begin{tikzpicture}[scale=0.68]
\hspace*{-0.5in}
\draw[style=thick,densely dashed,color=blue] (-0.5,3) -- (2.5,-3);
\draw[style=thick,densely dotted,color=red]  (0.5,3) -- (-2.5,-3);
\draw[style=thick,color=dkgreen] (-3,-2) -- (3,-2);
\draw[style=thick,densely dotted,color=red]  (3,-2.68) -- (-2,0.68);
\draw[style=thick,densely dashed,color=blue] (-3,-2.68) -- (2,0.68);
\draw[style=thick,color=dkgreen] (0,-3.1) -- (0,3.1);
\node at (-0.1,-2) {$\bullet$};
\node at (3.9,-2) {$\bullet$};
\node at (1.9,2) {$\bullet$};
\node at (1.9,-0.7) {$\bullet$};
\node at (-0.5,0.75) {$4$};
\node at (4.3,0.75) {$6$};
\node at (1.9,-3.5) {$2$};
\node at (-1.5,-2) {$1$};
\node at (-0.65,-3.4) {$3$};
\node at (4.45,-3.4) {$5$};
\end{tikzpicture}
\hspace*{0.65in}
\begin{tikzpicture}[scale=0.6]
\path[use as bounding box] (1,-1.5) rectangle (6,6.8); 
\path (0,0) node[draw, ultra thick, shape=circle, 
inner sep=1.3pt, outer sep=0.7pt,  color=blue] (v1) {\small{6}};
\path (2,2.5) node[draw, ultra thick, shape=circle, 
inner sep=1.3pt, outer sep=0.7pt, color=red!90!white] (v2) {\small{3}};
\path (4,5) node[draw, ultra thick, shape=circle, 
inner sep=1.3pt, outer sep=0.7pt,  color=dkgreen!80!white] (v3) {\small{1}};
\path (6,2.5) node[draw, ultra thick, shape=circle, 
inner sep=1.3pt, outer sep=0.7pt,  color=red!90!white] (v4) {\small{4}};
\path (8,0) node[draw, ultra thick, shape=circle, 
inner sep=1.3pt, outer sep=0.7pt,  color=blue] (v5) {\small{5}};
\path (4,1.666) node[draw, ultra thick, shape=circle, 
inner sep=1.3pt, outer sep=0.7pt,  color=dkgreen!80!white] (v6) {\small{2}};
\draw (v1) -- (v2) -- (v3) -- (v4) -- (v5); 
\draw (v2) -- (v6) -- (v5); 
\draw (v1) -- (v6) -- (v4);
\end{tikzpicture}
\caption{A $(3,2)$-net on the ${\rm A}_3$ arrangement, 
and on the corresponding matroid}
\label{fig:braid}
\end{figure}
\setcounter{figure}{0}

\subsection{Multinets on matroids}
\label{subsec:multinets}

Guided by the work of Falk and Yuzvinsky \cite{FY}, we 
define the following structure on a matroid $\M$---in 
fact, on the poset $L_{\le 2}(\M)$. 

A {\em multinet}\/ on $\M$ is a partition into $k\ge 3$ subsets 
$\M_1, \dots, \M_k$, together with an assignment of multiplicities, 
$m\colon \M\to \N$, and a subset $\XX\subseteq L_2(\M)$ 
with $\abs{X}>2$ for each $X\in \XX$, called the base locus, such that:
\begin{enumerate}
\item  \label{mu1} 
There is an integer $d$ such that $\sum_{u\in \M_{\alpha}} m_u=d$, 
for all $\alpha\in [k]$.
\item  \label{mu2} 
For any two points $u, v\in \M$ in different classes, the flat spanned by 
$\set{u,v}$ belongs to $\XX$.
\item  \label{mu3} 
For each $X\in\XX$, the integer 
$n_X:=\sum_{u\in \M_\alpha \cap X} m_u$ 
is independent of $\alpha$.
\item  \label{mu4} 
For each $1\leq \alpha \leq k$ and $u,v\in \M_{\alpha}$, there is a 
sequence $u=u_0,\ldots, u_r=v$ such that $u_{i-1}\vee u_i\not\in\XX$ for
$1\leq i\leq r$.
\end{enumerate}

We say that a multinet  $\NN$ as above has $k$ classes 
and weight $d$, and refer to it as a $(k,d)$-multinet, or simply as 
a $k$-multinet.  Without essential loss of generality, we may 
assume that $\gcd \{m_u\}_{u\in \M}=1$.  

If all the multiplicities are equal to $1$, the multinet is said 
to be {\em reduced}.  If $n_X=1$, for all $X\in \XX$, 
the multinet is called a {\em $(k,d)$-net};  in this case, 
the multinet is reduced, and every flat in the base locus contains precisely 
one element from each class.  A $(k,d)$-net $\NN$ is {\em non-trivial}\/ if $d>1$, 
or, equivalently, if the matroid $\M$ has rank at least $3$.

The symmetric group $\Sigma_k$ acts freely on the set of $(k,d)$-multinets 
on $\M$, by permuting the $k$ classes. Note that the $\Sigma_k$-action 
on $k$-multinets preserves reduced $k$-multinets as well as $k$-nets. 

\begin{figure}
\centering
\subfigure{%
\begin{minipage}{0.47\textwidth}
\centering
\begin{tikzpicture}[scale=0.7]
\draw[style=thick,color=dkgreen] (0,0) circle (3.1);
\node at (-2.4,0.3){2}; 
\node at (0,-2.6){2}; 
\node at (3.35,0.5){2}; 
\clip (0,0) circle (2.9);
\draw[style=thick,densely dashed,color=blue] (-1,-2.1) -- (-1,2.5);
\draw[style=thick,densely dotted,color=red] (0,-2.2) -- (0,2.5);
\draw[style=thick,densely dashed,color=blue] (1,-2.1) -- (1,2.5);
\draw[style=thick,densely dotted,color=red] (-2.5,-1) -- (2.5,-1);
\draw[style=thick,densely dashed,color=blue] (-2.5,0) -- (2.5,0);
\draw[style=thick,densely dotted,color=red] (-2.5,1) -- (2.5,1);
\draw[style=thick,color=dkgreen]  (-2,-2) -- (2,2);
\draw[style=thick,color=dkgreen](-2,2) -- (2,-2);
\end{tikzpicture}
\caption{\!A $(3,4)$-multinet} 
\label{fig:b3 arr}
\end{minipage}
}
\subfigure{%
\begin{minipage}{0.5\textwidth}
\centering
\begin{tikzpicture}[scale=0.8]
\draw[style=thick,densely dotted,color=red] (0,0) circle (3.1);
\clip (0,0) circle (2.9);
\draw[style=thick,densely dotted,color=red] (0,-2.8) -- (0,2.8);  
\draw[style=thick,densely dotted,color=red] (-2.6,-1) -- (2.6,-1); 
\draw[style=thick,densely dotted,color=red] (-2.6,1) -- (2.6,1); 
\draw[style=thick,densely dashed,color=blue] (-0.5,-2.7) -- (-0.5,2.7); 
\draw[style=thick,densely dashed,color=blue] (1.5,-2.4) -- (1.5,2.4); 
\draw[style=thick,densely dashed,color=blue]  (-2.5,-2) -- (2.2,2.7); 
\draw[style=thick,densely dashed,color=blue](-2.2,1.7) -- (2.2,-2.7); 
\draw[style=thick,color=dkgreen] (-1.5,-2.4) -- (-1.5,2.4);  
\draw[style=thick,color=dkgreen] (0.5,-2.7) -- (0.5,2.7);  
\draw[style=thick,color=dkgreen]  (-1.7,-2.2) -- (2.2,1.7);  
\draw[style=thick,color=dkgreen](-2,2.5) -- (2.2,-1.7);  
\end{tikzpicture}
\caption{A reduced $(3,4)$-multi\-net, but not a $3$-net}
\label{fig:non-net}%
\end{minipage}
}
\end{figure}

We will say that an arrangement $\A$ admits a multinet if the matroid 
realized by $\A$ does. 
The various possibilities are illustrated in the above figures. 
Figure \ref{fig:braid} shows a $(3,2)$-net on a planar slice 
of the reflection arrangement of type ${\rm A}_3$.  
Figure \ref{fig:b3 arr} shows a non-reduced $(3,4)$-multinet 
on a planar slice of the reflection arrangement of type ${\rm B}_3$.  
Finally, Figure \ref{fig:non-net} shows a simplicial arrangement of $12$ 
lines in $\CP^2$ supporting a reduced $(3,4)$-multinet which is not a $3$-net. 
For more examples, we refer to  \cite{FY, Yo, Yu04}.

\subsection{Reduced multinets and nets}
\label{subsec:red multi}

Let $\NN$ be a multinet on a matroid $\M$, with associated 
classes $\{\M_1,\dots, \M_k\}$. For each flat $X\in L_2(\M)$, let us write 
\begin{equation}
\label{eq:supp}
\supp_{\NN}(X)=\{\alpha \in [k] \mid \M_{\alpha}\cap X \ne \emptyset\}.
\end{equation} 

Evidently, $\abs{\supp_{\NN}(X)} \le \abs{X}$.  
Notice also that $\abs{\supp_{\NN}(X)}$ is either $1$ (in which case 
we say $X$ is mono-colored), or $k$ (in which case we say 
$X$ is multi-colored).  Here is an elementary lemma.

\begin{lemma}
\label{lem:rednet}
Suppose a matroid $\M$ has no $2$-flats of multiplicity $kr$, 
for any $r>1$.  
Then every reduced $k$-multinet on $\M$ is a $k$-net. 
\end{lemma} 

\begin{proof}
Let $X$ be a flat in the base locus of a $k$-multinet $\NN$; 
then $\abs{\supp_{\NN}(X)}=k$.  If the multinet is reduced, we 
have that $\abs{X}=k n_X$.  Since, by assumption, 
$L_2(\M)$ has no flats of multiplicity $kr$, with $r>1$, 
we must have $n_X=1$.  Thus, the multinet is a net. 
\end{proof} 

In the case when $k=3$, Lemma  \ref{lem:rednet} proves implication  
\eqref{a1} $\Rightarrow$ \eqref{a2}  from Theorem \ref{thm:main1}.  
(The implication \eqref{a2} $\Rightarrow$ \eqref{a1}  from that theorem 
is of course trivial.)

Work of Yuzvinsky \cite{Yu04, Yu09} and Pereira--Yuzvinsky \cite{PY} 
shows that, if $\NN$ is a $k$-multinet on a realizable matroid, 
with base locus of size greater than $1$, then $k=3$ 
or $4$; furthermore, if $\NN$ is not reduced, then $k$ 
must equal $3$.

Work of Kawahara from \cite[\S3]{Ka} shows that nets on matroids abound.   
In particular, this work shows that, for any $k\ge 3$, there is a simple matroid 
$\M$ supporting a non-trivial $k$-net.  By the above, if $\M$ is realizable, 
then $k$ can only be equal to $3$ or $4$.

Let us look in more detail at the structure of nets on matroids.  
First recall the following well-known definition.  
Given a matroid $\M$ and  a subset $S \subseteq \M$, 
we say that $S$ is {\em line-closed}\/ 
(in $\M$) if $X$ is closed in $\M$, for all $X\in L_2(S)$.  
Clearly, this property is stable under intersection.

\begin{lemma}
\label{lem:net props}
Assume a matroid $\M$ supports a $k$-net with parts 
$\M_{\alpha}$.  Then:
\begin{enumerate}
\item \label{n1}
Each submatroid $\M_{\alpha}$ has the same cardinality, 
equal to $d:=\abs{\M}/k$. 

\item \label{n3}
Each submatroid $\M_{\alpha}$ is line-closed in $\M$. 
\end{enumerate}
\end{lemma}

\begin{proof}
Part \eqref{n1} follows from the definitions.  To 
prove \eqref{n3}, let $X$ be a flat in $L_2(\M_{\alpha})$, and 
suppose there is a point $u\in (\M\setminus \M_{\alpha})\cap \overline{X}$,
where  $\overline{X}$ denotes the closure in $\M$. 
Since  $X\in L_2(\M_{\alpha})$, there exist distinct points 
$v, w\in \M_{\alpha}\cap  X$.  Since also $u\in \overline{X}$, the flat  
$\overline{X}$ must belong to the base locus.  On the other hand, 
$\overline{X}$ contains a single point from $\M_{\alpha}$. This is  
a contradiction, and we are done.
\end{proof}

\subsection{Nets and Latin squares}
\label{subsec:nets}

The next lemma provides an alternative definition of nets. 
The lemma, which was motivated by \cite[Definition 1.1]{Yu04} 
in the realizable case, will prove to be useful in the sequel.

\begin{lemma}
\label{lem:lsq}
A $k$-net $\NN$ on a matroid $\M$ is a partition with non-empty blocks,
$\M =\coprod_{\alpha \in [k]} \M_{\alpha}$, with the property that, 
for every $u\in \M_{\alpha}$ and $v\in \M_{\beta}$ with $\alpha \ne \beta$ 
and every $\gamma\in [k]$, 
\begin{equation}
\label{eq=altnet}
\abs{(u \vee v) \cap \M_{\gamma}}=1.
\end{equation}
\end{lemma}

\begin{proof}
Plainly, the net axioms \eqref{mu2} and \eqref{mu3} from \S\ref{subsec:multinets} 
imply the following dichotomy, for an arbitrary flat $X\in L_2(\M)$: either $X$ is mono-colored, 
or $X$ belongs to the base locus $\XX$ and $\abs{X}=\abs{\supp_{\NN}(X)}=k>2$. 
Hence, we may replace $\XX$ by the subset of multi-colored, rank-$2$ flats, 
i.e., the set of flats of the form $u\vee v$, where $u\in \M_{\alpha}$, 
$v\in \M_{\beta}$, and $\alpha \ne \beta$. In this way, axiom \eqref{mu2}
may be eliminated, and  axiom \eqref{mu3} reduces to property \eqref{eq=altnet}.
In turn, this property clearly implies axiom \eqref{mu4}, by taking $r=1$, 
$u_0=u$, and $u_1=v$. 

In order to complete the proof, we are left with showing that 
property \eqref{eq=altnet} implies axiom \eqref{mu1}. To this end, 
let us fix $v\in \M_{\gamma}$, and define for each $\alpha \ne \gamma$ 
a function 
\begin{equation}
\label{eq:fv}
f_v \colon \M_{\alpha} \to \{ X\in \XX \mid v\in X \}
\end{equation}
by setting $f_v (u)=u\vee v$. By property \eqref{eq=altnet}, the function 
$f_v$ is a bijection. Finally, given $\alpha \ne \beta$ in $[k]$, pick a third 
element $\gamma \in [k]$ and a point $v\in \M_{\gamma}$ to infer that 
$\abs{\M_{\alpha}}= \abs{\M_{\beta}}$. This verifies axiom \eqref{mu1},
and we are done.
\end{proof}

A {\em Latin square}\/ of size $d$ is a matrix corresponding 
to the multiplication table of a quasi-group of order $d$; that is to say, 
a $d\times d$ matrix $\Lambda$, with each row and column a 
permutation of the set $[d]=\{1,\dots, d\}$.  

In the sequel, we will make extensive use of $3$-nets. In view of 
Lemma \ref{lem:lsq}, a $3$-net on a matroid $\M$ is a partition into 
three non-empty subsets $\M_1, \M_2, \M_3$  with the 
property that, for each pair of points $u,v\in \M$ in 
different classes, we have $u\vee v=\set{ u,v,w }$, for some 
point $w$ in the third class.

Three-nets  are intimately related to Latin squares. If 
$\M$ admits a $(3,d)$-net with parts $\M_1,\M_2,\M_3$, then 
the multi-colored $2$-flats define a Latin square $\Lambda$: 
if we label the points of $\M_{\alpha}$ as 
$u^{\alpha}_{1},\dots, u^{\alpha}_{d}$, 
then the $(p,q)$-entry of this matrix is the integer $r$ 
given by the condition that 
$\{ u^1_p, u^2_q, u^3_{r} \} \in L_2(\M)$.
A similar procedure shows that a $k$-net is 
encoded by a $(k-2)$-tuple of orthogonal Latin squares. 

The realizability of $3$-nets by line arrangements in $\CP^2$ 
has been studied by several authors, including Yuzvinsky \cite{Yu04}, 
Urz\'{u}a \cite{Ur}, and Dimca, Ibadula, and M\u{a}cinic \cite{DIM}. 

\begin{example}
\label{ex:kawa}
Particularly simple is the following construction, due to 
Kawahara \cite{Ka}: given any Latin square, there is a 
matroid with a $3$-net realizing it, such that each submatroid 
obtained by restricting to the parts of the $3$-net is a uniform matroid. 

In turn, some of these matroids may be realized by line arrangements 
in $\CP^2$. For instance, suppose $\Lambda$ is the multiplication table 
of one of the groups $\Z_2$, $\Z_3$, $\Z_4$, or $\Z_2\times \Z_2$.  
Then the corresponding realization is the braid arrangement, 
the Pappus $(9_3)_1$ configuration, the Kirkman configuration, 
and the Steiner configuration, respectively. 
\end{example}

In general, though, there are many other realizations of Latin 
squares.  For example, the group $\Z_3$ admits two more 
realizations, see \cite[Theorem 2.2]{DIM} and Examples \ref{ex:second}, 
\ref{ex:third}.  

\section{Modular Aomoto--Betti numbers and resonance varieties}
\label{sect:res}

We now study two inter-related matroidal invariants:  
the Ao\-moto--Betti numbers $\beta_p$ and the resonance 
varieties in characteristic $p>0$.  In the process, we explore  
some of the constraints imposed on the $\beta_p$-invariants  
by the existence of nets on the matroid. 

\subsection{The Orlik--Solomon algebra}
\label{subsec:os alg}
As before, let $\A$ be an arrangement of hyperplanes in $\C^{\ell}$.  
The main topological invariant associated to such an arrangement 
is its complement, $M(\A)=\C^{\ell}\setminus \bigcup_{H\in \A} H$.  
This is a smooth, quasi-projective variety, with the homotopy type 
of a connected, finite CW-complex of dimension $\ell$.  

Building on work of Brieskorn, Orlik and Solomon described in 
\cite{OS} the cohomology ring $A(\A)=H^*(M(\A),\Z)$ as the quotient 
of the exterior algebra on degree-one classes dual to the 
meridians around the hyperplanes of $\A$, modulo a certain 
ideal determined by the intersection lattice. 

Based on this combinatorial description, one may associate an 
Orlik--Solomon algebra $A(\M)$ to any (simple) matroid $\M$, as follows.  
Let $E=\bigwedge(\M)$ be the exterior algebra on degree 
$1$ elements $e_u$ corresponding to the points of the matroid, 
and define a graded derivation $\partial\colon E\to E$ of degree $-1$ 
by setting $\partial(1)=0$ and $\partial (e_u)=1$, for all 
$u\in \M$. Then
\begin{equation}
\label{eq:os matroid}
A(\M) = E/\text{ideal $\{\partial (e_S) \mid$ \text{$S$ a circuit in $\M$}\}},
\end{equation}
where $e_S = \prod_{u\in S} e_u$.  As is well-known, this graded ring 
is torsion-free, and the ranks of its graded pieces are determined  
by the M\"obius function of the matroid. In particular, $A^1(\M)=\Z^{\M}$ 
(this is one instance where the simplicity assumption on $\M$ is needed).

This construction enjoys the following naturality property:  
if $\M' \subseteq \M$ is a submatroid, then the canonical 
embedding $\bigwedge(\M')\inj \bigwedge(\M)$ induces 
a monomorphism of graded rings, $A(\M')\inj A(\M)$. 

\subsection{Resonance varieties}
\label{subsec:res}

Let $A$ be a graded, graded-commutative algebra over a 
commutative Noetherian ring $\k$.  We will assume that each 
graded piece $A^q$ is free and finitely generated over $\k$, 
and $A^0=\k$. Furthermore, we will assume that $a^2=0$, for all $a\in A^1$, 
a condition which is automatically satisfied if $\k$ is a field with $\ch(\k)\ne 2$, 
by the graded-commutativity of multiplication in $A$.

For each element $a \in A^1$, we turn the algebra $A$ into a 
cochain complex, 
\begin{equation}
\label{eq:aomoto}
\xymatrix{(A , \delta_a)\colon  \ 
A^0\ar^(.66){ \delta_{a}}[r] & A^1\ar^{ \delta_{a}}[r] & A^2  \ar[r]& \cdots},
\end{equation}
using as differentials the maps  $ \delta_{a}(b)=ab$. For a finitely generated 
$\k$-module $H$, we denote by $\rank_{\k} H$ the minimal number of 
$\k$-generators for $H$.  The (degree $q$, depth $r$) {\em resonance varieties}\/ 
of $A$ are then defined as the jump loci for the cohomology of this complex, 
\begin{equation}
\label{eq:resa}
\RR^q_r(A)= \{a \in A^1 \mid \rank_{\k} H^q(A , \delta_{a}) \ge  r\}.
\end{equation}
When $\k$ is a field, it is readily seen that these sets are Zariski-closed, homogeneous 
subsets of the affine space $A^1$.  

For our purposes here, we will only consider the degree $1$ resonance 
varieties, $\RR_r(A)=\RR^1_r(A)$. Clearly, these varieties depend  
only on the degree $2$ truncation of $A$, denoted $A^{\le 2}$. Over a field $\k$, 
$\RR_r(A)$ consists of $0$, together with all elements 
$a \in A^1$ for which there exist $b_1,\dots ,b_r \in A^1$ 
such that $\dim_{\k}\spn\{a,b_1,\dots,b_r\}=r+1$ and $ab_i=0$ in $A^2$.

The degree $1$ resonance varieties over a field enjoy the following naturality 
property:  if $\varphi\colon A\to A'$ is a morphism of commutative 
graded algebras, and  $\varphi$ is injective in degree $1$, then 
$\varphi^1$ embeds $\RR_r(A)$ into $\RR_r(A')$, for each $r\ge 1$.

\subsection{The Aomoto--Betti numbers}
\label{subsec:aomoto betti}

Consider now  the algebra $A=A(\M)\otimes \k$, i.e., the Orlik--Solomon 
algebra of the matroid $\M$ with coefficients in a commutative Noetherian 
ring $\k$.  Since $A$ is a quotient of an exterior algebra, we have that 
$a^2=0$ for all $a\in A^1$. Thus, we may define the resonance varieties 
of the matroid $\M$ as 
\begin{equation}
\label{eq:res mat}
\RR_r(\M,\k) := \RR_r(A(\M)\otimes \k).
\end{equation}

If the coefficient ring is a field, these varieties essentially depend only on 
the characteristic of the field.  Indeed, if $\k\subset \mathbb{K}$ is a 
field extension, then $\RR_r(\M,\k) = \RR_r(\M,\mathbb{K}) \cap \k^{\M}$.
The resonance varieties of a realizable matroid were first 
defined and studied by Falk \cite{Fa97} for $\k=\C$, by 
Matei and Suciu \cite{MS00} for arbitrary fields, and then 
by Falk \cite{Farx, Fa07} in the general case. A complete 
description of the resonance varieties $\RR_1(\M,\C)$ for 
a (not necessarily realizable) simple matroid $\M$ were 
given by Marco Buzun\'ariz in \cite{MB}. 

Over an arbitrary Noetherian ring $\k$, the $\k$-module $A^1=\k^{\M}$ 
comes endowed with a preferred basis, which we will also write as 
$\{e_u\}_{u\in \M}$. Consider the ``diagonal" element
\begin{equation}
\label{eq:omega1}
\sigma=\sum_{u\in  \M} e_u \in A^1, 
\end{equation}
and define the {\em cocycle space}\/ of the matroid (with respect to $\sigma$) 
to be  
\begin{equation}
\label{eq:z1}
Z_{\k}(\M) = \{  \tau\in A^1 \mid \sigma \cup \tau=0\}.
\end{equation}

Over a field, the dimension of $Z_{\k}(\M)$ 
depends only on the characteristic of $\k$, and not on $\k$ itself.  
The following lemma gives a convenient system of linear 
equations for the cocycle space $Z_{\k}(\M)$, in general.
  
\begin{lemma}
\label{lem:eqs}
Let $\M$ be a matroid, and let $\k$ be a commutative Noetherian ring. 
A vector $\tau=\sum_{u\in  \M}\tau_u e_u \in \k^{\M}$ 
belongs to $Z_\k(\M)$ if and only if, for each flat $X\in L_2(\M)$ and $v\in X$, 
the following equation holds:
\begin{equation}
\label{eq=zfalk}
\sum_{u\in X} \tau_u =\abs{X} \cdot \tau_v .
\end{equation}

Furthermore, if $\k$ is a field of characteristic $p>0$, the above equations 
are equivalent to the system
\begin{equation}
\label{eq:zsigma}
\begin{cases}
\:\sum_{u\in X} \tau_u = 0 & \text{if $p\mid \abs{X}$},\\[3pt]
\:\tau_{u}=\tau_{v},\ \text{for all $u,v\in X$} & \text{if $p\nmid \abs{X}$}.
\end{cases}
\end{equation}
\end{lemma}

\begin{proof}
The first assertion follows from \cite[Theorem 3.5]{Farx}, while   
the second assertion is a direct consequence of the first one.
\end{proof}

Note that $\sigma \in Z_\k(\M)$. Define the {\em coboundary space}\/ 
of the matroid to be the submodule $B_\k(\M)\subseteq Z_\k(\M)$ 
spanned by $\sigma$.  Clearly, a vector $\tau$ as above belongs to 
$B_\k(\M)$ if and only if all its components $\tau_u$ are equal. 
Let us define the {\em Aomoto--Betti number}\/ 
over $\k$ of the matroid $\M$ as 
\begin{equation}
\label{eq=betafalk}
\beta_{\k}(\M):=\rank_{\k} Z_\k(\M)/B_\k(\M) .
\end{equation}

Clearly, $\beta_{\k}(\M)=0$ if and only if $Z_\k(\M)= B_\k(\M)$.
If $\k$ is a field of positive characteristic, the Aomoto--Betti number  
of $\M$ depends only on $p=\ch \k$, and so we will write it simply as 
$\beta_p(\M)$.  We then have
\begin{equation}
\label{eq:betap bis}
\beta_p(\M)= \dim_{\k} Z_\k(\M) -1 .  
\end{equation}

Note that $Z_\k(\M)/B_\k(\M) = H^1(A(\M)\otimes \k,\delta_{\sigma})$. 
Thus, the above $\bmod$-$p$ matroid invariant may be reinterpreted as 
\begin{equation}
\label{eq:betap again}
\beta_p(\M) = \max\{r \mid \sigma\in \RR_r(\M,\k)\}.
\end{equation} 

\subsection{Constraints on the Aomoto--Betti numbers}
\label{subsec:constraints}

We now use Lemma \ref{lem:eqs} to derive useful 
information on the matroidal invariants defined above.  
We start with a vanishing criterion, which is an immediate 
consequence of that lemma.

\begin{corollary}
\label{cor:beta zero}
If $p\nmid \abs{X}$, for any $X\in L_2(\M)$ with $\abs{X}>2$, 
then $\beta_p(\M)=0$. 
\end{corollary}

For instance, if $\M$ is a rank~$3$ uniform matroid, then $\beta_p(\M)=0$, 
for all $p$.  At the other extreme, if  all the points of $\M$ are collinear, 
and $p$ divides $\abs{\M}$, then $\beta_p(\M)=\abs{\M}-2$.

The next application provides constraints on the Aomoto--Betti 
numbers in the presence of nets.

\begin{prop}
\label{prop:betapknet}
Assume that $\M$ supports a $k$-net. Then $\beta_p (\M)=0$ if $p\not\mid k$,
and $\beta_p (\M) \ge k-2$, otherwise.
\end{prop}

\begin{proof}
Set $\k=\FF_p$.  
First suppose that $p\not\mid k$. Let $\tau\in Z_{\k}(\M)$. 
Pick $u,v\in \M_{\alpha}$, and then $w\in \M_{\beta}$,  with $\beta\ne \alpha$. 
Since $\abs{u\vee w}= \abs{v\vee w}=k$, by the net property, we infer from 
Lemma \ref{lem:eqs} that $\tau_u=\tau_w=\tau_v$. Hence, $\tau$ is constant 
on each part $\M_{\alpha}$. Applying once again Lemma \ref{lem:eqs} to a 
multi-colored flat, we deduce that $\tau\in B_{\k}(\M)$. 

Now suppose $p\mid k$, and consider the $k$-dimensional subspace of $\k^{\M}$ 
consisting of those elements $\tau$ that are constant (say, equal to $c_{\alpha}$) 
on each $\M_{\alpha}$. By Lemma \ref{lem:eqs} and the net property,
the subspace  of $\k^{\M}$  given by the equation 
$\sum_{\alpha\in [k]} c_{\alpha}=0$ is contained in
$Z_{\k}(\M)$. Thus, $\dim_{\k} Z_{\k}(\M) \ge k-1$, 
and the desired inequality follows at once.
\end{proof}

Finally, let us record a construction which relates the cocycle space 
of a matroid to the cocycle spaces of the parts of a net supported 
by the matroid. 

\begin{lemma}
\label{lem:proj}
Let $\k$ be a field. 
Suppose a matroid $\M$ supports a net.  For each part $\M_{\alpha}$, 
the natural projection $\k^{\M}\to \k^{\M_{\alpha}}$ restricts to a 
homomorphism $h_{\alpha}\colon Z_\k(\M) \to Z_\k(\M_{\alpha})$, 
which in turn induces a homomorphism 
$\bar{h}_{\alpha}\colon Z_\k(\M)/B_\k(\M) \to 
Z_\k(\M_{\alpha})/B_\k(\M_{\alpha})$. 
\end{lemma}

\begin{proof}
In view of Lemmas \ref{lem:net props}\eqref{n3} and \ref{lem:eqs}, 
the equations defining $Z_\k(\M_{\alpha})$ form a subset of the set of 
equations defining $Z_\k(\M)$.  Thus, the projection $\k^{\M}\to \k^{\M_{\alpha}}$ 
restricts to a homomorphism $Z_\k(\M) \to Z_\k(\M_{\alpha})$.  Clearly, this 
homomorphism takes $\sigma$ to $\sigma_{\alpha}$, and the second 
assertion follows. 
\end{proof}

\subsection{More on $3$-nets and the $\beta_3$ numbers}
\label{eq:beta3}

In the case when the net has $3$ parts and the ground field 
has $3$ elements, the previous lemma can be made be more precise.

\begin{prop}
\label{prop:ker1}
Let $\NN$ be a $3$-net on a matroid $\M$, and let $\k=\FF_3$.  
For each part $\M_{\alpha}$, we  have an exact sequence of 
$\k$-vector spaces, 
\begin{equation}
\label{eq:ex seq}
\xymatrix{0\ar[r] & \k \ar[r] & Z_\k(\M)/B_\k(\M) 
\ar^(.48){\bar{h}_{\alpha}}[r] & Z_\k(\M_{\alpha})/B_\k(\M_{\alpha})}.
\end{equation}
\end{prop}

\begin{proof}
Let $\tau_{\alpha}\in \k^{\M}$ be the vector whose components 
are equal to $0$ on $\M_{\alpha}$, and are equal to $1$, 
respectively $-1$ on the other two parts of the net. 
It follows from Lemmas \ref{lem:net props}\eqref{n3} 
and \ref{lem:eqs} that $\tau_{\alpha}\in Z_\k(\M)$. 
Clearly, $\tau_{\alpha}\in \ker(h_{\alpha})$.  In fact, 
as we shall see next, $\tau_{\alpha}$ generates the 
kernel of $h_{\alpha}$. 

Suppose $h_{\alpha}(\eta)=0$, for some $\eta \in Z_\k(\M)$. 
We first claim that $\eta$ must be constant on the other two 
parts, $\M_{\beta}$ and $\M_{\gamma}$. To verify this claim, 
fix a point $u\in \M_{\gamma}$ and pick $v,w \in \M_{\beta}$. 
By the net property and \eqref{eq:zsigma}, $\eta_{u} + \eta_{v} + 
\eta_{v'}=0$ and $\eta_{u} + \eta_{w} + \eta_{w'}=0$,
for some $v',w'\in \M_{\alpha} $. But $\eta_{v'}=\eta_{w'}=0$, 
by assumption.  Hence, $\eta_{v}=\eta_{w}$, and our claim follows. 

Writing now condition 
\eqref{eq:zsigma} for $\eta$ on a multi-colored flat of $\NN$, 
we conclude that $\eta \in \k \cdot \tau_{\alpha}$. 
It follows that sequence \eqref{eq:ex seq} is exact in the middle.  
Exactness at $\k$ is obvious, and so we are done.
\end{proof}

\begin{corollary}
\label{cor:betabounds}
If a matroid $\M$ supports a $3$-net with parts $\M_{\alpha}$, then
$1\le \beta_3(\M)\le \beta_3(\M_{\alpha})+1$, for all $\alpha$.  
\end{corollary}

If $\beta_3(\M_{\alpha})= 0$, for some $\alpha$, then $\beta_3(\M)=1$, while 
if $\beta_3(\M_{\alpha})=1$, for some $\alpha$, then $\beta_3(\M)=1$ or $2$.
Furthermore, as the next batch of examples shows, all three possibilities 
do occur, even among realizable matroids. 

\begin{figure}
\centering
\subfigure{%
\begin{minipage}{0.46\textwidth}
\centering
\begin{tikzpicture}[scale=0.58]
\path[use as bounding box] (-2,-0.5) rectangle (6,7); 
\path (0,4) node[fill, draw, inner sep=2pt, shape=circle, color=blue] (v1) {};
\path (2,4) node[fill, draw, inner sep=2pt, shape=circle, color=red!90!white] (v2) {};
\path (4,4) node[fill, draw, inner sep=2pt, shape=circle, color=dkgreen!80!white] (v3) {};
\path (0,2) node[fill, draw, inner sep=2pt, shape=circle, color=blue] (v4) {};
\path (2,2) node[fill, draw, inner sep=2pt, shape=circle, color=red!90!white] (v5) {};
\path (4,2) node[fill, draw, inner sep=2pt, shape=circle, color=dkgreen!80!white] (v6) {};
\path (0,0) node[fill, draw, inner sep=2pt, shape=circle, color=blue] (v7) {};
\path (2,0) node[fill, draw, inner sep=2pt, shape=circle, color=red!90!white] (v8) {};
\path (4,0) node[fill, draw, inner sep=2pt, shape=circle, color=dkgreen!80!white] (v9) {};
\draw (v1) -- (v2) -- (v3); \draw (v4) -- (v5) -- (v6); \draw (v7) -- (v8) -- (v9);
\draw (v1) -- (v4) -- (v7);  \draw (v2) -- (v5) -- (v8); \draw (v3) -- (v6) -- (v9);
\draw (v2) -- (v4) -- (v8) -- (v6) -- (v2);
\draw (v1) -- (v5) -- (v9); \draw (v3) -- (v5) -- (v7); 
\draw (v7) .. controls (-5,6) and (-1,6.8) .. (v2);
\draw (v4) .. controls (-4,7) and (0,7.8) .. (v3);
\draw (v1) .. controls (5,7.8) and (8,7) .. (v6);
\draw (v2) .. controls (6,6.8) and (9,6) .. (v9);
\end{tikzpicture}
\caption{A $(3,3)$-net on the Ceva matroid}
\label{fig:ceva}%
\end{minipage}
}
\subfigure{%
\begin{minipage}{0.5\textwidth}
\centering
\begin{tikzpicture}[scale=0.5]
\path[use as bounding box] (-1.5,-3.2) rectangle (6,6.8); 
\path (0,4) node[fill, draw, inner sep=2pt, shape=circle, color=blue] (v1) {};
\path (2,4) node[fill, draw, inner sep=2pt, shape=circle, color=red!90!white] (v2) {};
\path (4,4) node[fill, draw, inner sep=2pt, shape=circle, color=dkgreen!80!white] (v3) {};
\path (0,2) node[fill, draw, inner sep=2pt, shape=circle, color=red!90!white] (v4) {};
\path (2,2) node[fill, draw, inner sep=2pt, shape=circle, color=dkgreen!80!white] (v5) {};
\path (4,2) node[fill, draw, inner sep=2pt, shape=circle, color=blue] (v6) {};
\path (0,0) node[fill, draw, inner sep=2pt, shape=circle, color=dkgreen!80!white] (v7) {};
\path (2,0) node[fill, draw, inner sep=2pt, shape=circle, color=blue] (v8) {};
\path (4,0) node[fill, draw, inner sep=2pt, shape=circle, color=red!90!white] (v9) {};
\draw (v1) -- (v2) -- (v3); \draw (v4) -- (v5) -- (v6); \draw (v7) -- (v8) -- (v9);
\draw (v1) -- (v4) -- (v7);  \draw (v2) -- (v5) -- (v8); \draw (v3) -- (v6) -- (v9);
\draw (v4) -- (v8); \draw (v2) -- (v6);
\draw (v1) -- (v5) -- (v9); 
\draw (v7) .. controls (-5,6) and (-1,6.8) .. (v2);
\draw (v4) .. controls (-4,7) and (0,7.8) .. (v3);
\path (2,-3) node[fill, draw, inner sep=2pt, shape=circle,  color=dkbrown!90!white] (v10) {};
\path (6.12,-2.12) node[fill, draw, inner sep=2pt, shape=circle,  color=dkbrown!90!white] (v11) {};
\path (7,2) node[fill, draw, inner sep=2pt, shape=circle,  color=dkbrown!90!white] (v12) {};
\draw (v8) .. controls (2.65,-0.6) and (3,-1.5) .. (v11);
\draw (v6) .. controls (4.6,1.35) and (5.5,1) .. (v11);
\draw (v7) .. controls (0.4,-2.2) and (1.2,-2.6) .. (v10);
\draw (v9) .. controls (3.6,-2.2) and (2.8,-2.6) .. (v10);
\draw (v8) -- (v10);  \draw (v9) -- (v11);  \draw (v6) -- (v12); 
\draw (v3)  .. controls (5,3.8) and (6,3.6) .. (v12); 
\draw (v9) .. controls (5,0.2) and (6,0.4) ..  (v12); 
\end{tikzpicture}
\caption{A $(4,3)$-net on the Hessian matroid}
\label{fig:hessian}%
\end{minipage}
}
\end{figure}

\begin{example}
\label{ex:first}
First, let $\A$ be the braid arrangement from Figure \ref{fig:braid}. 
Then $\A$ admits a $(3,2)$-net with all parts $\A_{\alpha}$ in general 
position.  Hence, $\beta_3(\A_{\alpha})=0$ 
for each $\alpha$, and thus $\beta_3(\A)=1$.  
\end{example}

\begin{example}
\label{ex:second}
Next, let $\A$ be the realization of the configuration described by the 
Pappus hexagon theorem.  As noted in \cite[Example 2.3]{DIM}, 
$\A$ admits a $(3,3)$-net with two parts in general position and 
one not.  Hence, $\beta_3(\A_1)=\beta_3(\A_2)=0$ while 
$\beta_3(\A_3)=1$.  Therefore, $\beta_3(\A)=1$. 
\end{example}

\begin{example}
\label{ex:third}
Finally, let $\A$ be the Ceva arrangement, defined by the polynomial 
$Q=(z_1^3-z_2^3)(z_1^3-z_3^3)(z_2^3-z_3^3)$. As can be seen  
in Figure \ref{fig:ceva}, this arrangement admits a $(3,3)$-net with 
no parts in general position. Hence, $\beta_3(\A_{\alpha})=1$ 
for each $\alpha$. Moreover,  $\beta_3(\A)=2$, by direct computation, 
or by Proposition \ref{prop=mb3} below.
The classification results from \cite{DIM} and the above considerations 
imply that the only rank $3$ arrangement $\A$ of at most $9$ planes 
that supports a $3$-net and has $\beta_3(\A) \ge 2$ is the Ceva arrangement. 
\end{example}

\section{From modular resonance to nets}
\label{sect:small}

In this section, we construct suitable parameter sets for nets 
on matroids, and relate these parameter sets to modular resonance. 
The approach we take leads to a proof of Theorem \ref{teo=lambdaintro}  
and the combinatorial parts of Theorems \ref{thm:main1} and 
\ref{thm:2main1} from the Introduction. 

\subsection{A parameter set for nets on a matroid}
\label{ssec=34}
As before, let $\M$ be a simple matroid. Generalizing the 
previous setup, let $\k$ be a finite set of size $k\ge 3$. 
Inside the set $\k^{\M}$ of all functions $\tau\colon \M\to \k$, 
we isolate two subsets.  

The first subset, $Z'_{\k}(\M)$, consists of all functions $\tau$ 
with the property that, for every $X\in L_2(\M)$, the restriction 
$\tau\colon X\to \k$ is either constant or bijective.  The second 
subset, $B_{\k}(\M)$, consists of all constant functions.
Plainly, $B_{\k}(\M)\subseteq Z'_{\k}(\M)$. In view of 
Lemma \ref{lema=sigmadich} below, we will call the 
elements of $Z'_{\k}(\M)$ {\em special}\/ $\k$-cocycles.

Now define a function
\begin{equation}
\label{eq=lambda}
\xymatrix{\lambda_{\k} \colon \{\text{$k$-nets on $\M$}\} \ar[r]& \k^{\M}},  
\end{equation}
by associating to a $k$-net $\NN$, with partition $\M= \coprod_{\alpha \in \k} \M_{\alpha}$,
the element $\tau:= \lambda_{\k}(\NN)$ which takes the value $\alpha$ on $\M_{\alpha}$.

\begin{lemma}
\label{lema=lambdabij}
The above construction induces a bijection,
\[
\xymatrix{\lambda_{\k} \colon \{\text{$k$-nets on $\M$}\}  
\ar^(.52){\simeq}[r]& Z'_{\k}(\M)\setminus B_{\k}(\M)}.
\]
\end{lemma}

\begin{proof}
Plainly, $\lambda_{\k}$ is injective, with image disjoint from $B_{\k}(\M)$.
To show that $\lambda_{\k}(\NN)$ belongs to $Z'_{\k}(\M)$, for any $k$-net $\NN$, 
pick $X\in  L_2(\M)$. If the flat $X$ is mono-colored with respect to $\NN$, 
the restriction $\tau\colon X\to \k$ is constant, by construction. 
If $X$ is multi-colored, this restriction is a bijection, according 
to Lemma \ref{lem:lsq}, and we are done.

Finally, let $\tau$ be an element in $Z'_{\k}(\M)\setminus B_{\k}(\M)$. 
Define a partition $\M= \coprod_{\alpha \in \k} \M_{\alpha}$ by setting 
$\M_{\alpha}= \{ u\in \M \mid \tau_u= \alpha \}$.  Since $\tau$ is 
non-constant on $\M$, there must be a flat $X\in  L_2(\M)$ with 
$\tau\colon X\to \k$ bijective, which shows that all blocks of the 
partition are non-empty.  For $u\in \M_{\alpha}$ and $v\in \M_{\beta}$ 
with $\alpha \ne \beta$, we infer that $\tau\colon u \vee v\to \k$ must 
be bijective, since $\tau\in Z'_{\k}(\M)$. It follows from Lemma \ref{lem:lsq} 
that the partition defines a $k$-net $\NN$ on $\M$. By construction, 
$\lambda_{\k}(\NN)=\tau$, and this completes the proof.
\end{proof}

\begin{corollary}
\label{coro=realobs}
If $\abs{\k}>4$, then $Z'_{\k}(\M)= B_{\k}(\M)$, for any realizable matroid of 
rank at least $3$. 
\end{corollary}

\subsection{Sums in finite abelian groups}
\label{subsec:finite abelian}
Before proceeding with our main theme, let us 
consider a purely algebraic general situation. Given a finite abelian 
group $\k$, define 
\begin{equation}
\label{eq:sigmak}
\Sigma (\k)= \sum_{\alpha\in \k} \alpha.
\end{equation} 
to be the sum of the elements of the group. 
The following formula is then easily checked:
\begin{equation}
\label{eq=sigmaprod}
\Sigma (\k \times \k')= (\abs{\k'} \cdot \Sigma (\k), \abs{\k} \cdot \Sigma (\k')) .
\end{equation}

Clearly, $2 \Sigma (\k)=0$.  Moreover, if  $\k$ is cyclic, 
then $\Sigma (\k)=0$ if and only if the order of $\k$ is odd.   
These observations readily imply the following elementary lemma. 

\begin{lemma}
\label{lema=sigmazero}
Let $\k=\prod_{\text{$p$ prime}} \prod_{s\ge 1} (\Z/p^s \Z)^{e(p,s)}$
be the primary decomposition of $\k$.
Then $\Sigma (\k)=0$ if and only if\/ $\sum_s e(2,s) \ne 1$. 
\end{lemma}

Next, we examine the conditions under which $\Sigma(\k)$ vanishes 
when the group $\k$ is, in fact, a (finite) commutative ring.

\begin{lemma}
\label{lema=2mod4}
Let $k$ be a positive integer.  There is a finite commutative ring $\k$ 
with $k$ elements and satisfying $\Sigma (\k)=0$
if and only if $k \not \equiv 2$ mod $4$. Moreover, if $2\ne k=p^s$, we 
may take $\k=\FF_{p^s}$.
\end{lemma}

\begin{proof}
Let $k=\prod_p p^{v_p(k)}$ be the prime decomposition. If $\abs{\k}=k$, then
$\sum_s se(2,s)= v_2(k)$, and $k \equiv 2$ mod $4$ if and only if $v_2(k)=1$.
If $v_2(k)\ne 1$, we may take $\k= \prod_p \FF_p^{v_p(k)}$, and infer from
Lemma \ref{lema=sigmazero} that $\Sigma (\k)=0$. If $v_2(k)=1$ and $\abs{\k}=k$,
then clearly  $\sum_s e(2,s)= 1$. Again by Lemma \ref{lema=sigmazero}, this implies that 
$\Sigma (\k)\ne 0$. For the last claim, note that $\FF_{p^s}= \FF_{p}^s$, as an additive group. 
\end{proof}

\subsection{Modular resonance and multinets}
\label{subsec:special}
For the rest of this section, we will assume $\k$ is a finite commutative ring. 
In this case, $B_{\k}(\M)$ coincides with the coboundary space defined in 
\S\ref{subsec:aomoto betti}. 
The next lemma establishes a relationship between the subset $Z'_{\k}(\M) \subseteq \k^{\M}$ 
and the modular cocycle space $Z_{\k}(\M)$.  

\begin{lemma}
\label{lema=sigmadich}
Let $\M$ be a matroid and let $\k$ be a finite commutative ring. If $\Sigma (\k)=0$,
then $Z'_{\k}(\M)\subseteq Z_{\k}(\M)$. Otherwise, $Z'_{\k}(\M)\cap Z_{\k}(\M)= B_{\k}(\M)$.
\end{lemma}

\begin{proof}
For $\tau\in Z'_{\k}(\M)$ and $X\in  L_2(\M)$ with $\tau \equiv \alpha$ on $X$, 
equations \eqref{eq=zfalk} reduce to $\abs{X}\cdot \alpha=\abs{X}\cdot \alpha$.
When $\tau\colon X\to \k$ is a bijection, these equations take the form 
$\Sigma (\k)= \abs{\k}\cdot \alpha$, for every $\alpha\in \k$, or, equivalently, 
$\Sigma (\k)=0$. The desired conclusions follow.
\end{proof}

Lemmas \ref{lema=lambdabij}, \ref{lema=2mod4}, and \ref{lema=sigmadich} 
together prove Theorem \ref{teo=lambdaintro} from the Introduction.
In turn, Theorem \ref{teo=lambdaintro} applied to the case when $\k=\FF_4$ 
proves the combinatorial part of Theorem \ref{thm:2main1}, i.e.,
the equivalence \eqref{2m1}$\Leftrightarrow$\eqref{2m3} from there.

\begin{remark}
\label{rem:zz}
When $\M$ supports a $4$-net, the inclusion 
$Z'_{\FF_4}(\M)\subseteq Z_{\FF_4}(\M)$ from Theorem \ref{teo=lambdaintro}\eqref{li2}
is always strict. Indeed, let $\M =\coprod_{\alpha \in [4]} \M_{\alpha}$ be a $4$-net
partition. Define $\tau\in \FF_4^{\M}= (\FF_2 \times \FF_2)^{\M}$ to be
equal to $(0,0)$ on $\M_1$ and $\M_2$,  and to $(1,0)$ 
on $\M_3$ and $\M_4$. Using \eqref{eq:zsigma}, we easily see that 
$\tau \in Z_{\FF_4}(\M)\setminus Z'_{\FF_4}(\M)$.

In the case of $3$-nets, this phenomenon no longer occurs.  For instance, 
if $\A$ is the Ceva arrangement from Example \ref{ex:third}, then 
$\A$ admits a $3$-net, while $Z'_{\FF_3}(\A)= Z_{\FF_3}(\A)$, 
by Lemma \ref{lem:lambda}.
\end{remark}

Next, we provide an extension of Lemma \ref{lema=lambdabij}, 
from nets to multinets. 

\begin{lemma}
\label{lem:special}
Let $\k$ be a finite commutative ring with $\Sigma (\k)=0$. Then the 
function $\lambda_{\k} \colon \{\text{$k$-nets on 
$\M$}\} \inj  Z_{\k}(\M)\setminus B_{\k}(\M)$ has an injective extension,
\[
\lambda_{\k}\colon \{\text{reduced $k$-multinets on $\M$}\} 
\inj  Z_{\k}(\M)\setminus B_{\k}(\M).
\]
\end{lemma}

\begin{proof}
Let $\NN$ be a reduced $k$-multinet on $\M$. Define $\lambda_{\k} (\NN)\in \k^{\M}$
by using the underlying partition, $\M= \coprod_{\alpha \in \k} \M_{\alpha}$, exactly as in
\eqref{eq=lambda}. Clearly, $\lambda_{\k} (\NN)$ determines $\NN$. By the multinet axiom
\S\ref{subsec:multinets}\eqref{mu3}, the map $\lambda_{\k} (\NN)\colon \M \to \k$ 
is surjective; hence
$\lambda_{\k} (\NN) \not\in B_\k (\M)$. 

Now, if $X\in  L_2(\M)$ is mono-colored, i.e., 
$X\subseteq \M_{\alpha}$ for some $\alpha\in \k$, then the system of 
equations \eqref{eq=zfalk} reduces to
$\abs{X}\cdot \alpha=\abs{X}\cdot \alpha$, which is trivially satisfied. 
Otherwise, those equations take the form 
\[
\sum_{\alpha\in \k} \abs{X\cap \M_{\alpha}} \cdot \alpha= \abs{X} \cdot \beta,
\]
for all $\beta\in \k$, or, equivalently, $n_X \cdot \Sigma (\k)=0$ and $\abs{X}=0$, 
and we are done. 
\end{proof} 

\subsection{A multiplicity assumption}
\label{ssec=41}

Finally, let us consider the case when $k=3$ (and $\k=\FF_3$) in 
Theorem \ref{teo=lambdaintro}.  Under a natural multiplicity 
assumption, we are then able to say more about the cocycle space 
of our matroid $\M$.

\begin{lemma}
\label{lem:lambda}
Suppose $L_2(\M)$ has no flats of multiplicity properly divisible by $3$. 
Then $Z'_{\FF_3}(\M)=Z_{\FF_3}(\M)$. 
\end{lemma}

\begin{proof}
By \eqref{eq:zsigma}, an element  $\tau=\sum_{u\in \M} \tau_u e_u \in \FF_3^{\M}$ 
belongs to $Z_{\FF_3}(\M)$ if and only if, for each 
$X\in  L_2(\M)$, either $3$ divides $\abs{X}$ 
and  $\sum_{u\in X} \tau_u =0$, or else $\tau$  
is constant on $X$. 

In view of our multiplicity hypothesis, the first possibility 
only occurs when $X$ has size $3$, in which case the equation 
$\sum_{u\in X} \tau_u =0$ implies that the restriction 
$\tau\colon X \to \FF_3$ is either constant or bijective. 
Hence, the element  $\tau$ belongs to $Z'_{\FF_3}(\M)$, and we are done.
\end{proof}

Putting now together Theorem \ref{teo=lambdaintro}\eqref{li1} 
with Lemma \ref{lem:lambda} establishes equivalence 
\eqref{a2}$\Leftrightarrow$\eqref{a3} from Theorem \ref{thm:main1} 
in the Introduction. The remaining combinatorial part of Theorem \ref{thm:main1},
i.e., equivalence \eqref{a1}$\Leftrightarrow$\eqref{a2}, follows from
Lemma \ref{lem:rednet}.

\section{Flat connections and holonomy Lie algebras}
\label{sect:flat}

In this section, we study the space of $\g$-valued flat connections 
on the Orlik--Solomon algebra of a simple matroid $\M$, and the closely related 
holonomy Lie algebra $\h(\M)$. We construct flat connections from non-constant 
special cocycles on $\M$, and we characterize the key axiom for multinets 
on $\M$ by using $\h(\M)$.

\subsection{Flat, $\g$-valued connections} 
\label{subsec:flat}

We start by reviewing some standard material on flat connections, 
following the approach from \cite{DP, DPS, MPPS}.

Let $(A,d)$ be a commutative, differential graded 
algebra over $\C$, for short, a \cdga.  
We will assume that $A$ is connected (i.e., $A^0=\C$) 
and of finite $q$-type, for some $q\ge 1$ (i.e., $A^i$ is 
finite-dimensional, for all $i\le q$). 
The cohomology groups $H^i(A)$ are $\C$-vector spaces, 
of finite dimension if $i\le q$.  
Since $A$ is connected, the differential $d\colon A^0\to A^1$ 
vanishes, and so we may view $H^1(A)$ as a 
linear subspace of $A^1$. 

Now let $\g$ be a finite-dimensional Lie algebra over $\C$. 
On the graded vector space $A\otimes \g$, we may define 
a bracket by $[a\otimes x, b\otimes y]= ab\otimes [x,y],\ 
\text{for $a,b\in A$ and $x,y\in \g$}$. 
This functorial construction produces a differential 
graded Lie algebra $A\otimes \g$, with grading inherited 
from $A$, and differential $d (a\otimes x) = da \otimes x$. 

An element $\omega\in A^1\otimes \g$ is called an 
{\em infinitesimal, $\g$-valued flat connection}\/ on 
$(A,d)$ if $\omega$ satisfies the Maurer--Cartan equation,
\begin{equation}
\label{eq:flat}
d\omega + \tfrac{1}{2} [\omega,\omega] = 0 . 
\end{equation}

We will denote by $\F(A,\g)$ the subset of $A^1\otimes \g$ 
consisting of all flat connections.  This set  has a 
natural affine structure, and depends functorially 
on both $A$ and $\g$. Notice that $\F(A,\g)$ depends 
only on the degree $2$ truncation $A^{\le 2}=A/\bigoplus_{i>2}A^i$ 
of our \cdga.

Consider the algebraic map 
$\pi\colon A^1\times \g\to A^1\otimes \g$ given by 
$(a,x)\mapsto a\otimes x$.  Notice that $\pi$ restricts 
to a map $\pi\colon H^1(A)\times \g \to \F(A,\g)$. 
The set $\F^{(1)}(A,\g):=\pi(H^1(A)\times \g)$ is an irreducible, 
Zariski-closed subset of $\F(A,\g)$, which is equal to either $\{0\}$, 
or to the cone on $\PP(H^1(A)) \times \PP(\g)$.  We call its 
complement the {\em regular}\/ part of $\F(A,\g)$.

\subsection{Holonomy Lie algebra} 
\label{subsec:holo}

An alternate view of the parameter space of flat connections 
involves only Lie algebras.  Let us briefly review this approach, 
following the detailed study done in \cite{MPPS}.  

Let $A_i=\Hom_{\C} (A^i, \C)$ be the dual vector space.  
Let $\nabla \colon A_2 \to A_1\wedge A_1$ be the dual 
of the multiplication map 
$A^1\wedge A^1\to A^2$, and let $d_1\colon A_2\to A_1$ be 
the dual of the differential $d^1\colon A^1\to A^2$. 
By definition, the {\em holonomy Lie algebra}\/ of $(A,d)$ is 
the quotient of the free Lie algebra on the $\C$-vector space 
$A_1$ by the ideal generated by the image of $d_1 + \nabla$:
\begin{equation}
\label{eq:holo}
\h(A) = \Lie(A_1) / (\im(d_1 + \nabla)). 
\end{equation}

This construction is functorial.  Indeed, if $\varphi\colon A\to A'$ 
is a \cdga~map, then the linear map $\varphi_1=(\varphi^1)^*\colon A'_1\to A_1$ 
extends to Lie algebra map $\Lie(\varphi_1)\colon \Lie(A'_1)\to \Lie(A_1)$, 
which in turn induces a Lie algebra map $\h(\varphi)\colon \h(A')\to \h(A)$.

When $d=0$, the holonomy Lie algebra $\h(A)$ inherits a natural 
grading from the free Lie algebra, compatible with the Lie bracket. 
Thus, $\h(A)$ is a finitely presented, graded Lie algebra, 
with generators in degree $1$, and relations in degree $2$.  
In the particular case when $A$ is the cohomology algebra 
of a path-connected space $X$ with finite first Betti number, 
$\h(A)$ coincides with the classical holonomy Lie algebra $\h(X)$ 
of K.T.~Chen.

Given a finite set $\k=\{c_1,\dots ,c_k\}$, let us define the 
{\em reduced}\/ free Lie algebra on $\k$ as  
\begin{equation}
\label{eq:holo set}
\bLie(\k)=\Lie(c_1,\dots ,c_k)/\Big(\sum\nolimits_{\alpha=1}^{k} c_\alpha =0 \Big). 
\end{equation}
Clearly, $\bLie(\k)$ is a graded Lie algebra, isomorphic to the free Lie 
algebra of rank $k-1$.

\begin{example}
\label{ex:holo surf}
Consider the $k$-times punctured sphere, $S=\CP^1\setminus \{\text{$k$ points}\}$.  
Letting $\k=\{c_1,\dots,c_k\}$  
be the set of homology classes in $H_1(S,\C)$ represented by standardly 
oriented loops around the punctures, we readily see that $\h(S)=\bLie(\k)$. 
\end{example}

As before, let $\g$ be a finite-dimensional Lie algebra.
As noted in \cite{MPPS}, the canonical isomorphism
$\iota\colon A^1\otimes \g \isom  \Hom_{\C} (A_1,\g)$  
restricts to a functorial isomorphism 
\begin{equation}
\label{eq:hom holo}
\iota\colon \F(A,\g) \isom  \Hom_{\Lie} (\h(A), \g) . 
\end{equation}
Under this isomorphism, the subset $\F^{(1)}(A,\g)$ corresponds 
to the set $\Hom^1_{\Lie} (\h(A), \g)$ of Lie algebra morphisms 
whose image is at most $1$-dimensional. 

If $\varphi\colon A\to A'$ is a \cdga~map, we will let 
$\varphi^{!}\colon \Hom_{\Lie} (\h(A), \g) \to \Hom_{\Lie} (\h(A'), \g)$ 
denote the morphism of algebraic varieties induced by $\h(\varphi)$.

\subsection{The holonomy Lie algebra of a matroid}
\label{subsec:holo matroid}

Let $\M$ be a simple matroid, and let $A=A(\M)\otimes \C$ be the 
Orlik--Solomon algebra of $\M$ with coefficients in $\C$. 
As noted before, the $\C$-vector space $A^1$ has basis 
$\{e_u\}_{u\in \M}$.  Let $A_1$ be the dual vector space, 
with dual basis $\{a_u\}_{u\in \M}$.

By definition, the holonomy Lie algebra of the matroid, 
$\h(\M):=\h(A)$, is the quotient of the free Lie algebra on $A_1$ 
by the ideal generated by the image of the dual of the multiplication map, 
$A^1\wedge A^1\to A^2$.  Using the presentation \eqref{eq:os matroid}
for the algebra $A(\M)$, it is proved in \cite[\S 11]{PS04} that $\h(\M)$ 
has the following quadratic presentation: 
\begin{equation}
\label{eq:holo rel}
\h(\M)= \Lie(a_u,\:  u\in \M)\Big/ 
\Big( \sum_{v\in X}\,  [a_u, a_v ],\:\:  X\in L_2(\M), u\in X \Big)\Big. .
\end{equation}

Now let $\g$ be a finite-dimensional Lie algebra over $\C$. 
Once we identify $A^1\otimes \g \cong \g^{\A}$, a $\g$-valued $1$-form 
$\omega$ may be viewed as a vector with components  
$\omega_u \in \g$ indexed by the points $u\in \M$. 
By \eqref{eq:hom holo}, $\omega\in \F(A,\g)$ if and only if
\begin{equation}
\label{eq:flat omega}
\sum_{v\in X}\,  [\omega_u, \omega_v ]=0,\: \text{for all $X\in L_2(\M)$ 
and $u\in X$}.
\end{equation}

Let $\spn(\omega)$ be the linear subspace of $\g$ spanned 
by the set $\{\omega_u\}_ {u\in \M}$. Clearly, 
if $\dim \spn(\omega)\le 1$, then $\omega$ 
is a solution to the system of equations \eqref{eq:flat omega}; 
the set of such solutions is precisely $\F^{(1)}(A,\g)$. 
We call a solution $\omega$ 
{\em regular}\/ if $\dim \spn(\omega) \ge 2$. 

Noteworthy is the case when $\g=\sl_2$, a case studied in a more 
general context in \cite{MPPS}.  In this setting, an $\sl_2$-valued $1$-form 
$\omega=(\omega_u)_{u\in \M}$ is a solution to the system of 
equations \eqref{eq:flat omega} if and only if, for each $X\in L_2(\M)$, 
\begin{equation}
\label{eq:flat sl2}
\text{either}\  \sum_{v\in X}  \omega_v=0, \text{ or }
\dim\spn \{\omega_v\}_{v\in X} \le 1. 
\end{equation}

\subsection{Holonomy Lie algebra and multinets}
\label{subsec:holo multinet}

We may now characterize the key multinet axiom \eqref{mu3} 
from \S\ref{subsec:multinets} in terms of certain Lie algebra 
morphisms defined on the holonomy Lie algebra, as follows.

Let $\M$ be a matroid, endowed with a partition into non-empty blocks, 
$\M= \coprod_{\alpha \in \k} \M_{\alpha}$, and a multiplicity function,
$m\colon \M \to \N$. For each flat $X\in L_2(\M)$, define $\supp (X)$ 
as in \eqref{eq:supp}, and call the flat mono-colored if $\abs{\supp (X)}=1$, 
and multi-colored, otherwise.
Write $\k=\{c_1,\dots, c_k\}$, and let  $\bLie(\k)$ be the reduced 
free Lie algebra from \eqref{eq:holo set}.  
To these data, we associate a graded epimorphism of free Lie algebras,
\begin{equation}
\label{eq:phimap}
\xymatrix{\varphi \colon \Lie (a_u, u\in \M) \ar@{->>}[r]& \bLie(\k)}
\end{equation}
by sending $a_u$ to $m_u \cdot c_{\alpha}$, for each $u\in \M_{\alpha}$. 

\begin{prop}
\label{prop=multihol}
Given a matroid partition $\M= \coprod_{\alpha \in \k} \M_{\alpha}$ 
and a multiplicity function $m\colon \M \to \N$, the following conditions 
are equivalent:
\begin{enumerate}
\item \label{lie1}
The map $\varphi$ defined above factors through a graded 
Lie algebra epimorphism, $\varphi \colon \h(\M) \surj \bLie(\k)$.
\item  \label{lie2}
The integer $n_X:=\sum_{u\in \M_\alpha \cap X} m_u$ 
is independent of $\alpha$, for each multi-colored flat $X\in L_2(\M)$.
\end{enumerate}
\end{prop}

\begin{proof}
The morphism $\varphi$ factors through $\h(\M)$ if and only if 
equations \eqref{eq:flat omega} are satisfied by $\omega_u= \varphi (a_u)$. 
In turn, these equations are equivalent to 
\begin{equation}
\label{eq:lie alg eqs}
\Big[ \sum_{\alpha\in \k} \Big(\sum_{u\in X\cap \M_{\alpha}}m_u \Big) 
c_{\alpha}, c_{\beta} \Big]=0,
\end{equation}
for all $X\in L_2(\M)$ and $\beta \in \supp (X)$. Clearly, equations \eqref{eq:lie alg eqs} 
are always satisfied if $X$ is a mono-colored flat.

Now assume condition \eqref{lie1} holds.  As is well-known, commuting elements in a 
free Lie algebra must be dependent; see for instance \cite{MKS}. It follows that
$\sum_{\alpha\in \k} (\sum_{u\in X\cap \M_{\alpha}}m_u) c_{\alpha}$ belongs to
$\C \cdot c_{\beta} + \C \cdot (\sum_{\alpha \in \k} c_{\alpha})$, 
for all $\beta \in \supp (X)$. When $X$ is multi-colored, this constraint implies that
$\sum_{u\in X\cap \M_{\alpha}}m_u$ is independent of $\alpha$.

Conversely, assume \eqref{lie2} holds.  Equations \eqref{eq:lie alg eqs} for 
a multi-colored flat $X$ reduce then to $n_X \cdot [\sum_{\alpha \in \k} 
c_{\alpha}, c_{\beta}]=0$, and these are satisfied since 
$\sum_{\alpha \in \k} c_{\alpha}=0$ in $\bLie(\k)$. 
\end{proof}

\subsection{An evaluation map}
\label{subsec:ev}

Let $V$ be a finite-dimensional $\C$-vector space, and 
let $\k$ be a finite set with $k\ge 3$ elements. Inside the vector space 
$V^{\k}$, consider the linear subspace
\begin{equation}
\label{eq:hyper}
\HH^{\k} (V)=\big\{ x=(x_{\alpha}) \in V^{\k} \mid  
\sum_{\alpha \in \k} x_{\alpha}=0\big\}.
\end{equation}

Given a family of elements of a vector space, we may speak about 
its {\em rank}, that is, the dimension of the vector subspace generated 
by that family.  Inside $\HH^{\k} (V)$, we define the {\em regular part}\/ 
to be the set
\begin{equation}
\label{eq:hyper-reg}
\HH^{\k}_{\reg}(V)=\{ x \in \HH^{\k} (V) \mid \rank (x)>1 \}.
\end{equation}

Let us view an element $x \in V^{\k}$ as a map, $x\colon \k \to V$. 
Given a matroid $\M$, let us denote the induced map, $\k^{\M} \to V^{\M}$, 
by $\ev_{\hdot}(x)$.  For a fixed element $\tau \in \k^{\M}$, we obtain in 
this way a linear ``evaluation" map
\begin{equation}
\label{eq:ev}
\ev_{\tau} \colon V^{\k} \to V^{\M}, \quad  \ev_{\tau}(x)_u = x_{\tau (u)},  
\  \text{for $u\in \M$}.
\end{equation}

We will use this simple construction in the case when $V$ is a Lie algebra $\g$ and
$\tau$ is a non-constant, special $\k$-cocycle on $\M$.

\begin{prop}
\label{prop=liftev}
Let $\M$ be a matroid and let $\g$ be a finite-dimensional, complex Lie algebra.
For every $\tau \in Z'_{\k}(\M)\setminus B_{\k}(\M)$, the map $\ev_{\tau}$ induces a
linear embedding, 
\[
\ev_{\tau}\colon \HH^{\k}(\g) \inj \F (A(\M)\otimes \C, \g). 
\]
Moreover, $\ev_{\tau}$ is rank-preserving, and so the regular parts are preserved. 
\end{prop}

\begin{proof}
We first check that $\ev_{\tau}(x)\in \g^{\M}$ satisfies the flatness conditions
\eqref{eq:flat omega}, for $x\in \HH^{\k}(\g)$, where $\omega_u= x_{\tau (u)}$,
for $u\in \M$. If $\tau$ is constant on $X\in L_2(\M)$, this is clear. Otherwise,
$\tau \colon X \to \k$ is a bijection, hence the system \eqref{eq:flat omega} becomes
$[\sum_{\alpha \in \k} x_{\alpha}, x_{\tau (u)}]=0$, for all $u\in X$, and
we are done, since $x\in \HH^{\k}(\g)$.

We also know that $\tau\colon \M \to \k$ is surjective, since $\tau \not\in B_{\k}(\M)$.
This implies that $\rank (\ev_{\tau}(x))= \rank (x)$ for all $x\in \g^{\k}$. 
In particular, $\ev_{\tau}$ is injective.
\end{proof}

In the setup from Theorem \ref{teo=lambdaintro}\eqref{li2}, the above result may be
interpreted as an explicit way of lifting information on modular resonance to $\C$,
via flat connections.

\section{Complex resonance varieties and pencils}
\label{sect:res vars}

We now narrow our focus to realizable matroids, and recall the description 
of the (degree $1$, depth $1$) complex resonance variety of an 
arrangement $\A$ in terms of multinets on subarrangements of $\A$.
As an application of our techniques, we prove
Theorem \ref{thm:essintro} and Theorem \ref{thm:main1}\eqref{a7} 
from the Introduction.

\subsection{Resonance varieties of arrangements}
\label{subsec:res arr}

Let $\A$ be a hyperplane arrangement in $\C^{\ell}$, and let 
$A=H^*(M(\A),\C)$ be its Orlik--Solomon algebra over $\C$.  
The (first) resonance variety of the arrangement, 
$\RR_1(\A):=\RR_1(A)$, is a closed algebraic subset 
of the affine space $H^1(M(\A),\C)=\C^{\A}$. 
Since the slicing operation described in \S\ref{subsec:arrs} 
does not change $\RR_1(\A)$, we may assume without loss  
of generality that $\ell=3$.

The basic structure of the (complex) resonance varieties of 
arrangements is explained in the following theorem, which 
summarizes work of Cohen--Suciu \cite{CS99} and 
Libgober--Yuzvinsky \cite{LY}.  (We refer to \cite{DPS} 
for a more general context where such a statement holds.) 

\begin{theorem}
\label{thm:tcone}
All irreducible components 
of the resonance variety $\RR_1(\A)$ are linear subspaces, 
intersecting pairwise only at $0$. Moreover, the positive-dimensional 
components have dimension at least two, and the cup-product map 
$A^1\wedge A^1 \to A^2$ vanishes identically on each such component. 
\end{theorem}

We will also need a basic result from Arapura 
theory \cite{Ar} (see also \cite{DPS}), a result which adds 
geometric meaning to the aforementioned properties of  
$\RR_1(\A)$. Let $S$ denote $\CP^1$ with at least $3$ points removed. 
A map $f\colon M(\A) \to S$ is said to be {\em admissible}\/ 
if $f$ is a regular, surjective map with connected generic fiber. 

\begin{theorem}
\label{thm:arapura}
The correspondence $f\leadsto f^*(H^1(S,\C))$ gives a bijection between
the set of admissible maps (up to reparametrization at the target) and
the set of positive-dimensional components of $\RR_1(\A)$.
\end{theorem}

Most important for our purposes are the {\em essential}\/ 
components of $\RR_1(\A)$, i.e., those irreducible components 
which do not lie in any coordinate subspace of $\C^{\A}$. 
We will give a complete description of these 
components in the next subsection. 

\subsection{Pencils and multinets}
\label{subsec:pen multi}

As shown by Falk and Yuzvinsky in \cite{FY}, the essential components of 
$\RR_1(\A)$ can be described in terms of pencils arising from multinets on 
$L_{\le 2}(\A)$.  An alternate description of $\RR_1(\M,\C)$, valid for arbitrary 
simple matroids $\M$, was given in \cite{MB}.  We will follow here the approach 
from \cite{FY}.

Suppose we have a $k$-multinet $\NN$ on $\A$, 
with parts $\A_1,\dots, \A_k$ and multiplicity vector $m$.  
Let $Q(\A)=\prod_{H\in \A} f_H$ be a defining polynomial 
for $\A$, and set 
\begin{equation}
\label{eq:qalpha}
Q_{\alpha}=\prod_{H\in\A_{\alpha}}f_H^{m_H}.
\end{equation}

The polynomials  $Q_1,\dots,Q_k$ define a pencil of 
degree $d$ in $\CP^1$, having $k$ completely reducible fibers that correspond 
to $\A_1,\dots, \A_k$. For each $\alpha>2$, we may write 
$Q_{\alpha}$ as a linear combination $a_{\alpha} Q_1+b_{\alpha}Q_2$. 
In this way, we obtain a $k$-element subset 
\begin{equation}
\label{eq:d set}
D=\set{(0:-1), (1:0), (b_3:-a_3), \dots , (b_k:-a_k)} \subset \CP^1.
\end{equation}

Consider now the arrangement $\A'$ in $\C^2$ defined by the  
polynomial $Q(\A')=g_1\cdots g_k$, where 
$g_{\alpha}(z_1,z_2)= a_{\alpha} z_1+b_{\alpha}z_2$.  Projectivizing 
gives a canonical projection $\pi\colon M(\A') \to  S:=\CP^1\setminus D$. 
Setting $\psi(x)=(Q_1(x), Q_2(x))$ gives a regular map $\psi\colon M(\A) \to M(\A')$. 
It is now readily verified that the regular map 
\begin{equation}
\label{eq:pen}
f_{\NN}= \pi \circ \psi\colon M(\A) \to S
\end{equation}
is admissible. Hence, the linear subspace 
$f_{\NN}^*(H^1(S,\C))\subset H^1(M(\A),\C)$ 
is a component of $\RR_1(\A)$.  Moreover, this subspace has 
dimension $k-1$, and does not lie in any coordinate subspace. 
Conversely, as shown in \cite[Theorem 2.5]{FY}, every essential 
component of $\RR_1(\A)$ can be realized as $f_{\NN}^*(H^1(S,\C))$, 
for some multinet $\NN$ on $\A$.

\subsection{An induced homomorphism}
\label{subsec:induced}
To describe the above subspace explicitly, and for further purposes, 
we need to compute the homomorphism induced in homology 
by the map $f_{\NN}$.  To that end, let 
$\gamma_1,\dots ,\gamma_k$ be compatibly oriented, 
simple closed curves on $S=\CP^1\setminus D$, going around 
the points of $D$, so that $H_1(S,\Z)$ is generated by the homology 
classes $c_{\alpha}=[\gamma_{\alpha}]$, subject to the single relation 
$\sum_{\alpha=1}^k c_{\alpha}=0$. 

Recall that the cohomology ring $H^{\bullet}(M(\A),\Z)$ 
is generated by the degree $1$ classes $\{e_H\}_{H\in \A}$ 
dual to the meridians about the hyperplanes of $\A$.  
We shall abuse notation, and denote by the same symbol 
the image of $e_H$ in $H^1(M(\A),\C)$.  As is well-known, 
$e_H$ is the de~Rham cohomology class of the logarithmic $1$-form 
$\frac{1}{2\pi \ii}\, d\log f_H$ on $M(\A)$.

For each index $\alpha\in [k]$, set
\begin{equation}
\label{eq=defu}
u_{\alpha} :=\sum_{H\in \A_{\alpha}} m_H e_H \in H^1(M(\A),\C) .
\end{equation}

\begin{lemma}
\label{lem:pen h1}
The induced homomorphism $(f_{\NN})_* \colon H_1(M(\A),\Z) \to H_1(S,\Z)$ 
is given by 
\begin{equation*}
\label{eq:multi hom}
(f_{\NN})_*(a_H) = m_H c_{\alpha}, \quad\text{for  $H\in \A_{\alpha}$}.
\end{equation*}
In other words, $(f_{\NN})_*$ is the $\Z$-form of the homomorphism $\varphi$
associated to $\NN$ that appears in Proposition \ref{prop=multihol}. 
\end{lemma}

\begin{proof}
Given the construction of $f_{\NN}$ from \eqref{eq:pen}, 
it is plainly enough to check that the dual homomorphism, 
$\psi^* \colon H^1(M(\A'),\C) \to H^1(M(\A),\C)$, sends the de~Rham 
cohomology class of $d\log g_{\alpha}$ to $u_{\alpha}$. An easy 
calculation shows that 
\begin{equation}
\label{eq:dlog}
\psi^*(d\log g_{\alpha})=d\log Q_{\alpha}= 
\sum_{H\in \A_{\alpha}} m_H d\log f_H,
\end{equation} 
and the claim follows.
\end{proof}

Taking the transpose of $(f_{\NN})_*$ and using linear algebra, we obtain 
the following immediate corollary.

\begin{corollary}[\cite{FY}]
\label{cor:res comp}
Let $\NN$ be a $k$-multinet on an arrangement $\A$, and let 
$f_{\NN}\colon M(\A)\to S=\CP^1 \setminus \{\text{$k$ points}\}$ 
be the associated admissible map. Then the pull-back
$f_{\NN}^*(H^1(S,\C))$ is the linear subspace of 
$H^1(M(\A),\C)$ spanned by the vectors $u_2-u_1,\dots , u_k-u_1$, 
where $u_{\alpha} =\sum_{H\in \A_{\alpha}} m_H e_H$.
\end{corollary}

\subsection{Mapping multinets to resonance components}
\label{subsec:multinet map}
Let $\E(\A)$ be the set of essential components of the resonance 
variety $\RR_1(\A)$. The preceding discussion allows us to define a map 
\begin{equation}
\label{eq:multi map}
\Psi \colon \{ \text{multinets on $\A$} \} \to \E(\A), \quad 
\NN\mapsto f_{\NN}^*(H^1(S,\C)).  
\end{equation}
This map sends $k$-multinets to essential, $(k-1)$-dimensional 
components of $\RR_1(\A)$. 
By the above-mentioned result of Falk and Yuzvinsky, the map 
$\Psi$ is surjective.  The next lemma describes the fibers of this map. 

\begin{lemma}
\label{lem=fibres}
The surjective map $\Psi$ defined in \eqref{eq:multi map} is constant 
on the orbits of the natural $\Sigma_k$-action on $k$-multinets.  
Moreover, the fibers of $\Psi$ coincide with those orbits.
\end{lemma}

\begin{proof}
The first claim is 
an immediate consequence of the description of the action of $\Sigma_k$ 
on $k$-multinets, given in \S\ref{subsec:multinets}, coupled with the 
construction of  $\Psi (\NN)$. 

Suppose now that $\NN$ is a $k$-multinet, and $\Psi (\NN)=\Psi (\NN')$, 
for some multinet $\NN'$. As noted before, $\dim \Psi (\NN)=k-1$;  
hence, $\NN'$ is also a $k$-multinet. Let $f_{\NN}$ and $f_{\NN'}$ 
be the corresponding admissible maps from $M(\A)$ to 
$S=\CP^1 \setminus \{\text{$k$ points}\}$. Since 
$f_{\NN}^*(H^1(S,\C))=f_{\NN'}^*(H^1(S,\C))$, Arapura theory 
implies that $f_{\NN}$ and $f_{\NN'}$ differ by an automorphism 
of the curve $S$. 

In turn, this automorphism extends to an automorphism of $\CP^1$, 
inducing a permutation 
$g\in \Sigma_k$ of the $k$ points. Hence, the automorphism induced on
$H_1(S, \Z)$ sends $c_{\alpha}$ to $c_{g\alpha}$, for each $\alpha \in [k]$.  
Using Lemma \ref{lem:pen h1}, we conclude that $\NN$ and
$\NN'$ are conjugate under the action of $g$. 
\end{proof}

More generally, every positive-dimensional component 
$P$ of $\RR_1(\A)$ may be described in terms of multinets. 
Indeed, denote by $\proj_H \colon \C^{\A} \to \C$ 
the coordinate projections, and consider the subarrangement $\B\subseteq \A$ 
consisting of those hyperplanes $H$ for which $\proj_H \colon P \to \C$
is non-zero. It is easy to check that the subspace $P\subseteq \C^{\B}$ 
belongs to $\E(\B)$.
Hence, there is a multinet $\NN$ on $\B$ such that  $P=\Psi (\NN)$. Denoting by
$f_P \colon M(\A)\to S$ the regular map given by the restriction of the admissible map 
$f_{\NN} \colon M(\B)\to S$ to the complement of $\A$, 
this means that $P=f_{P}^*(H^1(S,\C))$.

\begin{corollary}
\label{cor=complrepr}
When $P$ runs through the set of positive-dimensional irreducible 
components of $\RR_1(\A)$, the regular maps $f_P$ constructed 
above form a complete set of representatives for the admissible 
maps on  $M(\A)$, modulo reparametrization at the target. 
\end{corollary}

\begin{proof}
By Theorem \ref{thm:arapura}, we only need to check that each regular 
map $f_P \colon M(\A)\to S$ is admissible. Since $f_{\NN} \colon M(\B)\to S$ 
is admissible, it is easy to infer that $f_P$ is non-constant, 
with connected generic fiber. Therefore, it is enough to prove 
that $f_P$ is surjective. In turn, this is a consequence of the fact that 
$P=f_{P}^*(H^1(S,\C))$ is a component of $\RR_1(\A)$. 

To prove this last claim, let us denote by $S'$ 
the image of $f_P$. Then $S'$ is obtained from $S$ by 
removing a finite set of points, and $f_P=j\circ f'$, where 
$j \colon S' \to S$ is the inclusion and $f' \colon M(\A) \to S'$ is 
the corestriction of $f$.  Clearly, the map $f'$ is admissible.  
Hence, both $j^* \colon H^1(S,\C) \to H^1(S',\C)$ 
and $f'^* \colon H^1(S',\C) \to H^1(M(\A), \C)$
are injections. On the other hand, $f'^*(H^1(S',\C))$ is a 
component of $\RR_1(\A)$, by Theorem \ref{thm:arapura}. 
Therefore, $f^*$ and $f'^*$ have the same image.  Hence,  
$H^1(S,\C)$ and $H^1(S', \C)$ have the same dimension,  
and so $S'=S$.  This proves the claim, and we are done.
\end{proof}

\subsection{Counting essential components}
\label{ssec=55}

Recall that $\E_k(\A)$ denotes the set of essential 
components of $\RR_1(\A)$ arising from $k$-nets on $\A$.  
As mentioned previously, this set is empty for $k\ge 5$.  

\begin{proof}[Proof of Theorem \ref{thm:essintro}]
Let $k=3$ or $4$, and let $\k$ be the corresponding Galois field, $\FF_k$. 
By Lemma \ref{lem=fibres} and Theorem \ref{teo=lambdaintro}, we have that 
\begin{equation}
\label{eq:eka}
\abs{\E_k(\A)}= \frac{1}{k!} \abs{Z'_{\k} (\A) \setminus B_{\k} (\A)} \le
\frac{1}{k!} \abs{Z_{\k} (\A) \setminus B_{\k} (\A)} .
\end{equation}
Clearly, 
\begin{equation}
\label{eq:zka}
\abs{Z_{\k} (\A) \setminus B_{\k} (\A)}= \abs{\k} \cdot 
\abs{Z_{\k} (\A)/ B_{\k} (\A) \setminus \{ 0\} }
= k\cdot (k^{\beta_{\k}(\A)}-1) .
\end{equation}
Inequality \eqref{eq=essboundintro} now follows at once.

Next, assume that both $\E_3(\A)$ 
and $\E_4(\A)$ are non-empty.   
From Proposition \ref{prop:betapknet} we then infer that 
$\beta_2 (\A)=0$ and $\beta_2 (\A) \ge 2$, a contradiction.
This completes the proof.
\end{proof}

\begin{proof}[Proof of Theorem \ref{thm:main1}\eqref{a7}]
Suppose $L_2(\A)$ has no flats of multiplicity properly divisible by $3$. 
By Lemma \ref{lem:lambda}, $Z'_{\FF_3} (\A) = Z_{\FF_3} (\A)$.
The above proof then shows that $\abs{\E_3(\A)}= (3^{\beta_{3}(\A)}-1)/2$, 
and we are done. 
\end{proof}

\section{Evaluation maps and multinets}
\label{sect:flat res net}

We extend in this section our construction of evaluation maps, from nets to
multinets. For realizable matroids, we exploit evaluation maps in two directions.
First, we construct the inverse of the bijection $\lambda_{\k}$ from Lemma
\ref{lema=lambdabij} by using the variety of $\sl_2 (\C)$-valued flat
connections on the Orlik--Solomon algebra. Second, we provide in 
Theorem \ref{thm=freg3} (Theorem \ref{teo=modtoflatintro} from 
the Introduction) a combinatorial 
condition insuring that this variety can be reconstructed explicitly from 
information on modular resonance.

\subsection{Flat connections coming from multinets}
\label{ssec=60}

We first extend the construction from Proposition \ref{prop=liftev} to
a broader context. Let $\k$ be a finite set with $k\ge 3$ elements
and let $\M$ be a simple matroid. For a $k$-multinet $\NN$ on $\M$, 
denote by $\varphi_{\NN} \colon \h (\M) \surj \bLie(\k)$
the epimorphism of graded Lie algebras constructed in 
Proposition \ref{prop=multihol}.  

Let $\g$ be a finite-dimensional complex Lie algebra, and let  
\begin{equation}
\label{eq:shriek}
\xymatrix{\varphi_{\NN}^! \colon \Hom_{\Lie}(\bLie (\k), \g) \ar[r]&
\Hom_{\Lie}(\h(\M), \g)}
\end{equation}
be the induced map on $\Hom$-sets.   Using \eqref{eq:holo set} 
and \eqref{eq:hom holo}, we may identify $\Hom_{\Lie}(\bLie (\k), \g)$ 
with $\HH^{\k}(\g)$ and $ \Hom_{\Lie}(\h(\M), \g)$ with $\F (A(\M)\otimes \C, \g)$.  
Let 
\begin{equation}
\label{eq:bigev}
\xymatrix{\Ev_{\NN} \colon \HH^{\k}(\g)  \ar[r]&\F (A(\M)\otimes \C, \g)}
\end{equation}
be the map corresponding to $\varphi_{\NN}^!$ under these identifications.
Finally, for each $\alpha\in \k$, let 
$\proj_{\alpha} \colon \HH^{\k}(\g) \to \g$ be the restriction to $\HH^{\k}(\g)$ of the 
$\alpha$-coordinate projection $\g^{\k}\to \g$. 

\begin{prop}
\label{prop=multiflat}
With notation as above, the following hold.
\begin{enumerate}
\item \label{pmf1}
The evaluation map $\Ev_{\NN}$ is a rank-preserving, linear embedding. 
\item \label{pmf2}
For any $u\in \M$, there is $\alpha\in \k$ such that the restriction of
$\proj_u \otimes \id_{\g} \colon A^1 (\M) \otimes \g \to \g$ to $\HH^{\k}(\g)$ 
via $\Ev_{\NN}$ belongs to $\C^* \cdot \proj_{\alpha}$. 
\item \label{pmf3}
If $\NN$ is a $k$-net, then $\Ev_{\NN}= \ev_{\tau}$, where $\tau= \lambda_{\k} (\NN)$.
\end{enumerate}
\end{prop}

\begin{proof}
 \eqref{pmf1}  By construction, the map $\Ev_{\NN}$ is linear. 
Since $\varphi_{\NN}$ is surjective, the map $\varphi_{\NN}^!$ is  rank-preserving; 
hence, $\Ev_{\NN}$ is also rank-preserving, and thus, injective.

\eqref{pmf2} Using the underlying partition of $\NN$, we find that $u\in \M_{\alpha}$,
for a unique $\alpha \in \k$. By construction of $\varphi_{\NN}$, we have that 
$\Ev_{\NN}^* (\proj_u \otimes \id_{\g})= m_u \cdot \proj_{\alpha}$.

 \eqref{pmf3} Let $\NN$ be a $k$-net. For $x\in \HH^{\k}(\g)$ and $u\in \M_{\alpha}$,
we have that $\Ev_{\NN}(x)_u= m_u x_{\alpha}$, by construction.  Moreover, $m_u=1$,
since $\NN$ is a reduced multinet. On the other hand, 
$\ev_{\tau}(x)_u=  x_{\tau (u)}$, by \eqref{eq:ev}, and $\tau (u)=\alpha$,
by \eqref{eq=lambda}. This completes the proof.
\end{proof}

\subsection{Flat connections and complex resonance varieties}
\label{subsec:flat res}

In the case of realizable matroids, a crucial ingredient in our approach 
is a general result relating resonance and flat connections, based 
on the detailed study done in \cite{MPPS}. 

To start with, let $A$ be a graded, graded-commutative algebra over $\C$. 
Recall we assume $A$ is connected and $A^1$ is finite-dimensional. 
Given a linear subspace $P\subset A^1$, define a connected 
sub-algebra $A_{P} \subset A^{\le 2}$ by setting $A^1_{P}=P$ and 
$A^2_{P}=A^2$, and then restricting the multiplication map accordingly. 

Now let $\g$ be a complex Lie algebra.  The following equality is then easily verified:
\begin{equation}
\label{eq:fap}
\F(A_{P},\g) = \F(A, \g)\cap (P\otimes \g).
\end{equation}
Thus, if $\g$ is finite-dimensional, then $\F(A_P, \g)$ is a Zariski-closed 
subset of $\F(A, \g)$.

\begin{theorem}
\label{lem:disjoint}
Suppose $\RR_1(A)= \bigcup_{P\in \cP} P$, where $\cP$ is a finite
collection of linear subspaces of $A^1$, intersecting pairwise only 
at $0$.  Then, for any finite-dimensional Lie algebra $\g$, the 
following hold:
\begin{enumerate}
\item  \label{r2}
 $\F(A_{P},\g)\cap \F(A_{P'},\g)=\{0\}$, for all distinct subspaces $P,P'\in \cP$.
 \\[-9pt]
\item \label{r1}
$\F(A,\g) \supseteq \F^{(1)}(A,\g) \cup \bigcup_{P\in \cP} \F(A_{P},\g)$.
\\[-9pt]
\item \label{r3}
If $\g=\sl_2$, then the above inclusion holds as an equality. 
\item \label{r4}
If $\g= \sl_2$ and all subspaces from $\cP$ are isotropic, then
$\F(A_{P},\g) = P\otimes \g$, for every $P \in \cP$.
\end{enumerate}
\end{theorem}

\begin{proof}
Claim \eqref{r2}  follows from our transversality hypothesis, while
claim \eqref{r1} is obvious, by the naturality property of flat connections.
Claim \eqref{r3} is proved in \cite[Proposition 5.3]{MPPS}. 
Here, the assumption that $\g=\sl_2$ is crucial. In the proof of
Proposition 5.3 from \cite{MPPS} it is also shown that $P\otimes \g \subseteq \F(A,\g)$,
when $P$ is isotropic. Claim \eqref{r4} follows then from \eqref{eq:fap}.
\end{proof}

\subsection{From evaluation maps to multinets}
\label{subsec:pencil}

We now return to the situation  when $A=H^*(M(\A),\C)$  is 
the Orlik--Solomon algebra of an arrangement $\A$.  
In view of Theorem \ref{thm:tcone}, 
all the hypotheses of Theorem \ref{lem:disjoint} are satisfied 
in this case.

Guided by Proposition \ref{prop=multiflat}, we take $\g=\sl_2 (\C)$ and 
define the evaluation space $\cE_k (\A)$ to be the set of all maps,
$e\colon \HH^{\k}(\g) \to \F (A(\A)\otimes \C, \g)$, satisfying properties 
\eqref{pmf1} and \eqref{pmf2} from that proposition, and having the 
property that $\im (e)$ is an irreducible component of $\F (A(\A)\otimes \C, \g)$. 

\begin{lemma}
\label{lem:submat}
Let $A$ be the complex Orlik--Solomon algebra of an arrangement $\A$, 
and let $\RR_1(\A) = \bigcup_{P\in \cP} P$ be the irreducible decomposition of its 
resonance variety.  The following then hold.
\begin{enumerate}
\item \label{731}
The irreducible decomposition of the variety $\F(A,\sl_2)$
is given by 
\[
\F(A,\sl_2) = \F^{(1)}(A,\sl_2) \cup \bigcup_{P\in \cP} \F(A_{P},\sl_2).
\] 
\item \label{732}
For every $k$-multinet $\NN$ on $\A$, 
\[
\im (\Ev_{\NN})= P\otimes \sl_2= \F (A_P, \sl_2),
\]
where $P=\Psi (\NN)$.
\end{enumerate}
\end{lemma}

\begin{proof}
\eqref{731} By Theorem \ref{thm:arapura}, every non-zero 
subspace $P\in \cP$ is of the form $P=f^*(H^1(S, \C))$, for some admissible map
$f \colon M(\A) \to S:= \CP^1 \setminus \{\text{$k$ points}\}$. Theorem 7.4 from \cite{MPPS}
gives the irreducible decomposition of $\F (A, \sl_2)$, with $\F (A_P, \sl_2)$
replaced by $ f^! (\F (H^{\hdot}(S, \C), \sl_2))$. 
But $\F (H^{\hdot}(S, \C), \sl_2)= H^{1}(S, \C)\otimes \sl_2$, since $H^{2}(S, \C)=0$.
Hence, by Theorem \ref{lem:disjoint}\eqref{r4}, 
\begin{equation}
\label{eq:fhdot}
f^! (\F (H^{\hdot}(S, \C), \sl_2))= P\otimes \sl_2= \F (A_P, \sl_2), 
\end{equation}
and this proves our claim.

\eqref{732} Let $f_{\NN}\colon M(\A) \to S$  
be the admissible map associated to the multinet $\NN$. By Lemma \ref{lem:pen h1}, 
the map $\varphi^1_{\NN} \colon \h^1 (\A) \to \bLie^1(\k)$ may be identified with 
$(f_{\NN})_* \otimes \C$.  It follows that $\im (\Ev_{\NN})= \Psi (\NN)\otimes \sl_2$, 
by the construction \eqref{eq:multi map} of $\Psi$.
In view of \eqref{eq:fhdot}, we are done. 
\end{proof}

In view of the above lemma, 
the construction from Proposition \ref{prop=multiflat} 
gives a correspondence, 
\begin{equation}
\label{eq:bigev corr}
\xymatrix{\Ev \colon \{ \text{$k$-multinets on $\A$}\} \ar[r]& \cE_k(\A)} .
\end{equation}

We now define another function,
\begin{equation}
\label{eq:nmap}
\xymatrix{\Net \colon \cE_k(\A) \ar[r]& \{ \text{$k$-multinets on $\A$}\}/ \Sigma_k} ,
\end{equation}
as follows. Given a map $e\colon \HH^{\k}(\sl_2) \to \F (A(\A)\otimes \C, \sl_2)$ 
belonging to $\cE_k(\A)$, the variety $\im (e)$ cannot be the irreducible component 
$\F^{(1)} (A, \sl_2)$. Indeed, $\HH^{\k}(\sl_2)$ contains a regular element $x$, 
since $\dim \sl_2 \ge 2$. Since $e$ is a rank-preserving linear map  
we must have $\rank (e(x))\ge 2$, and therefore $e(x)$ is a regular 
flat connection. Hence, by Lemma \ref{lem:submat}\eqref{731} 
and Theorem  \ref{lem:disjoint}\eqref{r4}, $\im (e)= P\otimes \sl_2$,
for a unique $0\ne P\in \cP$. 

We claim that $P$ must be an essential component of $\RR_1(\A)$. 
For otherwise we could find a hyperplane $H\in \A$ such that
$\proj_H \colon A^1 \to \C$ vanishes on $P$.  But now recall $e$ 
satisfies property \eqref{pmf2} from Proposition \ref{prop=multiflat}; hence, 
$\proj_{\alpha}$ must vanish on $ \HH^{\k}(\sl_2)$, for some $\alpha \in \k$,
a contradiction. By Lemma \ref{lem=fibres}, then, $P=\Psi (\NN)$, for some $k'$-multinet 
$\NN$ on $\A$, uniquely determined up to the natural $\Sigma_{k'}$-action.
Moreover, 
\[
3(k'-1)= \dim P\otimes \sl_2= \dim  \HH^{\k}(\sl_2)= 3(k-1),
\]
and thus $k'=k$. We then define $\Net(e)$ to be the class modulo $\Sigma_k$ of the 
$k$-multinet $\NN$.

\begin{corollary}
\label{coro=lambdasurj}
The composition $\Net\circ \Ev$ is the canonical projection that associates to a
$k$-multinet $\NN$ on $\A$ its $\Sigma_k$-orbit.
\end{corollary}

\begin{proof}
By Lemma \ref{lem:submat}, we have that $\im (\Ev_{\NN})= \Psi (\NN)\otimes \sl_2$.
By construction, $\Net(\Ev_{\NN})$ is the  $\Sigma_k$-orbit of $\NN$.
\end{proof}

\begin{remark}
\label{rem=lambdainv}
Returning to the bijection from Lemma \ref{lema=lambdabij}, 
let us take a special $\k$-cocycle $\tau \in Z'_{\k}(\A)\setminus B_{\k}(\A)$, 
and set $\NN= \lambda_{\k}^{-1} (\tau)$. By Proposition \ref{prop=multiflat}\eqref{pmf3}
and Corollary \ref{coro=lambdasurj}, $\ev_{\tau}\in \cE_k (\A)$ and $\Net(\ev_{\tau})$ 
is the $\Sigma_k$-orbit of the $k$-net $\NN$. 
This shows that, for realizable matroids, the inverse of the modular construction 
$\lambda_{\k}$ may be described in terms of $\sl_2 (\C)$-valued flat
connections on the Orlik--Solomon algebra.
\end{remark}

\subsection{Flat connections from special cocycles}
\label{subsect:flat 3net}

Let $\B \subseteq \A$ be a subarrangement. Denote the associated monomorphism 
between complex Orlik--Solomon algebras by $\psi \colon B\inj A$. This map 
in turn induces a rank-preserving inclusion, 
$\psi\otimes \id_{\g} \colon \F (B, \g) \inj \F (A, \g)$, for
any finite-dimensional complex Lie algebra $\g$. For 
$\tau\in Z'_{\k}(\B)\setminus B_{\k}(\B)$, denote by 
$\ev^{\B}_{\tau} \colon \HH^{\k} (\g) \to \F (A, \g)$ the linear, rank-preserving embedding
$(\psi\otimes \id_{\g}) \circ \ev_{\tau}$. 

\begin{theorem}
\label{thm=freg3}
Assume that all essential components of $\RR_1(\B)$ arise from nets on $\B$, 
for every subarrangement $\B \subseteq \A$. Then 
\begin{equation}
\label{eq:freg}
\F_{\reg}(A(\A)\otimes \C, \sl_2 (\C))= \bigcup_{\B, \tau} \; 
\ev^{\B}_{\tau} (\HH^{\k}_{\reg} (\sl_2 (\C)))\, ,
\end{equation}
where the union is taken over all $\B \subseteq \A$ 
and all $\tau\in Z'_{\k}(\B)\setminus B_{\k}(\B)$.
\end{theorem}

\begin{proof}
Let $\RR_1(A)= \bigcup_{P \in \cP} P$ be the decomposition 
of the complex resonance variety of $\A$ into (linear) irreducible 
components. For simplicity, write $\HH=\HH^{\k}(\sl_2)$.  
In view of Theorem \ref{lem:disjoint}, we only have to show that, for 
every non-zero component $P$ of $\RR_1(A)$, the set 
$(P\otimes \sl_2)_{\reg}$ is contained in the right-hand side 
of \eqref{eq:freg}.  As explained in \S\ref{subsec:multinet map}, 
the subspace $P$ belongs to $\E(\B)$, for some subarrangement 
$\B \subseteq \A$. Thus, we may replace $\A$ by $\B$, and 
reduce our proof to showing that, for any $P\in \E (\B)$, the set 
$(P\otimes \sl_2)_{\reg}$ is contained in $\ev_{\tau}(\HH_{\reg})$, 
for some $\tau\in Z'_{\k}(\B)\setminus B_{\k}(\B)$. 

Using our hypothesis, we infer that $P=\Psi (\NN)$, 
for some $k$-net $\NN$ on $\B$.  Set 
$\tau= \lambda_{\k} (\NN) \in Z'_{\k}(\B)\setminus B_{\k}(\B)$. 
It follows from Proposition \ref{prop=multiflat}\eqref{pmf3} and
Lemma \ref{lem:submat} that $\ev_{\tau}(\HH) =P\otimes \sl_2$. 
By construction, $\rank (\ev_{\tau}(x))= \rank (x)$, for all 
$x\in \HH$. Hence, $\ev_{\tau}(\HH_{\reg})= 
(P\otimes \sl_2)_{\reg}$.
\end{proof}

\subsection{Examples}
\label{subsec:ess disc}

We conclude this section with a couple of extended examples. 

\begin{example}
\label{ex=flatsharp}
Let $\A$ be the reflection arrangement of type 
${\rm B}_3$, defined by the polynomial 
$Q=z_1z_2z_3(z_1^2-z_2^2)(z_1^2-z_3^2)(z_2^2-z_3^2)$.
As shown in \cite{MP}, we have that $\beta_p(\A)=0$, for all $p$.
In particular, $\A$ supports no net, by \eqref{eq=essboundintro}. 
On the other hand, this arrangement admits the multinet from Figure \ref{fig:b3 arr}.
Thus, the hypothesis of Theorem \ref{thm=freg3} is violated in this example.

We claim that Theorem \ref{thm=freg3} does not hold 
for this arrangement.  To verify this claim, pick 
$x=(x_0,x_1,x_2)\in \HH^{\FF_3}_{\reg}(\sl_2)$ and define 
$\omega= \sum_{H\in \A} e_H \otimes \omega_ H \in A^1\otimes \sl_2$ by
\[
\omega_{z_1}=2 x_1,\,  \omega_{z_2}=2 x_2, \, \omega_{z_3}=2 x_0,\, 
\omega_{z_1 \pm z_2}=x_0, \,
\omega_{z_2 \pm z_3}=x_1, \, \omega_{z_1 \pm z_3}=x_2.
\] 
It is easy to check that $\omega \in \F_{\reg}(A, \sl_2)$.
Since clearly $x_i\ne 0$ for all $i$, we infer that
$\omega$ is supported on the whole arrangement $\A$.  
On the other hand, all elements from the 
right-hand side of \eqref{eq:freg} are supported on proper 
subarrangements of $\A$, by Lemma \ref{lema=lambdabij}.  
Thus, equality does not hold in \eqref{eq:freg}  in this case.  
\end{example}

\begin{example}
\label{ex:graphic}

Let $\Gamma$ be a finite simplicial graph, with vertex set 
$[\ell]$ and edge set $E$.  The corresponding (unsigned) 
graphic arrangement, $\A_{\Gamma}$, is the arrangement in 
$\C^{\ell}$ defined by the polynomial $Q=\prod_{(i,j)\in E} (z_i-z_j)$. 

For instance, if $\Gamma=K_{\ell}$ is the complete graph on 
$\ell$ vertices, then $A_{\Gamma}$ is the reflection arrangement 
of type $\operatorname{A}_{\ell-1}$. 
The Milnor fibrations of graphic arrangements were studied in \cite{MP}. 
Clearly, $\mult(\A_{\Gamma}) \subseteq \{3\}$, and so $\beta_p(\A_{\Gamma})=0$, 
unless $p=3$.  It turns out that $\beta_3(\A_{\Gamma})=0$ for all graphs 
$\Gamma$ except $\Gamma=K_3$ and $K_4$, in which case 
$\beta_3(\A_{\Gamma})=1$. 

Theorem \ref{thm:main0} was proved in \cite{MP} for the class 
of graphic arrangements.  
For such arrangements, the inequalities \eqref{eq=essboundintro} are sharp.  
Indeed, $\RR_1(\A_{\Gamma})$ 
has an essential component if and only if $\Gamma=K_3$ 
or $K_4$, in which case $\abs{\E(\A_{\Gamma})}= \abs{\E_3(\A_{\Gamma})}=1$; 
see \cite{SS, CS99}. Moreover, it follows  that Theorem \ref{thm=freg3}
holds for all graphic arrangements.

Finally, by \cite[Theorem A]{MP}, Conjecture \ref{conj:mf} holds in the strong form
\eqref{eq:delta arr}, for all (not necessarily unsigned) graphic arrangements.
\end{example}

\section{Characteristic varieties and the Milnor fibration}
\label{sect:cjl milnor}

In this section, topology comes to the fore.  Using the jump loci 
for homology in rank $1$ local systems, we prove implication 
\eqref{a5}$\Rightarrow$\eqref{a6} from Theorem \ref{thm:main1},
and we finish the proof of Theorem \ref{thm:2main1}.

\subsection{Characteristic varieties and finite abelian covers}
\label{subsec:cv}

Let $X$ be a connected, finite-type CW-complex.  Without loss 
of generality, we may assume $X$ has a single $0$-cell.  
Let $\pi=\pi_1(X,x_0)$ be the fundamental group of $X$, 
based at this $0$-cell.  
  
Let $\Hom(\pi,\C^*)$ be the affine 
algebraic group of $\C$-valued, multiplicative characters on $\pi$, 
which we will identify with $H^1(\pi,\C^*)=H^1(X,\C^*)$. 
The (degree $q$, depth $r$) {\em characteristic varieties}\/ of 
$X$ are the jump loci for homology with coefficients in rank-$1$ 
local systems on $X$:
\begin{equation}
\label{eq:cvs}
\VV^q_r(X)=\{\xi\in\Hom(\pi,\C^*)  \mid  
\dim_{\C} H_q(X,\C_\xi)\ge r\}.
\end{equation}

By construction, these loci are Zariski-closed subsets of 
the character group. Here is a simple example, that we will need 
later on. 

\begin{example}
\label{ex:cv surf}
Let $S=\CP^1\setminus \{\text{$k$ points}\}$.  Then 
$\VV^1_r(S)$ equals $H^1(S,\C^*)=(\C^*)^{k-1}$ if 
$1\le  r  \le k-2$, it equals $\{1\}$ if $r  = k-1$, and it is empty 
if $r \ge k$.
\end{example}

As is well-known, the geometry of the characteristic varieties 
controls the Betti numbers of regular, finite abelian covers of $X$.  
For instance, suppose that the deck-trans\-formation group is 
cyclic of order $n$, and fix an inclusion $\iota \colon \Z_n \inj \C^*$, 
by sending $1 \mapsto e^{2\pi \ii/n}$. With this choice, the 
epimorphism $\nu \colon \pi\surj \Z_n$ defining the cyclic cover $X^{\nu}$ 
yields a (torsion) character, $\rho=\iota\circ \nu\colon \pi \to \C^*$. 
We then have an isomorphism of $\C[\Z_n]$-modules,
\begin{equation}
\label{eq:equiv}
H_q(X^{\nu},\C) \cong H_q(X,\C) \oplus 
\bigoplus_{1<d\mid n} (\C[t]/\Phi_d(t))^{\depth(\rho^{n/d})}, 
\end{equation}
where $\depth(\xi):=\dim_{\C} H_q(X, \C_{\xi}) = 
\max\{r \mid \xi\in \VV^q_r(X)\}$.  For a quick proof of this classical 
formula (originally due to A.~Libgober, M.~Sakuma, and E.~Hironaka), 
we refer to \cite[Theorem 2.5]{DS13} or \cite[Theorem B.1]{Su14}.

As shown in \cite{PS-tams}, the exponents in formula \eqref{eq:equiv} 
coming from prime-power divisors can be estimated in terms of the 
corresponding Aomoto--Betti numbers. More precisely, suppose 
$n$ is divisible by $d=p^s$, for some prime $p$.  Composing the 
canonical projection $\Z_n \surj \Z_p$ with $\nu$ defines a 
cohomology class $\bar\nu\in H^1(X,\FF_p)$.  

\begin{theorem}[\cite{PS-tams}]
\label{thm:mod bound}
With notation as above, assume $H_*(X,\Z)$ is 
torsion-free.  Then  
\[
\dim_{\C} H_q(X, \C_{\rho^{n/d}}) \le 
\dim_{\FF_p} H^q( H^{\hdot}(X, \FF_p), \delta_{\bar\nu}).
\]
\end{theorem}

\subsection{Characteristic varieties of arrangements}
\label{subsec:cv arr}

Let $\A$ be a hyperplane arrangement  in $\C^{\ell}$.  Since the 
slicing operation described in \S\ref{subsec:arrs} does not affect the 
character torus, $H^1(M(\A),\C^*)=(\C^*)^{\A}$, or the degree $1$ 
characteristic varieties of the arrangement, $\VV_r(\A):=\VV^1_r(M(\A))$, 
we will assume from now on that $\ell=3$.

The varieties $\VV_r(\A)$ are closed algebraic 
subsets of the character torus.  
Since $M(\A)$ is a smooth, quasi-projective variety, 
a general result of Arapura \cite{Ar} insures that 
$\VV_r(\A)$ is, in fact, a finite union of translated 
subtori.  Moreover, as shown in \cite{CS99, LY}, and, 
in a broader context in \cite{DPS}, the tangent cone 
at the origin to $\VV_1(\A)$ coincides with the resonance 
variety $\RR_1(\A)$. 

More explicitly, consider the exponential map 
$\C\to \C^*$, and the coefficient homomorphism 
$\exp\colon H^1(M(\A),\C)\to H^1(M(\A),\C^*)$. 
Then, if $P\subset H^1(M(\A),\C)$ is one of the linear 
subspaces comprising $\RR_1(\A)$, its image under 
the exponential map, $\exp(P)\subset H^1(M(\A),\C^*)$, 
is one of the subtori comprising $\VV_1(\A)$. Moreover, 
this correspondence gives a bijection between the components 
of $\RR_1(\A)$ and the components of $\VV_1(\A)$ passing through 
the origin.  

Now recall from Arapura theory (\cite{Ar, DPS})
that each positive-dimensional component of $\RR_1(\A)$ 
is obtained by pullback along 
an admissible map $f\colon M(\A)\to S$, where $S=\CP^{1}\setminus 
\{\text{$k$ points}\}$ and $k\ge 3$. Thus, each positive-dimensional 
component of $\VV_1(\A)$ containing the origin is of the form 
$\exp(P)=f^*(H^1(S,\C^*))$, with $f$ admissible. 
In view of Example \ref{ex:cv surf},  
the subtorus $f^*(H^1(S,\C^*))$ is a  positive-dimensional 
component of $\VV_1(\A)$ through the origin that
lies inside $\VV_{k-2}(\A)$, for any admissible map $f$ as above. 

Next, let $\bar{\A}$ be the projectivized line arrangement in 
$\CP^{2}$, and let $U(\A)$ be its complement.  The Hopf 
fibration, $\pi\colon \C^{3} \setminus \set{0} \to \CP^{2}$ 
restricts to a trivializable bundle map, 
$\pi\colon M(\A)\to U(\A)$, with fiber $\C^*$. 
Therefore, $M(\A)\cong U(\A)\times \C^*$, 
and the character torus $H^1(M(\A),\C^*)$ splits as 
$H^1(U(\A),\C^*)\times \C^*$.  Under this splitting, 
the characteristic varieties $\VV^1_r(M(\A))$ get 
identified with the varieties $\VV^1_r(U(\A))$ lying 
in the first factor. 

\subsection{The homology of the Milnor fiber}
\label{subsec:milnor}

Let $Q=Q(\A)$ be a defining polynomial for our arrangement. 
The restriction of $Q$ to the complement defines 
the Milnor fibration, $Q\colon M(\A) \to \C^*$, 
whose typical fiber, $F(\A)=Q^{-1}(1)$, is the Milnor fiber of 
the arrangement. 

The map $\pi\colon M(\A)\to U(\A)$ 
restricts to a regular, $\Z_n$-cover 
$\pi \colon F(\A) \to U(\A)$, where $n=\abs{\A}$. 
As shown in \cite{CS95} (see \cite[Theorem 4.10]{Su14} for 
full details), this cover is classified by the ``diagonal" epimorphism 
\begin{equation}
\label{eq:nu milnor}
\nu\colon H_1(U(\A),\Z)\surj \Z_n, \quad 
\nu(\pi_*(a_H) )= 1 \bmod n. 
\end{equation}

For each divisor $d$ of $n$, let $\rho_d\colon H_1(M(\A),\Z)\to \C^*$ 
be the character defined by $\rho_d(a_H)= e^{2\pi \ii /d}$. 
Using formula \eqref{eq:equiv}, we conclude that 
\begin{equation}
\label{eq:h1milnor}
H_1(F(\A),\C) = \C^{n-1} \oplus \bigoplus_{1<d\mid n} 
(\C[t]/\Phi_d(t))^{\e_d(\A)},
\end{equation}
as modules over $\C [\Z_n]$, where $e_d(\A)=\depth (\rho_d)$. Furthermore, 
applying Theorem \ref{thm:mod bound}, we obtain the ``modular upper bound"
\begin{equation}
\label{eq:bound bis}
\e_{p^s} (\A) \le \beta_p(\A), 
\end{equation}
valid for all primes $p$ and integers $s\ge 1$. 

\subsection{Milnor fibration and multinets}
\label{subsec:milnor multi}

A key task now is to find suitable lower bounds for the exponents 
appearing in formula \eqref{eq:h1milnor}.  The next result provides 
such bounds, in the presence of reduced multinets on the arrangement.

\begin{theorem}
\label{thm:be3}
Suppose that an arrangement $\A$ admits a reduced $k$-multinet. Letting 
$f\colon M(\A)\to S$ denote the associated admissible map, the following hold. 
\begin{enumerate}
\item \label{bb1}
The character $\rho_k$ belongs to $f^*(H^1(S,\C^*))$, and  
$\e_k(\A)\ge k-2$. 
\item \label{bb2}
If $k=p^s$, then $\rho_{p^r}\in f^*(H^1(S,\C^*))$
and $\e_{p^r}(\A)\ge k-2$, for all $1\le r\le s$.
\end{enumerate}
\end{theorem}

\begin{proof}
First let $\NN$ be an arbitrary $k$-multinet on $\A$, with 
parts $\A_{\alpha}$ and multiplicity function $m$, and let 
$f=f_{\NN}\colon M(\A)\to S$. It follows from 
Lemma \ref{lem:pen h1} that the induced morphism between 
character groups, $f^*\colon H^1(S, \C^*) \to  H^1(M(\A), \C^*)$, 
takes the character $\rho$ given by 
$\rho(c_{\alpha})=\zeta_{\alpha}$, where $\zeta_1\cdots \zeta_k=1$, 
to the character given by 
\begin{equation}
\label{eq:char hom}
f^*(\rho)(a_H) = \zeta_{\alpha}^{m_H}, \quad\text{for  $H\in \A_{\alpha}$}.
\end{equation}

Now assume $\NN$ is reduced.  Taking $\zeta_{\alpha}=e^{2\pi \ii/k}$ 
in the above, we see that $\rho_k=f^*(\rho)$ belongs to the 
subtorus $T=f^*(H^1(S, \C^*))$.  Since $T$  lies inside $\VV_{k-2}(\A)$, 
formula \eqref{eq:h1milnor} shows that $e_k(\A)=\depth (\rho_k)\ge k-2$.  

Finally, suppose $k=p^s$.  Then $\rho_{p^r}=f^*(\rho)^{p^{s-r}}$, which 
again belongs to the subtorus $T$, for $1\le r\le s$.  
The inequality $\e_{p^r}(\A)\ge k-2$ now follows as above.
\end{proof}

\begin{remark}
\label{rem:multik}
An alternate way of proving Theorem \ref{thm:be3}, part \eqref{bb1} is by 
putting  together \cite[Theorem 3.11]{FY} and \cite[Theorem 3.1(i)]{DP11}.  
The proof we give here, though, is more direct, and, besides, it 
will be needed in the proof of Theorem \ref{thm:b3e3} below.  
Furthermore, the additional part \eqref{bb2} will be used 
in proving Theorem \ref{thm:2main1}.
\end{remark}

\subsection{From $\beta_3(\A)$ to $e_3(\A)$}
\label{subsec:e3beta3}

Before proceeding, we need to recall a result of Artal Bartolo, 
Cogolludo, and Matei, \cite[Proposition 6.9]{ACM}.

\begin{theorem}[\cite{ACM}]
\label{thm:acm}
Let $X$ be a smooth, quasi-projective variety.
Suppose $V$ and $W$ are two distinct, positive-dimensional 
irreducible components of $\VV_r(X)$ and $\VV_s(X)$, 
respectively.  If $\xi\in V\cap W$ is a torsion character, 
then $\xi \in \VV_{r+s} (X)$. 
\end{theorem}

The next theorem establishes implication \eqref{a5} $\Rightarrow$ \eqref{a6} 
from Theorem \ref{thm:main1} in the Introduction. 

\begin{theorem}
\label{thm:b3e3}
Let $\A$ be an arrangement that has no rank-$2$ flats 
of multiplicity properly divisible by $3$.  If\, $\beta_3(\A) \le  2$,  
then $\e_3(\A) =\beta_3(\A)$. 
\end{theorem}

\begin{proof}
From the modular bound \eqref{eq:bound bis}, we know that 
$\e_3(\A) \le \beta_3(\A)$.  Since we are assuming 
that $\beta_3(\A) \le  2$, there are only three cases to consider.
First, if $\beta_3(\A)=0$, then clearly $e_3(\A)=0$.
Second, if $\beta_3(\A)=1$, then $e_3(\A)=1$, by 
implication \eqref{a3} $\Rightarrow$ \eqref{a1} from 
Theorem \ref{thm:main1} and Theorem \ref{thm:be3}\eqref{bb1}.

Finally, suppose $\beta_3(\A)=2$.  
We then know from Theorem \ref{thm:main1}\eqref{a7}  that the resonance variety 
$\RR_1(\A)$ has at least $4$ essential components, all corresponding 
to $3$-nets on $\A$.  Pick two of them, given by $3$-nets $\NN$ and $\NN'$, 
constructed as in Corollary \ref{cor:res comp}.

By Theorem \ref{thm:be3}, the characteristic variety $\VV_1(\A)$ 
has two positive-dimensional components, $f_{\NN}^*(H^1(S, \C^*))$ 
and $f_{\NN'}^*(H^1(S, \C^*))$, both passing through  the torsion 
character $\rho_3$.  These components must be distinct, since the 
corresponding components of the resonance variety, $f_{\NN}^*(H^1(S, \C))$ 
and  $f_{\NN'}^*(H^1(S, \C))$, are distinct.  By Theorem \ref{thm:acm}, 
then, $\rho_3$ belongs to $\VV_2(\A)$.  Formula \eqref{eq:h1milnor} 
now gives $e_3(\A)\ge 2$.  By the modular bound, $e_3(\A)= 2$, 
and the proof is complete.
\end{proof}

\subsection{From $\beta_2(\A)$ to $e_2(\A)$ and $e_4(\A)$}
\label{subsec:e2beta2}

We are now ready to complete the proof of Theorem \ref{thm:2main1} 
from the Introduction. 

\begin{proof}[Proof of Theorem \ref{thm:2main1}]
The equivalence \eqref{2m1}$\Leftrightarrow$\eqref{2m3} is a direct consequence 
of Theorem \ref{teo=lambdaintro}, case $\k=\FF_4$.

Suppose $\A$ admits a  $4$-net. 
By Theorem \ref{thm:be3}\eqref{bb2}, then, both $e_2(\A)$ 
and $e_4(\A)$ are at least $2$. On the other hand, 
the modular bound \eqref{eq:bound bis} gives that both $e_2(\A)$ and 
$e_4(\A)$ are at most $\beta_2(\A)$. The further assumption that 
$\beta_2(\A)\le 2$ implies that this bound is sharp, and we are done.
\end{proof}

\begin{example}[cf.~\cite{DP11, FY, Yu04}]
\label{ex:hesse}
In Theorem \ref{thm:2main1}, we were guided 
by the properties of the Hessian arrangement. 
This is the arrangement $\A$ of $12$ lines in $\CP^2$ which 
consists of the $4$ completely reducible fibers of the cubic pencil 
generated by $z_1^3+z_2^3+z_3^3$ and $z_1z_2z_3$. 
Each of these fibers is a union of $3$ lines in general position.  
The resulting partition defines a $(4,3)$-net 
on $L_{\le 2}(\A)$, depicted in Figure \ref{fig:hessian}. 
Clearly, $\mult (\A)= \{ 4\}$. 

From the above information, we find that $\beta_2(\A)=2$.  
Using Theorem \ref{thm:2main1}, 
we recover the known result that $\Delta_{\A}(t)= (t-1)^{11}[(t+1)(t^2+1)]^2$.
The Hessian arrangement shows  that the hypothesis 
on multiplicities is needed in Theorem \ref{thm:main0}.  Indeed, 
$\beta_p(\A)=0$ for all primes $p\ne 2$, by Corollary \ref{cor:beta zero}; 
consequently, $\Delta_{\A}(t)\ne  (t-1)^{11}(t^2+t+1)^{\beta_3(\A)}$. 

Finally, we infer from the above discussion that Conjecture \ref{conj:mf} 
holds  for the Hessian arrangement, in the strong form \eqref{eq:delta arr}.
\end{example}

\subsection{More examples}
\label{subsec:examples}

We conclude this section with several applications to  
some concrete classes of examples.  To start with, 
Theorem \ref{thm:main1} provides a partial answer 
to the following question, raised by Dimca, Ibadula, and 
M\u{a}cinic in \cite{DIM}: If $\rho_d \in \VV_{1}(\A)$, 
must $\rho_d$ actually belong to a component of $\VV_{1}(\A)$ 
passing through the origin?  

\begin{corollary}
\label{cor:dimq}
Suppose $L_2(\A)$ has no flats of multiplicity $3r$, for any $r>1$. 
If $\rho_3 \in \VV_{1}(\A)$, then $\rho_3$ belongs to a 
$2$-dimensional component of $\VV_{1}(\A)$, passing 
through $1$.
\end{corollary}

\begin{proof}
Since $\rho_3 \in \VV_{1}(\A)$, formulas \eqref{eq:h1milnor}--\eqref{eq:bound bis} 
imply that $\beta_3(\A)\neq 0$.  By Theorem \ref{thm:main1}, then, 
$\A$ supports a reduced $3$-multinet $\NN$. By Theorem \ref{thm:be3},
the character $\rho_3$ belongs to the $2$-dimensional subtorus 
$f_{\NN}^*(H^1(S, \C^*))\subset \VV_{1}(\A)$.
\end{proof}

\begin{remark}
\label{rem:fourth}
In \cite{Di}, A. Dimca used superabundance methods to analyze the 
algebraic monodromy action on $H_1(F(\A),\C)$, for an arrangement $\A$ 
which has at most triple points, and which admits a reduced $3$-multinet. 
In the case when $\abs{\A}=18$, he discovered an interesting type of 
combinatorics, for which he proved the following dichotomy result: 
there are two possibilities for $\A$, defined in superabundance terms, 
and leading to different values for $e_3(\A)$. 

An example due to M. Yoshinaga and recorded in \cite{Di} shows 
that one of these two cases is actually realizable. Our Theorem \ref{thm:main0} 
then shows that the other case is not realizable, thereby answering the 
subtle question raised in \cite[Remark 1.2]{Di}. This indicates that 
our topological approach may also be used to solve difficult 
superabundance problems. 
\end{remark}

\begin{example}
\label{ex:cevad}
Let $\{\A_m\}_{m\ge 1}$ be the family of 
monomial arrangements, corresponding to the complex reflection 
groups of type $G(m,m,3)$, and defined by the polynomials 
\[
Q_m=(z_1^m-z_2^m)(z_1^m-z_3^m)(z_2^m-z_3^m).
\]
As noted in \cite{FY}, each arrangement $\A_m$ supports a $(3,m)$-net 
with partition given by the factors of $Q_m$, and has Latin square 
corresponding to $\Z_m$. There are $3$ mono-colored flats of multiplicity $m$,
and all the others flats in $L_2(\A_m)$ have multiplicity $3$.

With this information at hand, Lemma \ref{lem:eqs} easily implies that 
$\beta_3(\A_m)=1$ if $3 \nmid m$, and $\beta_3(\A_m)=2$, otherwise.
In the first case, we infer from Theorem \ref{thm:main1} 
that $e_3(\A_m)=1$.  If $m=3$, then $\A_3$ is the Ceva arrangement from 
Example \ref{ex:third}; in this case, Theorem \ref{thm:main1} shows that 
$e_3(\A_3)=2$.  Finally, if $m=3d$, with $d>1$, the multiplicity assumption 
from Theorem \ref{thm:main1} no longer holds;  nevertheless, the 
methods used here can be adapted to show that $e_3(\A_{3d})=2$, 
for all $d$.  In fact, it can be shown that $e_p(\A_m)=\beta_p(\A_m)$, for 
all $m\ge 1$ and all primes $p$, see \cite{MPP} for full details.
\end{example}

We conclude with an in-depth analysis of a family of arrangements 
which highlights the necessity of our multiplicity assumptions from 
Theorem \ref{thm:main1}, and reveals several other interesting 
phenomena.

\begin{example}
\label{ex:B3 bis}
Let $\{\A_m\}_{m\ge 1}$ be the family of full monomial arrangements, 
corresponding to the complex reflection groups of type $G(m,1,3)$, 
and defined by the polynomials 
\[
Q_m=z_1z_2z_3(z_1^m-z_2^m)(z_1^m-z_3^m)(z_2^m-z_3^m).
\]
It is easy to see that $\mult (\A_m)= \{ 3,m+2 \}$.  
Using Lemma \ref{lem:eqs}, we infer that $\beta_3(\A_m)=1$ 
if $m=3d+1$, and $\beta_3(\A_m)=0$, otherwise.

As noted in \cite{FY}, each arrangement $\A_m$ supports a
$3$-multinet $\NN_m$ (non-reduced for $m>1$), with multiplicity 
function equal to $m$ on the hyperplanes $z_1=0$, $z_2=0$, $z_3=0$, 
and equal to $1$, otherwise. Let $f\colon M(\A_m)\to 
S= \CP^1\setminus \{\text{$3$ points}\}$ be the associated 
admissible map, and let $\rho\in H^1(S, \C^*)$ be the
diagonal character used in the proof of Theorem \ref{thm:be3} for $k=3$.

If $m=1$, then $\A_1$ is the  reflection arrangement of type 
$\operatorname{A}_{3}$ from Example \ref{ex:graphic}, and 
$\NN_1$ is the $3$-net from Figure \ref{fig:braid}.  In this 
case, Theorem \ref{thm:main1} applies, giving $e_3(\A)=\beta_3(\A)=1$.  

Now suppose $m=3d+1$, with $d>0$.  In this case, even though $\NN_m$ 
is not reduced, the equality $f^*(\rho)=\rho_3$ still holds, 
since $m\equiv 1 \bmod 3$. Consequently, $\rho_3$ 
belongs to the component $T=f^*(H^1(S,\C^*))$ of $\VV_1(\A_m)$. 
Hence, by formula \eqref{eq:h1milnor}, $e_3(\A_m)\ge 1$.  

Although $\beta_3(\A_m)=1$, we claim that $\A_m$ 
supports no reduced $3$-multinet. Indeed,
suppose there was such a multinet $\NN'_m$, with 
corresponding admissible map $f' \colon M(\A_m)\to S$. 
The same argument as above shows that $\rho_3$ 
lies in the component $T'=f'^*(H^1(S,\C^*))$ of $\VV_1(\A)$. 
If $T= T'$, then $f^*(H^1(S,\C))=f'^*(H^1(S,\C))$, forcing $\NN_m$ 
and $\NN'_m$ to be conjugate under the natural $\Sigma_3$-action, 
by Lemma \ref{lem=fibres}. This is clearly impossible, since 
$\NN'_m$ is reduced and $\NN_m$ is not reduced. Hence, 
the components $T$ and $T'$ are distinct, and so  
Theorem \ref{thm:acm} implies that $\rho_3\in \VV_2(\A_m)$. 
By formula \eqref{eq:h1milnor}, we must then have $\e_3(\A_m)\ge 2$, 
thereby contradicting inequality \eqref{eq:bound bis}. 
\end{example}

\begin{remark}
\label{rem:mono1}
The above family of examples shows that implication 
\eqref{a3} $\Rightarrow$ \eqref{a1} from Theorem \ref{thm:main1} 
fails without our multiplicity restrictions.  Indeed, $\beta_3(\A_m)=1$, 
yet $\A_m$ supports no reduced $3$-multinet if $m\ge 2$.
\end{remark}

\begin{remark}
\label{rem:mono2}
These examples also show that the inclusion $Z'_{\FF_3}(\M) \subseteq Z_{\FF_3}(\M)$
from Theorem \ref{teo=lambdaintro}\eqref{li2} can well be strict, even for realizable 
matroids.  Indeed, consider the arrangements $\A_{3d+1}$ with $d>0$.
The equality $Z'_{\FF_3}(\A_{3d+1})= Z_{\FF_3}(\A_{3d+1})$
would imply that $Z'_{\FF_3}(\A_{3d+1})\setminus  B_{\FF_3}(\A_{3d+1}) \ne \emptyset$,
since $\beta_3 (\A_{3d+1})\ne 0$. By Lemma \ref{lema=lambdabij}, 
$\A_{3d+1}$ would then support a $3$-net, contradicting the conclusion 
of the previous remark.
\end{remark}

\begin{remark}
\label{rem:mono3}
Note that the inequality \eqref{eq=essboundintro} 
is strict for $\A_{3d+1}$ and $k=3$ when $d>0$, since 
$\E_3 (\A_{3d+1})= \emptyset$ and $\beta_3 (\A_{3d+1})=1$. 
The same argument also shows that equality \eqref{a7} from Theorem \ref{thm:main1} 
fails without our multiplicity assumptions.
\end{remark}

\begin{remark}
\label{rem:monomial}
The above analysis shows that $e_3(\A_m)=1$, and thus \eqref{eq:bound bis} 
holds as an equality if $3\mid m+2$.  We also have that  
$e_3(\A_m)=\beta_3(\A_m)=0$ if $3\nmid m+2$. 
In fact, it can be checked that Conjecture \ref{conj:mf} holds in 
the strong form \eqref{eq:delta arr}, for all full monomial arrangements. 
For a complete proof of Conjecture \ref{conj:mf}\eqref{eq:delta arr} for all complex 
reflection arrangements, we refer to \cite{MPP, D16, DS}.
\end{remark}

\section{A family of matroids}
\label{sec:matr}

We conclude by constructing an infinite family of matroids $\M(m)$ 
which are realizable over $\C$ if and only if $m\le 2$, and with the 
property that $\beta_3(\M(m))=m$.  As an application, we 
establish part \eqref{a5} of Theorem \ref{thm:main1} from 
the Introduction, thereby completing the proof of that theorem. 

\subsection{Matroids coming from abelian groups}
\label{ss71}

Given a finite abelian group $G$ and an integer $m\ge 1$, there is 
a simple matroid $\M(G, m)$ of rank at most $3$ on the product 
group $G^m$. The dependent subsets of size $3$ of this matroid 
are all $3$-tuples $\{ v,v',v'' \}$ for which $v+ v' +v''=0$.  

In this section, we take $G=\FF_3$ and omit $G$ from the notation.  
A useful preliminary remark is that $v+v'+v''=0$ in $\FF_3$ if and 
only if either $v=v'=v''$ or $v$, $v'$, and $v''$ are all distinct. 
Note also that $\GL_m(\FF_3)$ acts naturally on the matroid $\M(m)$.

Clearly, $\M(1)$ has rank $2$, and is realized in $\CP^2$ by an arrangement of $3$ 
concurrent lines. It is equally clear that $\M(m)$ has rank $3$, for all $m>1$. 
It is not hard to check that $\M(2)$ is realized over $\C$ by the Ceva arrangement 
from Example \ref{ex:third}. Our first goal is to show that $\M(m)$
cannot be realized over $\C$, for any $m>2$. 

It is useful to remark that the above construction is related to a classical topic, 
namely, finite affine geometries. Over a Galois field $\k=\FF_{p^s}$, 
the dependent subsets of the associated rank $m+1$ simple matroid $\AG (m, \k)$ 
are given by affine dependence in $\k^m$, for $m\ge 1$. 
Let $\M_{\k}(m)$ be the simple matroid of rank at most $3$ on $\k^m$ 
whose size $3$ dependent subsets are given by all collinearity relations. 
Plainly, all $2$-flats of $\M_{\k}(m)$ have multiplicity $p^s$. Furthermore,  $\M_{\k}(1)$ 
has rank $2$, hence is realizable in $\CP^2$, while $\M_{\k}(m)$ has rank $3$ 
for $m>1$, being obtained from $\AG (m, \k)$ by rank $3$ truncation.

It is straightforward to check that $\M(m)= \M_{\FF_3} (m)$, for all $m$; 
see for instance \cite{DM}.  With this remark, the next two lemmas (to be 
used later on in this section) become obvious from a geometric viewpoint.

For each $m\ge 1$, define a map 
\begin{equation}
\label{eq=defq}
q\colon \M(m) \times \M(m) \rightarrow \M(m)
\end{equation}
by setting $q(v,v)=v$ and $q(v,v')=-v-v'$ if $v\neq v'$. 
By construction, 
$q(v,v')$ is the unique point in $v\vee v'$ (the flat generated by $v,v'$) 
different from $v$ and $v'$.  

\begin{lemma}
\label{lem=mflats}
For any $v\neq v'$, the set  $\{ v,v',q(v,v') \}$ belongs to $ L_2(\M(m))$. 
Conversely, every rank $2$ flat of $\M(m)$ is of this form. In particular, 
all flats in $L_2(\M(m))$ have multiplicity $3$.
\end{lemma}

The matroids $\M(m)$ have a lot of $3$-nets. More precisely, 
fix $a\in [m]$ and write
\begin{equation}
\label{eq=mpart}
\M(m)= \coprod_{i\in \FF_3} \M_i(m),
\end{equation}
where $\M_i(m)= \{ v=(v_1,\dots,v_m) \in \FF_3^m \mid v_a=i \}$. 
We then have the following lemma.

\begin{lemma}
\label{lem=mnets}
For each $m\ge 2$ and $a\in [m]$, the partition \eqref{eq=mpart} 
defines a $3$-net on $\M(m)$, with all submatroids $\M_i(m)$ 
being line-closed and isomorphic to $\M(m-1)$. 
\end{lemma}

\begin{proof}
First note that an affine line whose direction is transversal to the
direction of an affine hyperplane intersects that hyperplane in exactly 
one point.  It follows that our partition satisfies the criterion from 
Lemma \ref{lem:lsq} for $k=3$, and thus defines a $3$-net on $\M(m)$. 
The proofs of the other claims are immediate.
\end{proof}

\begin{remark}
\label{rk=kge5}
More generally, for $\k=\FF_{p^s}$ and $k=p^s$, the same argument 
shows that the affine truncation $\M_{\k} (m)$ supports a non-trivial $k$-net, 
provided $m\ge 2$ and $k\ge 3$. As explained in \S\ref{subsec:red multi}, 
this fact implies that $\M_{\k} (m)$ is non-realizable over $\C$, for $k\ge 5$.  
We will give more precise results of this type in Proposition \ref{prop:mpm}\eqref{mat2} 
and  Theorem \ref{thm=mnotr} below.
\end{remark}

\subsection{The $\beta$-invariants and matroid realizability}
\label{subsec:beta realize}

Next, we show that our family of matroids is universal for 
$\beta_3$-computations, in the following sense.

\begin{prop}
\label{prop=mb3}
For all $m\ge 1$, we have that $\beta_3(\M(m))=m$.
\end{prop}

\begin{proof}
For $m=1$, this is clear. The inequality $\beta_3(\M(m))\le m$ follows 
by induction on $m$, using Corollary \ref{cor:betabounds} and 
Lemma \ref{lem=mnets}.  

For each $a\in [m]$, define a vector $\eta_a\in \FF_3^{\M(m)}$ 
by setting $\eta_a(v)=v_a$. We claim that $\eta_a\in Z_{\FF_3}(\M(m))$. 
Indeed, let $ \{ v,v',v'' \}$ be a flat in $L_2(\M(m))$, so that $v+v'+v''=0$. 
Then $\eta_a(v)+\eta_a(v')+\eta_a(v'')=0$, and the claim follows  
from Lemma \ref{lem:eqs}. 

To prove that $\beta_3(\M(m))\ge m$, it is enough to show 
that $\eta_1,\dots,\eta_m$ and $\sigma$ 
are independent, where $\sigma \in \FF_3^{\M(m)}$ is the standard 
diagonal vector (i.e., $\sigma_v = 1$, for all $v\in \M(m)$). 
To that end, suppose $\sum_a c_a \eta_a + c\sigma =0$.  Evaluating on $v=0$, 
we find that $c=0$. Finally, evaluation on the standard basis vectors 
of $\FF_3^m$ gives $c_a=0$ for all $a$, as needed.
\end{proof}

\begin{prop}
\label{prop:mpm}
Let  $\k=\FF_{p^s}$ be a finite field different from $\FF_2$. Then,  
\begin{enumerate}
\item \label{mat1}
$\beta_p(\M_{\k} (m))\ge m$, for all $m\ge 1$.
\item \label{mat2}
If $\k \ne \FF_3$, the matroids $\M_{\k} (m)$ are non-realizable over $\C$, 
for all $m\ge 2$.
\end{enumerate}
\end{prop}

\begin{proof}
From our hypothesis on $\k$, we have that $\Sigma(\k)=0$, 
by Lemma \ref{lema=sigmazero}.  
In view of Lemma \ref{lem:eqs}, any affine function $\tau \in \k^{\M_{\k} (m)}$
belongs to $Z_{\k}(\M_{\k} (m))$. Indeed, let 
$X= \{ \alpha u + (1-\alpha)v \mid \alpha \in \k \}$ be a rank-$2$ flat of $\M_{\k} (m)$.
Then 
\[
\sum_{w\in X} \tau_w= \sum_{\alpha \in \k} \alpha \tau_u + (1-\alpha) \tau_v=
p^s \cdot \tau_v + \Sigma(\k) (\tau_u -\tau_v)=0, 
\] 
as needed.  An argument as in the proof of Proposition \ref{prop=mb3} 
now shows that assertion \eqref{mat1} holds.

To prove assertion \eqref{mat2}, we use the Hirzebruch--Miyaoka--Yau 
inequality from \cite{Hi}.  This inequality involves the numbers 
\begin{equation}
\label{eq:ti}
t_i = \abs{\{ X\in L_2(\M) \mid \abs{X}=i \}}
\end{equation}
associated to a matroid $\M$ for each $1<i \le \abs{\M}$.  When there 
are no rank-$2$ flats of multiplicity $\abs{\M}$ or $\abs{\M}-1$, and 
$\M$ is realized by a line arrangement in  $\CP^2$, then 
\begin{equation}
\label{eq=hmy}
t_2+\frac{3}{4} t_3 \ge \abs{\M}+ \sum_{i\ge 4} (i-4)t_i .
\end{equation}

In our case, $\abs{\M}=p^{sm}$, and the only non-zero number 
$t_i$ occurs when $i=p^s>3$.  This clearly violates \eqref{eq=hmy}, 
thus showing that $\M$ is not realizable. (Note that this argument breaks down
for $\k=\FF_3$.)
\end{proof}

\subsection{Realizability of the $\M(m)$ matroids}
\label{subsec:realize}

The case $\k=\FF_3$ is much more subtle. In order to proceed 
with this case, we need to recall a result of 
Yuzvinsky (Corollary 3.5 from \cite{Yu04}).  Let $\A$ be an 
arrangement in $\C^3$. 
 
\begin{lemma}[\cite{Yu04}]
\label{lem:yuz}
Let $\NN$ be a $(3,mn)$-net ($m\ge 3$) on $\A$, such 
that each class $\A_i$ can be partitioned into $n$ blocks of size $m$, 
denoted $\A_{ij}$, and for every pair $i,j$, there is a $k$ such that 
$\A_{1i}$, $\A_{2j}$, and $\A_{3k}$ are the three classes of a $(3,m)$-subnet 
$\NN_{ij}$ of $\NN$.  If, moreover, each class of $\NN_{ij}$ is a pencil, 
then every class of $\NN$ is also a pencil.
\end{lemma}

Realizability in the family of matroids $\{\M(m)\}_{m\ge 1}$ is 
settled by the next result. 

\begin{theorem}
\label{thm=mnotr}
For any $m\ge 3$, the sub-lattice $L_{\le 2}(\M(m))$ is not realizable over $\C$, 
i.e., there is no arrangement $\A$ in $\C^{\ell}$ such that $L_{\le 2}(\M(m)) \cong 
L_{\le 2}(\A)$, as lattices.
\end{theorem}

\begin{proof}
By Lemma \ref{lem=mnets}, the matroid $\M(m-1)$ embeds in $\M(m)$; 
thus, we may assume $m=3$. Clearly, it is enough to show that 
$L_{\le 2}(\M(3))$ cannot be realized by any arrangement in $\C^3$. 

Assuming the contrary, we will use Lemma \ref{lem:yuz} 
to derive a contradiction. Take $a=1$ in Lemma \ref{lem=mnets}, 
and denote by $\NN$ the associated $(3,9)$-net on $\M(3)$. 
Write each class in the form
\begin{equation}
\label{eq:mi3}
\M_i(3)= \{ i \} \times \FF_3 \times \FF_3 = \coprod_{j\in \FF_3} \M_{ij}(3),
\end{equation}
where $\M_{ij}(3)=  \{ i \} \times\{ j \} \times  \FF_3$.  For $j,j'\in \FF_3$,
define $j''\in \FF_3$ by $j+j'+j''=0$. To check the first assumption from
Lemma \ref{lem:yuz}, we have to show that the partition 
$\M_{0j}(3) \coprod \M_{1j'}(3) \coprod \M_{2j''}(3)$ is a $3$-subnet of $\NN$.

Pick $v=(i,j,k)$ and $v'=(i',j',k')$ in two different classes of this partition. Note 
that necessarily $i\neq i'$, by construction.  By Lemma \ref{lem=mflats}, 
$v\vee v'= \{ v,v',v''\}$, where $v''=-v-v' $.  Thus, we must have $v''=(i'',j'',k'')$, 
where $i+i'+i''=j+j'+j''=0$. This implies that $v''$ belongs to the third class of 
the partition, as required for the $3$-net property. Clearly, the $3$-net defined 
by this partition is a $3$-subnet of $\NN$.

As noted before, each class of the partition has rank $2$, being 
isomorphic to the matroid $\M(1)$. Hence, Lemma \ref{lem:yuz} 
applies, and implies that all classes of $\NN$ have rank $2$. 
On the other hand, Lemma \ref{lem=mnets} insures that 
these classes are isomorphic to $\M(2)$.  This is a 
contradiction, and so the proof is complete.
\end{proof}

Let $\M$ be a simple matroid of rank at least $3$. By taking a generic slice, 
it is easy to check that the following statements are equivalent: 
\begin{enumerate}
\item \label{ox1}
The rank $3$ truncation $\tau_3(\M)$ is realizable over $\C$;
\item \label{ox2}
The sub-lattice $L_{\le 2}(\tau_3(\M))$ is realizable over $\C$;
\item \label{ox3}
The sub-lattice $L_{\le 2}(\M)$ is realizable over $\C$.
\end{enumerate}

Applying now Proposition \ref{prop:mpm}\eqref{mat2} 
and Theorem \ref{thm=mnotr} to the matroids $\M=\AG (m, \k)$, 
we obtain the following corollary. 

\begin{corollary}
\label{cor:strong ox}
Let $\k=\FF_{p^s}$ and set $k=p^s$.  
Suppose $m\ge 2$ and $k\ge 3$.  Then 
the lattice $L_{\le 2}(\AG (m, \k))$ is realizable over 
$\C$ if and only if $m=2$ and $k=3$.  
\end{corollary}

\begin{remark}
\label{rem:ox} 
It is well-known that, for $m\ge 2$ and $k\ge 3$, 
the matroid $\AG (m, \k)$ is realizable over $\C$ if and only if $m=2$ and $k=3$, 
see for instance Oxley \cite[p.~522]{Ox}.  
Clearly, our non-realizability 
result is stronger: not only is this matroid non-realizable, but even 
its collinearity relations are not realizable. 
\end{remark}

\subsection{Collections of $3$-nets}
\label{ss72}

For the rest of this section, $\A$ will denote an arrangement in $\C^3$.
Our goal is to prove the following theorem, which verifies assertion  
\eqref{a5} from Theorem \ref{thm:main1} in the Introduction. 

\begin{theorem}
\label{thm:main4}
Suppose $L_2(\A)$ has no flats of multiplicity properly divisible by $3$.  
Then $\beta_3(\A) \le 2$.
\end{theorem}

Our strategy is based on the map $\lambda_{\FF_3} \colon 
\{\text{$3$-nets on $\A$}\} \to 
Z_{\FF_3}(\A)\setminus B_{\FF_3}(\A)$ from Theorem \ref{teo=lambdaintro}. 
Recall that the map $\lambda_{\FF_3}$ is always injective.  Moreover, 
as shown in Lemma \ref{lem:lambda}, this map is also surjective 
when the above assumption on multiplicities is satisfied.
In this case, $\beta_3(\A)\ge m$ if and only if there is a collection 
$\NN^1,\dots,\NN^m$ of $3$-nets on $\A$ such that the classes 
$[\lambda_{\FF_3}(\NN^1)],\dots,[\lambda_{\FF_3}(\NN^m)]$ are independent in 
$Z_{\FF_3}(\A)/B_{\FF_3}(\A)$. 

Let $\A$ be an arbitrary arrangement. When the above property holds, we call the nets 
$\{ \NN^a \}_{a\in [m]}$ {\em independent}. For $v=(v^1,\dots,v^m)\in \M(m)$, 
we set $\A_v= \bigcap_{a\in [m]} \NN^a_{v^a}$, where $\A= \coprod_{i\in \FF_3} \NN^a_i$ 
is the partition associated to $\NN^a$. We say that the family $\{ \NN^a \}_{a\in [m]}$ 
has the {\em intersection property}\/ if $\A_v \neq \emptyset$ for all $v$, 
and  the {\em strong intersection property}\/ if there is an integer $d>0$ 
such that $\abs{\A_v} =d$ for all $v$.

Clearly, we have a partition, $\A_v \coprod \A_{v'} \coprod  \A_{v''}$, for 
any flat $\{v,v',v'' \} \in L_2(\M(m))$.  If all these partitions define $3$-nets, 
we say that $\{ \NN^a \}_{a\in [m]}$ has the {\em net property}.

Our starting point towards the proof of Theorem \ref{thm:main4} is 
the following theorem.

\begin{theorem}
\label{thm=3mnets}
Suppose there is an arrangement $\A$ in $\C^3$, and a collection 
of $3$-nets on $\A$,  $\{ \NN, \NN', \NN'' \}$, that has both the strong 
intersection property and the net property. Then the sub-lattice 
$L_{\le 2}(\M(3))$ is realizable over $\C$. 
\end{theorem}

\begin{proof}
Let $S^d(3)\subseteq \Sym^{\hdot}(3)$ be the vector space of degree $d$ 
polynomials in $3$ variables. We will realize $L_{\le 2}(\M(3))$ in the 
dual space, $V=S^d(3)^*$. 

For each hyperplane $H\in \A$, choose a linear form $f_H$ in $3$ 
variables such that $H=\ker(f_H)$. Next, we associate to a point $v\in \M(3)$ 
the vector $Q_v=\prod_{H\in \A_v} f_H \in V^* \setminus \{ 0\}$,
using the strong intersection property.  Note that $\{ f_H \}_{H\in \A}$ 
are distinct primes in the ring $\Sym (3)$. In particular,
$\{ \overline{Q}_v \}_{v\in \M(3)}$ are distinct elements of $\PP (V^*)$, 
since $\A_v \cap \A_{v'}= \emptyset$ for $v\neq v'$. 

We have to show that $\{ v_1,v_2,v_3 \}$ is a dependent set in 
$\M(3)$ if and only if the set $\{ Q_{v_1}, Q_{v_2}, Q_{v_3} \}$ 
has rank $2$. If $\{ v_1,v_2,v_3 \}$ is a flat in $L_2(\M(3))$,
the rank property for $\{ Q_{v_1}, Q_{v_2}, Q_{v_3} \}$ follows 
from the fact that the partition $\A_{v_1} \coprod \A_{v_2} \coprod  \A_{v_3}$ 
defines a $3$-net, according to \cite[Theorem 3.11]{FY}.

Conversely, let $\{ v_1,v_2,v'_3 \}$ be a size $3$ independent 
subset of $\M(3)$, so that $v_1+v_2+v'_3 \neq 0$. Consider the flat 
$\{ v_1,v_2,v_3 \} \in L_2(\M(3))$, where $v_1+v_2+v_3 = 0$; 
in particular, $v_3 \neq v'_3$.  Assume that 
$Q_{v'_3}= c_1 Q_{v_1} + c_2 Q_{v_2}$.  Pick hyperplanes $H_1 \in \A_{v_1}$
and $H_2 \in \A_{v_2}$, and let $X= H_1 \cap H_2$.  
By the net property, there is a hyperplane $H_3 \in \A_{v_3}$ such that 
$\{ H_1, H_2, H_3 \} \subseteq \A_X$.  

Let $x=\overline{X}$ be the intersection point of the projective 
lines $\overline{H_1}$ and $\overline{H_2}$. 
We then have $Q_{v_1}(x)=Q_{v_2}(x)=0$, and thus $Q_{v'_3}(x)=0$. 
Therefore, there is a hyperplane $H'_3 \in \A_{v'_3} \cap \A_X$. 
Consequently, the arrangement $\A_X$ contains the four distinct hyperplanes 
$\{ H_1, H_2, H_3 , H'_3\}$. Hence, the flat $X$ must be monocolor with 
respect to the collection $\{ \NN, \NN', \NN'' \}$, that is,  
$\A_X \subseteq \NN_i \cap \NN'_j \cap \NN''_k =\A_v$, where
$v=(i,j,k)\in \M(3)$. Since $H_i \in \A_{v_i}$ for $i=1,2$, we infer 
that $v_1=v_2=v$, a contradiction. Our realizability claim is thus verified.
\end{proof}

In view of Theorem \ref{thm=mnotr}, Theorem \ref{thm:main4} will be proved
once we are able to show that the independence property for  $\{ \NN, \NN', \NN'' \}$ 
forces both the strong intersection property and the net property. 

\begin{lemma}
\label{lem=upgr1}
For a collection $\{ \NN^1, \dots, \NN^m\}$ of $3$-nets on $\A$, the intersection 
property implies both the strong intersection property and the net property. 
\end{lemma}

\begin{proof}
We start with the net property. Let $X=\{ v,v',v'' \}$ be a flat in $L_2(\M(m))$. 
We know that the first two classes of the partition 
$(\A_v , \A_{v'} , \A_{v''})$ are non-empty.  Write $v=(v_a)\in \FF_3^m$, 
and similarly for $v',v''$. By construction of the matroid $\M(m)$,
the elements $\{ v_a,v'_a,v''_a \}$ are either all equal or all distinct, 
for any $a\in [m]$. 

Pick $H\in \A_v$ and $H'\in \A_{v'}$. Then 
$H\in \NN^a_{v_a}$ and $H' \in \NN^a_{v'_a}$, for some $a$ with 
$v_a\neq v'_a$, since $v\neq v'$. Hence $\{ H,H',H'' \} \in L_2(\A)$, 
for a unique hyperplane $H''$, distinct from $H$ and $H'$, and 
which belongs to $\NN^a_{v''_a}$, by the net property for $\NN^a$.

We are left with checking that $H''\in \A_{v''}$, that is, $H''\in \NN^b_{v''_b}$ 
for all $b\in [m]$.  If $v_b\neq v'_b$, we may use the previous argument. 
If $v_b =v'_b$, then $v''_b =v_b =v'_b$.  Note that $X\in L_2(\NN^b_{v_b})$ 
and therefore $H'' \in \NN^b_{v_b}$, which is line-closed in $\A$,
by Lemma \ref{lem:net props}\eqref{n3}. By the intersection property, 
$\A_{v''}\neq \emptyset$.  According to Lemma \ref{lem:lsq}, then, 
$\{ \NN^a\}$ has the net property.

It will be useful later on to extract from the preceding argument 
the following implication:
\begin{equation}
\label{eq=impli}
\A_v\, ,\A_{v'}\neq \emptyset \Rightarrow \A_{q(v,v')}\neq \emptyset \, ,
\end{equation}
where $q$ is defined in \eqref{eq=defq}. 

We deduce from the net property  that 
$\abs{\A_v}=\abs{\A_{v'}}>0$, for any $v\neq v' \in \M(m)$, by 
using the flat $X=\{ v,v',q(v,v') \}$. The strong intersection property
follows.
\end{proof}

\subsection{The closure operation}
\label{ss73}

We have to analyze the relationship between the independence 
and the intersection properties.  In one direction, things are easy.

\begin{lemma}
\label{lem=intind}
If a collection $\{ \NN^1, \dots, \NN^m\}$ of $3$-nets on $\A$ has the 
intersection property, the nets are independent. 
\end{lemma}

\begin{proof}
We have to show that $\lambda_{\FF_3}(\NN^1),\dots,\lambda_{\FF_3}(\NN^m)$ and 
$\sigma$ are independent in $\FF_3^{\A}$. Note that 
$\lambda_{\FF_3}(\NN^a)_H= v^a$, by construction, for any 
$v=(v^1,\dots,v^m)\in \FF_3^m$ and $H\in \A_v$. Independence 
then follows exactly as in the proof of Proposition \ref{prop=mb3}.
\end{proof}

For the converse, we will need the well-known line-closure operation 
from matroid theory (see for instance Halsey \cite{Ha}). In the case of our family
of matroids $\M(m)$, this operation may be conveniently described as follows. 
For a subset $\M \subseteq \M(m)$, put $C\M= \{ q(u,v) \mid u,v\in \M\}$.
Then iterate and define $\overline{C}\M= \bigcup_{s\ge 1} C^s \M$, the 
line-closure of $\M$. The next lemma follows easily from the definitions.

\begin{lemma}
\label{lem=cprop}
For each $m\ge 1$, the following hold.
\begin{enumerate}
\item \label{cp1}
If $\M \subseteq \M(m)$, then $\M \subseteq C\M \subseteq \overline{C}\M$.
\item \label{cp2}
If $\M \subseteq \M'$, then $C^s\M \subseteq C^s\M'$ for all $s$, 
and $\overline{C}\M \subseteq \overline{C}\M'$. 
\item \label{cp3}
If $\M \subseteq \M'$ and the submatroid $\M'$ is line-closed in $\M(m)$, then
$C^s\M$ and  $\overline{C}\M$ coincide, when computed in $\M(m)$ and $\M'$.
\end{enumerate}
\end{lemma}

Given a collection $\{ \NN^a \}_{a\in [m]}$ of $3$-nets on $\A$, 
set $\M= \{ v\in \M(m) \mid \A_v \neq \emptyset \}$. We deduce from 
\eqref{eq=impli} the following characterization of the intersection property.

\begin{corollary}
\label{cor=cprop}
The family $\{ \NN^1,\dots, \NN^m \}$ has the intersection property
if and only if $\overline{C}\M= \M(m)$.
\end{corollary}

For $m=1$, it is clear that independence implies the intersection property.
We need to establish this implication for $m=3$. We have to start with the case $m=2$.
In order to minimize the amount of subcase analysis, it is useful to make a couple of
elementary remarks on matroid symmetry in the family $\{ \M(m)\}_{m\ge 1}$. 

Clearly, $\Aut (\M(1))= \Sigma_3$. It is equally clear that a partition $[m]=[n] \coprod [n']$
induces a natural morphism, $\Aut (\M(n)) \times \Aut (\M(n')) \to \Aut (\M(m))$.
We will need more details for $m=2$.

Let $X= \{v,v',v''\}$ be a size $3$ subset of $\FF_3^2$. It is easy to see that 
$X$ is dependent if and only if $X= \{ (i,0), (i,1), (i,2) \}$, or $X= \{ (0,j), (1,j), (2,j) \}$, 
or $X= \{ (i,gi) \mid i\in \FF_3 \}$, for some $g\in \Sigma_3$. 

Now assume that $X= \{ (i,j), (i',j'), (i'',j'') \}$ is independent. Modulo 
$\Sigma_2 \subseteq \GL_2$, we may assume that $\abs{\{ i,i',i'' \}}=2$. 
By $(\Sigma_3 \times \id)$-symmetry, we may normalize this
to $i=i'=0$ and $i''=1$, hence $j\neq j'$. If  $\abs{\{ j,j',j'' \}}=3$, the flat 
$X$ is normalized to $\{ (0,0), (0,1), (1,2) \}$, by $(\id \times \Sigma_3)$-symmetry. 
Otherwise, $X= \{ (0,0), (0,1), v'' \}$, with $v''=(1,0)$ or $v''=(1,1)$, and these 
two cases are $\GL_2$-conjugate, as well as $\{ (0,0), (0,1), (1,0) \}$ and 
$\{ (0,0), (0,1), (1,2) \}$. To sum up, any independent subset of size $3$
can be put in the normal form $\{ (0,0), (0,1), (1,0) \}$, modulo $\Aut (\M(2))$. 

\begin{lemma}
\label{lem=prelm2}
Let $\M \subseteq \M(2)$ be a submatroid with at least $3$ elements. 
\begin{enumerate}
\item \label{prm1}
If $\abs{\M}=3$ and $\M$ is independent, then $\overline{C}\M= \M(2)$.
\item \label{prm2}
If $\abs{\M}\ge 4$, then $\overline{C}\M= \M(2)$.
\end{enumerate}
\end{lemma}

\begin{proof}
Part \eqref{prm1}. First, put $\M$ in normal form, as explained above.
Then compute 
\begin{align*}
&q((0,0), (0,1))= (0,2), & q((0,0), (1,0))= (2,0), && q((0,2), (2,0))= (1,1),\\[-2pt]
& q((1,0), (1,1))= (1,2), & q((0,1), (1,1))= (2,1), && q((0,0), (1,1))= (2,2),
\end{align*}
and note that all the resulting values of $q$ belong to $\overline{C}\M$.

Part \eqref{prm2}. Pick a size $4$ subset $\{ v_1,v_2,v_3,v_4 \} \subseteq \M$. 
Then $\{ v_1,v_2,v_3\}$ and $\{ v_1,v_2,v_4\}$ cannot be both dependent. 
Our claim follows from part \eqref{prm1} and Lemma \ref{lem=cprop}\eqref{cp2}.
\end{proof}

We will need to know the behavior of the map
$\lambda_{\FF_3} \colon \{\text{$3$-nets on $\A$}\} \to Z_{\FF_3}(\A)\setminus B_{\FF_3}(\A)$
with respect to the natural $\Sigma_3$-action on $3$-nets. The description below does not require 
the realizability of the matroid $\A$.

Denote by $\sigma \in B_{\FF_3}(\A)$ the constant cocycle equal to $1$ on $\A$, as usual.
Let $g\in \Sigma_3$ be the $3$-cycle 
$(1,2,0)$ and let $h\in \Sigma_3$ be the transposition $(1,2)$, both acting on $\FF_3$. 
It is readily checked that, for any  $3$-net $\NN$ on $\A$,
\begin{equation}
\label{eq=actions}
\lambda_{\FF_3} (g\cdot \NN)= \sigma+ \lambda_{\FF_3} (\NN) \quad\text{and}\quad 
\lambda_{\FF_3} (h\cdot \NN)= - \lambda_{\FF_3} (\NN).
\end{equation}

We may now settle the case $m=2$. 

\begin{prop}
\label{prop=indint2}
If $\NN$ and $\NN'$ are independent $3$-nets on $\A$, then the pair 
$\{ \NN,\NN'\}$ has the intersection property.
\end{prop}

\begin{proof}
Plainly, for any $i\in \FF_3$ there is a $j\in \FF_3$ such that 
$\A_{(i,j)}= \NN_i \cap \NN'_j \neq \emptyset$ and similarly, 
for any $j\in \FF_3$ there is an $i\in \FF_3$ such that 
$\A_{(i,j)} \neq \emptyset$. In particular, $\abs{\M}\ge 3$. If
either $\abs{\M}\ge 4$, or $\abs{\M}= 3$ and $\M$ is independent,
we are done, in view of Lemma \ref{lem=prelm2} and Corollary \ref{cor=cprop}.

Assume then that $\abs{\M}= 3$ and $\M$ is dependent. According to 
a previous remark, $\M= \{ (i,gi) \mid i\in \FF_3 \}$, for some $g\in \Sigma_3$. 
For any $i\in \FF_3$, we infer that $\NN_i= \coprod_j \NN_i \cap \NN'_j \subseteq \NN'_{gi}$. 
Therefore, $\NN= g \cdot \NN'$. It follows from \eqref{eq=actions} 
that $[\lambda_{\FF_3}(\NN)]= \pm [\lambda_{\FF_3}(\NN')]$,
in contradiction with our independence assumption. 
\end{proof}

\begin{corollary}
\label{cor=indint2}
Assume $L_2(\A)$ contains no flats of multiplicity $3r$, with $r>1$. Then 
$\beta_3(\A)\ge 2$ if and only if there exist two $3$-nets on $\A$,
with parts $\{ \NN_i \}$ and $\{ \NN'_j \}$, respectively, 
such that $\NN_i \cap \NN'_j \neq \emptyset$, for all $i,j\in \FF_3$. 
\end{corollary}

\begin{proof}
Follows from Lemma \ref{lema=lambdabij}, Lemma \ref{lem:lambda}, 
Lemma \ref{lem=intind} and Proposition \ref{prop=indint2}. 
\end{proof}

\subsection{A bound of $\beta_3(\A)$}
\label{ss74}

In this last subsection, we complete the proof of Theorem \ref{thm:main1} 
from the Introduction. We start by analyzing the critical case, $m=3$. 
Let $\{ \NN,\NN', \NN''\}$ be a triple of $3$-nets on an arrangement $\A$, 
with parts $\{ \NN_i \}$, $\{ \NN'_j \}$, and $\{ \NN''_k \}$. 

For a fixed $k\in \FF_3$, set $\M_k= \{ u\in \M(2) \mid (u,k)\in \M \}$. Hence, 
$\M = \coprod_{k\in \FF_3} \M_k \times \{k \}$, where $\M_k$ 
is identified with $\M \cap (\M (2) \times \{k \})$, and all submatroids 
$\M (2) \times \{k \}$ are line-closed in $\M(3)$ and isomorphic to $\M(2)$,
as follows from Lemma \ref{lem=mnets}. 

We first exploit the independence property.

\begin{lemma}
\label{lem=ind3}
If $\NN$, $\NN'$, and $\NN''$ are independent $3$-nets, then there is no 
$k\in \FF_3$ such that $\M_k$ is a size $3$ dependent subset of $\M(2)$. 
\end{lemma}

\begin{proof}
Assuming the contrary, let $h''_k \colon Z_{\FF_3}(\A) \to Z_{\FF_3}(\A''_k)$ 
be the canonical homomorphism associated to the net $\NN''$, 
as in Lemma \ref{lem:proj}.  We know from  Proposition \ref{prop:ker1} 
that $\ker(h''_k)$ is $1$-dimensional. Let $h$ be 
the restriction of $h''_k$ to the $4$-dimensional subspace of $Z_{\FF_3}(\A)$ 
spanned by $\lambda_{\FF_3}(\NN)$, $\lambda_{\FF_3}(\NN')$, $\lambda_{\FF_3}(\NN'')$
and $\sigma$, a subspace we shall denote by $Z$. 

Note that all these $4$ elements of $\FF_3^{\A}$ are constant on $\A_v$, 
for any $v\in \M$, by construction. We deduce from our assumption on 
$\M_k$ that $\A''_k$ is of the form 
\[
\A''_k= \A_{(u_1,k)} \coprod \A_{(u_2,k)} \coprod \A_{(u_3,k)},
\]
with $u_1+u_2+u_3=0 \in \FF_3^2$.  As noted before, composing $h$ with 
the restriction maps from $\A''_k$ to $\A_{(u_i,k)}$ gives three linear maps, 
denoted $r_i \colon Z\to \FF_3$. Moreover, $\ker (h)=\ker (r)$, where
$r=(r_1\: r_2\: r_3)\colon Z \to \FF_3^3$. 

Since $\lambda_{\FF_3} (\NN)\equiv i$ on $\NN_i$, and similarly for $\NN'$ and $\NN''$, 
we infer that the matrix of $r$ is $\left( \begin{smallmatrix}
u_1 & u_2 & u_3 \\ k & k & k\\ 1 & 1 & 1 \end{smallmatrix}\right)$.
The fact that $u_1+u_2+u_3=0$ implies that $\dim \ker (r)\ge 2$.  
Therefore, $\dim \ker (h''_k)\ge 2$, a contradiction.
\end{proof}

Here is the analog of Proposition \ref{prop=indint2}.

\begin{prop}
\label{prop=indint3}
If $\NN$, $\NN'$, $\NN''$ are independent $3$-nets on $\A$, then this triple 
of nets has the intersection property.
\end{prop}

\begin{proof}
By Corollary \ref{cor=cprop}, we have to show that $\overline{C}\M= \M(3)$.
Due to Proposition \ref{prop=indint2}, we know that 
$\NN_i \cap \NN'_j \neq \emptyset$, for all $i,j\in \FF_3$. We deduce that,
for any $u\in \M(2)$, there is $k\in \FF_3$ such that $(u,k)\in \M$. 
In particular, $\abs{\M}\ge 9$. 

Similarly, for any $u\in \M(2)$, there is an element $k\in \FF_3$ 
such that $(k,u)\in \M$, by using $\NN'$ and $\NN''$. This shows 
that we cannot have $\M \subseteq \M(2) \times \{ k\}$, for some 
$k$, by taking $u=(j',k')$, with $k'\neq k$. 

We claim that if $\M(2) \times \{ k\} \subseteq \overline{C}\M$, 
for some $k\in \FF_3$, then we are done. Indeed, pick  $(u',k')\in \M$ 
with $k'\neq k$, take an arbitrary element $u\in \M(2)$, and compute 
$q((u,k),(u',k'))= (-u-u', k'')\in \overline{C}\M$, where $k''$ is the 
third element of $\FF_3$. This shows that 
$\M(2) \times \{ k''\} \subseteq \overline{C}\M$. Again,
$q((u,k''),(u',k))= (-u-u', k')\in \overline{C}\M$, for all $u,u' \in \M(2)$.
Hence, $\M(2) \times \{ k'\} \subseteq \overline{C}\M$, and consequently
$\overline{C}\M= \M(3)$, as claimed. 

If $\abs{\M_k}\ge 4$ for some $k\in \FF_3$, then $\M(2) \times \{ k\} 
\subseteq \overline{C}\M$, by Lemma \ref{lem=prelm2}\eqref{prm2} 
and  Lemma \ref{lem=cprop}\eqref{cp2}--\eqref{cp3}.
Otherwise, $\abs{\M_k}=3$, for all $k$. If $\M_k$ is 
independent in $\M(2)$ for some $k$, 
Lemma \ref{lem=prelm2}\eqref{prm1} implies as before that 
$\M(2) \times \{ k\} \subseteq \overline{C}\M$. The case when each 
$\M_k$ is dependent in $\M(2)$ is ruled out by Lemma \ref{lem=ind3}. 
Our proof is now complete.
\end{proof}

Putting together Proposition \ref{prop=indint3}, Lemma \ref{lem=upgr1}, 
Theorem \ref{thm=3mnets}, and Theorem \ref{thm=mnotr}, we obtain 
the following corollary.

\begin{corollary}
\label{cor=t16gral}
No arrangement $\A$ supports a triple of $3$-nets $\NN,\NN',\NN''$ such that 
$[\lambda_{\FF_3}(\NN)]$, $[\lambda_{\FF_3}(\NN')]$,  and $[\lambda_{\FF_3}(\NN'')]$ 
are independent in $Z_{\FF_3}(\A)/B_{\FF_3}(\A)$. 
\end{corollary}

We are finally in a position to prove Theorem \ref{thm:main4}, and 
thus complete the proof of Theorem \ref{thm:main1} in the Introduction.

\begin{proof}[Proof of Theorem \ref{thm:main4}]
By assumption, $L_2(\A)$ has no flats of multiplicity properly 
divisible by $3$. Hence, by Lemma \ref{lema=lambdabij} and 
Lemma \ref{lem:lambda}, the image of the map 
$\lambda_{\FF_3}$ is $Z_{\FF_3}(\A) \setminus B_{\FF_3}(\A)$.  
Therefore, by Corollary \ref{cor=t16gral}, 
we must have $\beta_3(\A)\le 2$. 
\end{proof}

\begin{ack}
The foundational work for this paper was started while the two authors 
visited the Max Planck Institute for Mathematics in Bonn in April--May 2012. 
The work was pursued while the second author visited the Institute of 
Mathematics of the Romanian Academy in June, 2012 and 
June, 2013, and MPIM Bonn in September--October 2013.  
Thanks are due to both institutions for their hospitality, 
support, and excellent research atmosphere. 

A preliminary version of some of the results in this paper was presented 
by the first author in an invited address, titled  {\em Geometry of homology 
jump loci and topology}, and delivered at the Joint International Meeting of 
the American Mathematical Society and the Romanian Mathematical 
Society, held in Alba Iulia, Romania, in June 2013.  He thanks the 
organizers for the opportunity.

The construction of the matroid family from Section \ref{sec:matr}
emerged from conversations with Anca M\u{a}cinic, to whom we are grateful.
We also thank Masahiko Yoshinaga, for a useful discussion related to 
Lemma \ref{lem:lambda}, which inspired us to obtain Theorem \ref{teo=lambdaintro}.
Finally, we thank the referees for their useful suggestions.  
\end{ack}

\newcommand{\arxiv}[1]
{\texttt{\href{http://arxiv.org/abs/#1}{arxiv:#1}}}
\newcommand{\arx}[1]
{\texttt{\href{http://arxiv.org/abs/#1}{arXiv:}}
\texttt{\href{http://arxiv.org/abs/#1}{#1}}}
\newcommand{\doi}[1]
{\texttt{\href{http://dx.doi.org/#1}{doi:#1}}}
\renewcommand{\MR}[1]
{\href{http://www.ams.org/mathscinet-getitem?mr=#1}{MR#1}}

\end{document}